\documentclass{amsart}
\usepackage[margin=1.3in]{geometry}
\usepackage{amsmath}
\usepackage{amsfonts}
\usepackage{amsthm}
\usepackage{comment}
\usepackage{url}
\usepackage{color}
\usepackage{enumitem}
\setlist[enumerate]{font=\normalfont}
\usepackage{mathtools}
\usepackage[alphabetic]{amsrefs}
\usepackage{latexsym}
\usepackage{amssymb}
\usepackage{graphicx}        
\usepackage{float}
\usepackage{tikz-cd}
\usetikzlibrary{matrix,calc,decorations.pathreplacing}
\usepackage[all]{xy}
\usepackage{ mathrsfs }
\usepackage[new]{old-arrows}
\usepackage{nicematrix}
\tikzcdset{scale cd/.style={every label/.append style={scale=#1},
		cells={nodes={scale=#1}}}}

\usepackage[colorinlistoftodos]{todonotes}
\usepackage[colorlinks=true, allcolors=blue]{hyperref}
\usepackage{subcaption}
\usepackage{mathdots}
\usepackage{stmaryrd}

\theoremstyle{plain}
\newtheorem*{theorem*}{Theorem}
\newtheorem{theorem}{Theorem}[section]
\newtheorem{corollary}[theorem]{Corollary}
\newtheorem{lemma}[theorem]{Lemma}

\newtheorem{proposition}[theorem]{Proposition}

\numberwithin{equation}{section}

\theoremstyle{definition}
\newtheorem{definition}[theorem]{Definition}

\newtheorem*{example*}{Example}

\theoremstyle{remark}
\newtheorem{remark}[theorem]{Remark}
\newtheorem*{remark*}{Remark}  

\usepackage{xcolor}

\definecolor{RefBlue}{HTML}{2F6FA3}
\hypersetup{
	colorlinks=true,
	linkcolor=RefBlue,
	filecolor=green,   
	citecolor=purple,
	urlcolor=RefBlue,
}

\author{Juan Sebastian Numpaque-Roa}
\address{\flushleft Centro de Matemática da Universidade do Porto\\ Departamento de Matemática, Faculdade de Ciências da Universidade do Porto \\ Rua do Campo Alegre S/N, 4169-007 Porto, Portugal}
\email{\noindent js.numpaqueroa@gmail.com}

\begin{document}
	\title[Bia\l ynicki-Birula decompositions of Nakajima Quiver Varieties]{Bia\l ynicki-Birula Decompositions of Nakajima Quiver Varieties, Quiver Chains and Star-Shaped Quivers}
	\maketitle
	\begin{abstract}
This paper studies the Białynicki--Birula decomposition associated with the natural $\mathbb C^*$-action on Nakajima quiver varieties given by scaling the maps along the reversed arrows of the doubled quiver. Fixed points of this action are described in terms of representations with relations of auxiliary quivers, called quiver chains. Several aspects of the geometry of the resulting Bia\l ynicki--Birula decomposition are also investigated, including the dimensions of their attracting fibers and the corresponding motivic decomposition of the Nakajima quiver variety in terms of quiver-chain moduli spaces. Finally, the general framework is specialized to star-shaped quivers, where, for particular dimension vectors, the fixed-locus components are classified and the motivic class of the full Nakajima quiver variety is computed in a suitable localization of the Grothendieck ring of varieties.

	\end{abstract}
	\tableofcontents
	\section{Introduction}
	\noindent Nakajima quiver varieties are an important class of semiprojective hyperkähler varieties. They play a central role in geometric representation theory, the theory of moduli spaces, and mathematical physics (see, for instance, the foundational work of Nakajima \cite{Nakajima} and the survey by Ginzburg \cite{GinzburgSurvey}). More precisely, given an acyclic quiver $Q$, a dimension vector $\mathbf d$, and a stability parameter $\theta$, the Nakajima quiver variety $\mathcal M^\theta(\overline Q,\mathbf d)$ can be realized as the GIT quotient of the zero fiber of a certain complex moment map, which is a closed subvariety of the affine space $\text{Rep}(\overline{Q},\mathbf{d})$ parameterising $\mathbf{d}$-dimensional representations of the doubled quiver $\overline{Q}$, by the group $\operatorname{GL}(\mathbf d)$. For suitable choices of the stability parameter and dimension vector, the Nakajima quiver variety will be smooth. 
	\\
	\\
	The semiprojectivity condition means that there exists a $\mathbb{C}^*$-action for which the fixed-point locus is proper and such that every point in the variety has a well-defined limit as $t$ goes to zero under the $\mathbb{C}^*$-action \cite{HauselRodriguezVillegas}. For Nakajima quiver varieties this $\mathbb{C}^*$-action is as follows:  if a point of $\operatorname{Rep}(\overline Q,\mathbf d)$ is written as $(V,X,Y)$, where $X$ denotes the family of maps along the arrows of $Q$ and $Y$ denotes the family of maps along the opposite arrows, then $\mathbb C^*$ acts by
	$$
	t\cdot (V,X,Y)=(V,X,tY).
	$$
	This action preserves the complex moment-map condition so it descends to the Nakajima quiver variety. A key feature of smooth semiprojective varieties is that they admit a Bia\l ynicki-Birula decomposition or stratification which gives a description of several invariants of the variety, such as the Betti numbers and  the motivic class, in terms of that of the fixed locus of the action.  
	\\
	\\
	Another important example of a semiprojective variety is that of the moduli space of stable Higgs bundles of coprime rank and degree over a smooth projective curve. In this case, the study of the fixed point loci for the $\mathbb{C}^*$-action has proven to be successful in understanding the geometry and topology of the moduli space. Important examples include: Hitchin's and Gothen's computation of the Betti numbers of the moduli space of Higgs bundles for ranks 2 and 3 \cite{GothenBettiNumbers}\cite{HitchinSelfDuality}, Hoskins and Pepin-Lehalleur's study of the Voevodsky motive of the moduli space of Higgs bundles \cite{HoskinsVoevodskyMotive} and, more recently, Hausel and Hitchin's work on very stable Higgs bundles \cite{HauselHitchin}.  
	\\
	\\ 
	Inspired by the case of Higgs bundles, in this work we undertake a systematic study of the Bia\l ynicki-Birula decomposition for moduli spaces of Nakajima quiver representations. Our starting point is to identify a fixed point for the $\mathbb{C}^*$-action on the Nakajima quiver variety with a quiver chain (Theorem \ref{characterizationFixedPoints} and Remark \ref{quiverChainsAreQuiverReps}), just as fixed points for the $\mathbb{C}^*$-action on the Higgs bundle case are identified with holomorphic chains \cite{AlvarezCGarciaPDimensionalReduction}. Quiver chains are $\mathbf{d}^\ell$-dimensional representations of an auxiliary quiver, $Q^\ell$, satisfying some relations. The data of the quiver chain depends on the Nakajima quiver representations being fixed by the $\mathbb{C}^*$-action and a positive integer $\ell$. We then see that for a specific choice of stability parameter, $\theta^\ell$, stability of the quiver chain with respect to $\theta^\ell$ is equivalent to $\theta$-stability of the associated fixed point (Proposition \ref{semistableRepresentationIffSemistableChain}). Next, we borrow the techniques of King \cite{KingModuli} to construct the moduli space of $\theta^\ell$-stable quiver chains. Combining these results we show that the fixed-point locus $\mathcal{M}^\theta(\overline{Q},\mathbf{d})^{\mathbb{C}^*}$ corresponds to a disjoint union of the moduli spaces of quiver chains constructed (Theorem \ref{quiverChainsClosedImmersion}). 
	\\
	\\
	As we hinted before, the Bia\l ynicki-Birula theory gives a stratification of the moduli space $\mathcal{M}^{\theta}(\overline{Q},\mathbf{d})$. This stratification is by affine Zariski-locally trivial fibrations whose bases are, in this case, the moduli spaces of quiver chains. Thus the class in the Grothendieck ring of varieties of $\mathcal{M}^{\theta}(\overline{Q},\mathbf{d})$ can be written in terms of those of the moduli spaces of quiver chains corresponding to the fixed loci for the $\mathbb{C}^*$-action (Proposition \ref{motivicBialynickiBirulaDecomposition}). For the star-shaped quiver, which has connections with parabolic Higgs bundles over $\mathbb{P}^1$ \cite{FisherRayan} and moduli spaces of Hyperpolygons \cite{GodinhoMandini}, we compute the motivic class in a suitable localization of the Grothendieck ring of varieties of the full Nakajima quiver variety (Theorem \ref{motivicClassOfNakajimaQuiverVariety}) and give a classification of the fixed-point loci for particular dimension vectors (Theorem \ref{classificationType1...1QuiverChains}). The main ingredient for the former computation (Lemma \ref{MotivicFeiIdentity}) is a motivic analogue of a result of Fei \cite{FeiQuiversWithRelations} which is itself closely related to Reineke’s work on Harder-Narasimhan filtrations and point counting for quiver varieties \cite{ReinekePointCounting}. Although motivic classes of Nakajima quiver representations have already been studied by Wyss \cite{WyssMotive}, using a motivic analogue of Hausel's arithmetic Fourier transform \cite{HauselArithmeticFourierTransform}, our approach is different as it uses the Bia\l ynicki-Birula decomposition of the moduli space. This requires an explicit analysis of the motivic class of each quiver-chain moduli space associated to the fixed locus (Theorem \ref{motivicClassQuiverModuliWithRelations}), thereby providing additional information about the geometry of the Nakajima quiver variety.
	\\
	\\
	The paper is organized as follows. In Section \ref{theSetup} we recap the construction of Nakajima quiver varieties from doubled quivers, the hyperkähler moment map and reduction, the GIT interpretation of this, and the deformation theory of the moduli spaces obtained. In Section \ref{torusActionSection} we study the $\mathbb C^*$-action on Nakajima quiver varieties, characterize its fixed points and introduce the associated quiver chains and their moduli spaces. We also show the existence of a universal family for Nakajima quiver varieties and show the existence of liftings of the $\mathbb{C}^*$-action to it. Next, in Section \ref{BialynickiBirulaSection}, we review the generalities of the Bia\l ynicki--Birula decomposition in this quiver setting and study its geometry using tools from deformation theory. Finally, in Section \ref{theCaseOfTheStarShapedQuiver} we specialize all of our discussion to star-shaped quivers. We include an Appendix (Section \ref{theAppendix}) where we discuss moduli spaces of quiver representations with relations, flag varieties and generalities on the Grothendieck ring of varieties. We also prove useful lemmas on vector bundles and Zariski-locally trivial fibrations and explain why Fei's formula holds in the Grothendieck ring of varieties using the formalism of motivic Hall algebras. 
	\newline
	\newline
	\textbf{Funding.} This work was supported by Fundação para a Ciência e a Tecnologia, I.P. [grant \hyperref{https://doi.org/10.54499/UI/BD/154369/2023}{}{}{UI/BD/154369/2023}]; and partially supported by the project ``Higgs bundles: geometry, algebra and physics'' [grant \hyperref{https://doi.org/10.54499/2024.15931.PEX}{}{}{2024.15931.PEX}] and Centro de Matemática da Universidade do Porto [grant \hyperref{https://doi.org/10.54499/UID/00144/2025}{}{}{UID/00144/2025}].
	\newline
	\newline
	\textbf{Acknowledgements.} I am deeply grateful to my advisor, Peter B. Gothen, whose guidance, insight, and support were essential to the development of this work. Many stimulating conversations with Carlos Florentino helped shape my thinking and encouraged the project along the way. I am also indebted to David Alfaya for listening generously to my ideas and for suggesting the use of towers of fibrations to compute motivic classes. Discussions with Verónica Calvo-Cortés provided valuable insight into the toric-geometric aspects of the paper.
	\section{The setup}\label{theSetup}
	\subsection{Representations of the doubled quiver}\label{repsDoubleQuiverSection}
	Let $Q=(Q_0,Q_1,h,t)$ be a finite acyclic quiver. Here $Q_0$ and $Q_1$ are finite sets of vertices and edges respectively and $h,t:Q_1\to Q_0$ maps that assign to an edge its corresponding head and tail vertices. For $\mathbf{d}=(d_i)_{i\in Q_0}\in \mathbb{N}^{|Q_0|}$ a dimension vector, a $\mathbf{d}$-dimensional \emph{representation} of $Q$ is a tuple $\mathbf{V}=(V=(V_i)_{i\in Q_0},X=(X_{\alpha})_{\alpha\in Q_1})$ where $V_i=\mathbb{C}^{d_i}$ and $X_{\alpha}:V_{t\alpha}\to V_{h\alpha}$ is a linear map. Given two representations $\mathbf{V}=(V,X)$ and $\mathbf{V}'=(V',X')$ of $Q$ not necessarily of the same dimension, a \emph{morphism} $f:\mathbf{V}\to \mathbf{V}'$ is given by a family of linear maps $f=(f_i:V_i\to V'_i)_{i\in Q_0}$ for which the diagram  
	\begin{equation}\label{RepMorphismDiagram}
		\begin{tikzcd}[column sep =large, row sep =large]\
			V_{t\alpha} \arrow[r,"X_\alpha"] \arrow[d,"f_{t\alpha}",'] & V_{h\alpha}\arrow[d,"f_{h\alpha}"] \\
			V'_{t\alpha}\arrow[r,',"X'_\alpha"] & V'_{h\alpha}
		\end{tikzcd}
	\end{equation}
	commutes for all $\alpha\in Q_1$. 
	\\
	\\
	All $\mathbf{d}$-dimensional representations of $Q$ are parameterised  by the affine space
	$$
	\text{Rep}(Q,\mathbf{d}):=\bigoplus_{\alpha\in Q_1}\text{Hom}(\mathbb{C}^{d_{t\alpha}},\mathbb{C}^{d_{h\alpha}}).
	$$
	The reductive group $\text{GL}(\mathbf{d})=\prod_{i\in Q_0}\text{GL}(d_i)$ acts algebraically on the parameter space $\text{Rep}(Q,\mathbf{d})$ by conjugation:
	\begin{equation}\label{actionGLOnQuiverReps}
		g\cdot \mathbf{V}=(V,(g_{h\alpha}X_\alpha g_{t\alpha}^{-1})_{\alpha\in Q_1})
	\end{equation}
	for all $g:=(g_i)_{i\in Q_0}\in \text{GL}(\mathbf{d})$ and $\mathbf{V}=(V,X)\in \text{Rep}(Q,\mathbf{d})$. Notice that the orbits for this action correspond to isomorphism classes of $\mathbf{d}$-dimensional quiver representations.
	\\
	\\
	Let $\overline{Q}$ be the doubled quiver of $Q$. This is the quiver such that $\overline{Q}_0=Q_0$ and $\overline{Q}_1=Q_1\sqcup \{\overline{\alpha}\}_{\alpha\in Q_1}$ where $t\overline{\alpha}=h\alpha$ and $h\overline{\alpha}=t\alpha$. The total space of the cotangent bundle, $T^*\text{Rep}(Q,\mathbf{d})\to \operatorname{Rep}(Q,\mathbf{d})$, can be canonically identified with
	\begin{equation}\label{cotangentBundleFramedQuiverReps}
		\text{Rep}(\overline{Q},\mathbf{d})=\bigoplus_{\alpha\in\overline{Q} _1}\text{Hom}(\mathbb{C}^{d_{t\alpha}},\mathbb{C}^{d_{h\alpha}})
	\end{equation}
	via the trace pairing 
	\begin{equation}\label{tracePairing}
		\begin{matrix}
			\text{Hom}(W,V)&\longrightarrow&\text{Hom}(V,W)^{\vee}\\ \varphi&\longmapsto &\text{Tr}(\varphi \ \bullet)
		\end{matrix}.
	\end{equation}
	Note that there is an essentially surjective and forgetful functor, \texttt{Forget}, from the category of $\mathbf{d}$-dimensional representations of $\overline{Q}$ onto the category of $\mathbf{d}$-dimensional representations of $Q$. Indeed, for $\mathbf{V}=(V,(X_\alpha)_{\alpha\in Q_1},(Y_{\overline{\alpha}})_{\alpha\in Q_1})\in \text{Rep}(\overline{Q},\mathbf{d})$, $\texttt{Forget}(\mathbf{V})=(V,X)\in \text{Rep}(Q,\textbf{d})$, and for $f$ a morphism of $\mathbf{d}$-dimensional representations of the doubled quiver, \texttt{Forget}$(f)=f$.
	\begin{remark}
	Notice that at the level of objects, the functor $\texttt{Forget}$ is no more than the projection morphism from the cotangent bundle onto $\operatorname{Rep}(Q,\mathbf{d})$. 
	\end{remark}
	For all $i\in Q_0$ we fix a hermitian inner product on the vector spaces $\mathbb{C}^{d_i}$. This induces a hermitian inner product on each direct summand of Equation (\ref{cotangentBundleFramedQuiverReps}) via the trace pairing $(A,B)\mapsto\text{Tr}(AB^*)$, $(\_)^*$ being the hermitian adjoint, and hence one on $\text{Rep}(\overline{Q},\mathbf{d})$. The group $\text{U}(\mathbf{d})=\prod_{i\in Q_0}\text{U}(d_i)$ acts on $\text{Rep}(\overline{Q},\mathbf{d})$ by conjugation as in Equation (\ref{actionGLOnQuiverReps}). This action preserves the hermitian inner product just described.
	\begin{remark}\label{diagonalSubgroupsActingTrivially}
	Notice that for both group actions described, the diagonal groups $\mathbb{C}^*\hookrightarrow\operatorname{GL}(\mathbf{d})$ and $\operatorname{U}(1)\hookrightarrow\operatorname{U}(\mathbf{d})$ act trivially on $\operatorname{Rep}(\overline{Q},\mathbf{d})$ so they factor through the groups $\operatorname{PGL}(\mathbf{d})$ and $\operatorname{PU}(\mathbf{d})$ respectively. 
	\end{remark}
	The space $\text{Rep}(\overline{Q},\mathbf{d})$ has three complex structures 
	$$
	\mathbf{I}(X,Y)=\sqrt{-1}(X,Y),\ \mathbf{J}(X,Y)=(-Y_{\overline{\alpha}}^*,X_{\alpha}^*)_{\alpha\in Q_1},\ \mathbf{K}(X,Y)=\sqrt{-1}(-Y_{\overline{\alpha}}^*,X_{\alpha}^*)_{\alpha\in Q_1}
	$$
	which satisfy the quaternionic relations. Here we have denoted a representation simply as a pair $(X,Y)$. The Riemannian metric, given by the real part of the hermitian inner product,
	$$
	g((X,Y),(X',Y'))=\text{Re}(\operatorname{Tr}(\sum_{\alpha\in Q_1}X_\alpha (X'_\alpha)^*+(Y'_{\overline{\alpha}})^*Y_{\overline{\alpha}})),
	$$
	together with the three complex structures just defined endows $\text{Rep}(\overline{Q},\mathbf{d})$ with the structure of Hyperkähler manifold. The $\text{U}(\mathbf{d})$-action has an associated Hyperkähler moment map with real and complex components given by 
	\begin{equation}\label{realMomentMapDef}
		\mu_{\mathbb R}(\mathbf{V})
		=
		\left(
		\sum_{\substack{h\alpha=i\\ \alpha\in Q_1}} X_\alpha X_\alpha^*
		-
		\sum_{\substack{t\alpha=i\\ \alpha\in Q_1}} X_\alpha^*X_\alpha
		+
		\sum_{\substack{t\alpha=i\\ \alpha\in Q_1}} Y_{\overline{\alpha}} Y_{\overline{\alpha}}^*
		-
		\sum_{\substack{h\alpha=i\\ \alpha\in Q_1}} Y_{\overline{\alpha}}^*Y_{\overline{\alpha}}
		\right)_{i\in Q_0}\in \text{Herm}(\mathbf{d})
	\end{equation}
	\begin{equation}\label{complexMomentMapDef}
		\mu_{\mathbb{C}}(\mathbf{V})=\bigg(\sum_{\substack{h\alpha=i\\\alpha\in Q_1}}X_\alpha Y_{\overline{\alpha}}-\sum_{\substack{t\alpha=i\\\alpha\in Q_1}}Y_{\overline{\alpha}}X_\alpha\bigg)_{i\in Q_0}\in \mathfrak{gl}(\mathbf{d}).
	\end{equation}
	Here $\operatorname{Herm}(\mathbf{d})=\bigoplus_{i\in Q_0}\operatorname{Herm}(d_i)$ and $\mathfrak{gl}(\mathbf{d})$ is the Lie algebra of $\text{GL}(\mathbf{d})$. 
	\subsection{Hyperkähler reduction}\label{hyperKReduction}
	Good references for this section are Ballandras \cite{BallandrasTrivializations}, Nakajima \cite{Nakajima} and Hitchin et al. \cite{HitchinHyperK}. Let 
	$$
	\mathfrak{z}_{\mathbb{R}}(\mathbf{d})=\{(\theta_{i}\operatorname{Id}_{V_i})_{i\in Q_0}|\theta_i\in \mathbb{R}\} \text{ and } \mathfrak{z}_{\mathbb{C}}(\mathbf{d})=\{(\lambda_i \text{Id}_{V_i})_{i\in Q_0}|\lambda_i\in\mathbb{C}\}.
	$$
	For $(\theta,\lambda)\in \mathfrak{z}_{\mathbb{R}}(\mathbf{d})\times\mathfrak{z}_{\mathbb{C}}(\mathbf{d})$
	we define
	\begin{equation}\label{HyperKQuiverQuotient}
		\mathcal{M}^{\theta,\lambda}(\overline{Q},\mathbf{d}):=\mu_{\mathbb{R}}^{-1}(\theta)\cap\mu^{-1}_{\mathbb{C}}(\lambda)/\operatorname{PU}(\mathbf{d}).
	\end{equation}
	This quotient may have singularities, however, for an appropriate choice of stability parameters $\theta$ and $\lambda$, the corresponding quotient is a non-singular Hyperkähler manifold. Indeed, one can show that
	$$
	H_\mathbf{d}^{\text{reg}}=H_{\mathbf{d}}\setminus\bigcup_{0<\mathbf{e}<\mathbf{d}}H_{\mathbf{e}}
	$$
	is the set of parameters for which the quotient in Equation (\ref{HyperKQuiverQuotient}) is a non-singular Hyperkähler variety. Here $\mathbf{e}<\mathbf{d}$ means that $e_i\leq d_i$ for all $i\in Q_0$ and $\mathbf{e}\neq \mathbf{d}$, and 
	$$
	H_{\mathbf{d}}=\{(\theta,\lambda)\in \mathfrak{z}_{\mathbb{R}}(\mathbf{d})\times\mathfrak{z}_{\mathbb{C}}(\mathbf{d})\ | \ \theta\cdot\mathbf{d}=\lambda\cdot \mathbf{d}=0\}.
	$$
	We say that $(\theta,\lambda)\in H_\mathbf{d}^{\text{reg}}$ is a \emph{generic} parameter. 
	\\
	\\
	From now on we will assume $\theta$ to be generic and $\lambda=0$. Notice also that the genericity of $\theta$ implies that the dimension vector $\mathbf{d}$ is indivisible, that is, $\operatorname{gcd}\{d_i\}_{i\in Q_0}=1$.
	\subsection{Description of the points and comparison with GIT quotient}\label{GITDescriptionNakajimaQuiverVarieties}
	The closed affine subvariety $\mu^{-1}_{\mathbb{C}}(0)\subseteq\operatorname{Rep}(\overline{Q},\mathbf{d})$ is invariant under the action of the complexification, $\text{GL}(\mathbf{d})$, of the unitary group $\text{U}(\mathbf{d})$. Also notice $\mu^{-1}_{\mathbb{C}}(0)=\operatorname{Rep}(\overline{Q},\mathbf{d},\mathcal{R})$ (see Section \ref{representationsWIthRelations}), for 
	$$
	\mathcal{R}=\{\sum_{\substack{h\alpha=i\\ \alpha\in Q_1}}\alpha \overline{\alpha}-\sum_{\substack{t\alpha=i\\ \alpha\in Q_1}}\overline{\alpha}\alpha \mid i\in Q_0 \}.
	$$
	We will refer to the representations parameterised by the zero fiber of the complex moment map as \emph{Nakajima quiver representations}. By results of King \cite{KingModuli}, Nakajima \cite{Nakajima}, Harada and Wilkin \cite{HaradaWilkin} and Hoskins \cite{HoskinsStratifications}, for $\theta\in \mathbb{Z}^{|Q_0|}$ generic, there is a homeomorphism between the HyperKähler reduction in Equation (\ref{HyperKQuiverQuotient}) and the GIT quotient of $\theta$-stable points by the action of $\text{GL}(\mathbf{d})$:
	$$
	\mathcal{M}^{\theta}(\overline{Q},\mathbf{d}):=\mathcal{M}^{\theta,0}(\overline{Q},\mathbf{d})\simeq \mu^{-1}_\mathbb{C}(0)^{\theta\text{-}s}\sslash\text{GL}(\mathbf{d}).
	$$
	We will refer to this quotient so as its analytic incarnation in Equation (\ref{HyperKQuiverQuotient}) as \emph{Nakajima quiver varieties}.  $\mathcal{M}^\theta(\overline{Q},\mathbf{d})$ is a geometric quotient, meaning that it is parameterising isomorphism classes of $\theta$-stable representations. 
	Let us recall that a representation $\mathbf{V}\in \text{Rep}(\overline{Q},\mathbf{d})$ is $\theta$-(semi)stable if for all proper non-zero subrepresentation $\mathbf{W}\subseteq \mathbf{V}$ we have 
	\begin{equation}\label{stabilityConditionQuivers}
		\mu_{\theta}(\mathbf{W})(\leq)<\mu_{\theta}(\mathbf{V})
	\end{equation}
	where 
	$$
	\mu_{\theta}(\mathbf{W})=\frac{\theta\cdot\mathbf{d}_{\mathbf{W}}}{|\mathbf{d}_{\mathbf{W}}|}.
	$$
	Here $\mathbf{d}_{\mathbf{W}}$ is the dimension vector of $\mathbf{W}$ and $|\mathbf{d}_\mathbf{W}|=\sum_{i\in Q_0}\dim(W_i)$. Notice, since we assumed $\theta$ to be generic, that the right-hand side in the inequality of Equation (\ref{stabilityConditionQuivers}) is zero and that $\theta$-semistability is equivalent to $\theta$-stability. Important features of $\theta$-stable representations are that they are indecomposable, they have only scalar automorphisms and that there are no non-zero morphisms between non-isomorphic stable representations of the same slope. From the GIT formalism, we also note that there is a projective morphism
	$$
	\pi:\mathcal{M}^{\theta}(\overline{Q},\mathbf{d}) \to \mu^{-1}_{\mathbb{C}}(0)\sslash\text{GL}(\mathbf{d})
	$$ 
	which sends every $\theta$-stable representation to its corresponding semisimplification (see, for instance, \cite[Theorem 10.12]{libroKirillov}). The subvariety $\mathcal{N}:=\pi^{-1}(0)$ is known as the \emph{nilpotent cone}. 
	
	\subsection{Deformation theory}\label{deformationTheorySection}
	Let $\mathbf{V}=(V,X)$ and $\mathbf{W}=(W,X')$ be representations of the quiver $Q$ of dimensions $\mathbf{d}_\mathbf{V}$ and $\mathbf{d}_\mathbf{W}$ respectively. Consider the two-term complex of vector spaces 
	\begin{equation}\label{deformationComplexQuiverVarieties}
		\mathcal{H}om^\bullet(\mathbf{V},\mathbf{W}): \bigoplus_{i\in Q_0}\text{Hom}(V_i,W_i)\xlongrightarrow{\delta}\bigoplus_{\alpha\in Q_1}\text{Hom}(V_{t\alpha},W_{h\alpha})
	\end{equation}
	where 
	$$
	\delta(f_i)_{i\in Q_0}=(f_{h\alpha}X_{\alpha}-X'_\alpha f_{t\alpha})_{\alpha\in Q_1}.
	$$
	Also, one can see that 
	$$
	H^0(\mathcal{H}om^\bullet(\mathbf{V},\mathbf{W}))=\ker\delta=\text{Hom}(\mathbf{V},\mathbf{W}), \
	H^1(\mathcal{H}om^\bullet(\mathbf{V},\mathbf{W}))=\text{coker }\delta=\text{Ext}^1(\mathbf{V},\mathbf{W})
	$$
	yielding the exact sequence
	\begin{equation}\label{fundamentalExactSequenceQuivers}
		0\to \text{Hom}(\mathbf{V},\mathbf{W})\to  \bigoplus_{i\in Q_0}\text{Hom}(V_i,W_i) \to \bigoplus_{\alpha\in Q_1}\text{Hom}(V_{t\alpha},W_{h\alpha})\to \text{Ext}^1(\mathbf{V},\mathbf{W})\to 0.
	\end{equation}
	Moreover, 
	\begin{align*}
		\chi(\mathcal{H}om^\bullet(\mathbf{V},\mathbf{W}))&=\dim(\text{Hom}(\mathbf{V},\mathbf{W}))-\dim(\text{Ext}^1(\mathbf{V},\mathbf{W})) \\ &= \langle \mathbf{d}_\mathbf{V},\mathbf{d}_\mathbf{W}\rangle_Q.
	\end{align*}
	We recall that the pairing $\langle\bullet\ ,\bullet \rangle_Q:\mathbb{Z}^{Q_0}\times\mathbb{Z}^{Q_0}\to \mathbb{Z} $ given by
	$$
	\langle\mathbf{d},\mathbf{d}'\rangle_Q=\sum_{i\in Q_0}d_id_i'-\sum_{\alpha\in Q_1}d_{t\alpha}d'_{h\alpha}
	$$
	is the so-called \emph{Euler form} of $Q$.
	\\
	\\
	The exact sequence in Equation (\ref{fundamentalExactSequenceQuivers}) is fundamental in the study of the representation theory of $Q$ as well as in that of the geometry of the associated moduli spaces. For instance, $\mathcal{H}om^\bullet (\mathbf{V},\mathbf{V})$ is the deformation complex at the point $[\mathbf{V}]$, representing the isomorphism class of $\mathbf{V}$, of the moduli space $\text{Rep}(Q,\mathbf{d}_\mathbf{V})^{\theta\text{-}s}/\text{GL}(\mathbf{d})$ parameterising $\theta$-stable representations up to isomorphism. Also, $H^1(\mathcal{H}om^\bullet(\mathbf{V},\mathbf{V}))=\text{Ext}^1(\mathbf{V},\mathbf{V})$ is the tangent space of this moduli space at the point $[\mathbf{V}]$. See \cite[Section 2.2]{libroKirillov} for a detailed exposition of what we have just discussed. 
	\\
	\\
	In the case we are interested in, which is that of the moduli space $\mathcal{M}^{\theta}(\overline{Q},\mathbf{d})$, the deformation complex has one extra term which takes into account the complex moment map:
	\begin{equation}\label{deformationComplexNakajimaQuiverVarieties}
		\overline{\mathcal{H}om}^\bullet(\mathbf{V}):\mathfrak{gl}(\mathbf{d})\xlongrightarrow{\delta_0}\bigoplus_{\alpha\in Q_1}\text{Hom}(V_{t\alpha},V_{h\alpha})\oplus\text{Hom}(V_{t\overline{\alpha}},V_{h\overline{\alpha}})\xlongrightarrow{\delta_1}\mathfrak{gl}(\mathbf{d})
	\end{equation}
	where $\mathbf{V}=((V_i)_{i\in Q_0},(X_\alpha)_{\alpha\in Q_1},(Y_{\overline{\alpha}})_{\alpha\in Q_1})\in \mu^{-1}_\mathbb{C}(0)\subseteq\text{Rep}(\overline{Q},\mathbf{d})$,
	$$
	\delta_0(f_i)_{i\in Q_0}=(f_{h\alpha}X_{\alpha}-X_\alpha f_{t\alpha},f_{t\alpha}Y_{\overline{\alpha}}-Y_{\overline{\alpha}}f_{h\alpha})_{\alpha\in Q_1}
	$$
	and 
	$$
	\delta_1(\psi_\alpha,\psi_{\overline{\alpha}})_{\alpha\in Q_1}=\bigg(\sum_{\substack{h\alpha=i\\\alpha\in Q_1}}(X_\alpha\psi_{\overline{\alpha}}+\psi_\alpha Y_{\overline{\alpha}})-\sum_{\substack{t\alpha=i\\\alpha\in Q_1}}(Y_{\overline{\alpha}}\psi_\alpha+\psi_{\overline{\alpha}}X_\alpha)\bigg)_{i\in Q_0}.
	$$
	Notice that $\delta_0$ is precisely the infinitesimal action or derivative of the action in Equation (\ref{actionGLOnQuiverReps}) and that $\delta_1$ is the derivative of the complex moment map given in Equation (\ref{complexMomentMapDef}). It is also interesting to note that $\overline{\mathcal{H}om}^\bullet(\mathbf{V})=\text{Tot}(C^{\bullet\bullet}(\mathbf{V}))$ where $C^{\bullet\bullet}(\mathbf{V})$ is the double complex given by 
	\begin{center}
		\begin{tikzcd}[column sep =large, row sep =large]\
			\bigoplus\limits_{\alpha\in Q_1}\text{Hom}(V_{t\overline{\alpha}},V_{h\overline{\alpha}}) \arrow[r,"\delta_1^h"]  & \mathfrak{gl}(\mathbf{d})\\
			\mathfrak{gl}(\mathbf{d})\arrow[r,',"\delta_0^h"] \arrow[u,"\delta_0^v"]& \bigoplus\limits_{\alpha\in Q_1}\text{Hom}(V_{t\alpha},V_{h\alpha})\arrow[u,"\delta_1^v",'] 
		\end{tikzcd}
	\end{center}
	where 
	$$
	\delta^h_0(f_i)_{i\in Q_0}=(f_{h\alpha}X_{\alpha}-X_\alpha f_{t\alpha})_{\alpha\in Q_1}, \ \delta_0^v(f_i)_{i\in Q_0}=(f_{t\alpha}Y_{\overline{\alpha}}-Y_{\overline{\alpha}}f_{h\alpha})_{\alpha\in Q_1},
	$$
	$$
	\delta_1^h(\psi_{\overline{\alpha}})_{\alpha\in Q_1}=\bigg(\sum_{\substack{h\alpha=i\\\alpha\in Q_1}}X_\alpha\psi_{\overline{\alpha}}-\sum_{\substack{t\alpha=i\\\alpha\in Q_1}}\psi_{\overline{\alpha}}X_\alpha\bigg)_{i\in Q_0}\text{ and }\delta_1^v(\psi_\alpha)=-\bigg(\sum_{\substack{h\alpha=i\\\alpha\in Q_1}}\psi_\alpha Y_{\overline{\alpha}}-\sum_{\substack{t\alpha=i\\\alpha\in Q_1}}Y_{\overline{\alpha}}\psi_\alpha\bigg)_{i\in Q_0}.
	$$
	Using the trace pairing in Equation (\ref{tracePairing}) one can see that $\delta_1^h=(\delta_0^h)^{\vee}$ and $\delta_1^v=(\delta_0^v)^\vee$. Therefore $C^{\bullet 1}$ is the dual complex to $C^{\bullet 0}$ which is actually that in Equation (\ref{deformationComplexQuiverVarieties}).
	\begin{proposition}Let $\mathbf{V}\in\mu^{-1}_{\mathbb{C}}(0)$ be a $\theta$-stable representation, then \label{eulerCharacteristicDeformationComplexNakajimaQuiverVarieties}
		$\chi(\overline{\mathcal{H}om}^\bullet(\mathbf{V}))=2\langle\mathbf{d},\mathbf{d}\rangle_Q$.
	\end{proposition}
	\begin{proof}
		Recall the Euler characteristic of the complex is the alternating sum of the dimensions of the cohomology groups of the complex. In this case, this reads:
		\begin{align*}
			\chi(\overline{\mathcal{H}om}^\bullet(\mathbf{V}))&=\dim(\ker\delta_0)-\dim(\ker\delta_1)+\dim(\text{Im}\ \delta_0)+\dim(\text{coker}\delta_1)\\&= 2\dim(\mathfrak{gl}(\mathbf{d}))-\dim(\text{Rep}(\overline{Q},\mathbf{d}))
		\end{align*}
		so the conclusion follows. 
	\end{proof}
	What can we say about the cohomology groups of the complex $\overline{\mathcal{H}om}^\bullet(\mathbf{V})$ and what do they tell us about the geometry of the moduli space $\mathcal{M}^{\theta}(\overline{Q},\mathbf{d})$ in a neighborhood of $[\mathbf{V}]$? 
	\begin{proposition}\label{cohomologiesDeformationComplexNakajimaQuiverVarieties}
	Let $[\mathbf{V}]\in \mathcal{M}^{\theta}(\overline{Q},\mathbf{d})$, then $H^0(\overline{\mathcal{H}om}^\bullet(\mathbf{V}))\cong H^2(\overline{\mathcal{H}om}^\bullet(\mathbf{V}))^\vee\cong\mathbb{C}$ and $T_{[\mathbf{V}]}\mathcal{M}^{\theta}(\overline{Q},\mathbf{d})\cong H^1(\overline{\mathcal{H}om}^\bullet(\mathbf{V}))\cong H^1(\overline{\mathcal{H}om}^\bullet(\mathbf{V}))^\vee$.
	\end{proposition}
	\begin{proof}
		$H^0(\overline{\mathcal{H}om}^\bullet(\mathbf{V}))=\ker \delta_0$ consists of morphisms commuting with all the $X_\alpha$ and $Y_{\overline{\alpha}}$, $\alpha\in Q_1$. In other words, $H^0(\overline{\mathcal{H}om}^\bullet(\mathbf{V}))=\text{End}(\mathbf{V})$ and, by stability, the only endomorphisms are scalars. The claimed dualities come from the fact that $\delta_0^\vee=\delta_1$ meaning that the complex $\overline{\mathcal{H}om}^\bullet(\mathbf{V})$ is self-dual. Finally, for a $\theta$-stable point $\mathbf V$, the Zariski tangent space to the complex moment-map fiber is $\ker \delta_1$, and the tangent space to the $\operatorname{GL}(\mathbf d)$-orbit is $\operatorname{Im}\delta_0$. Therefore the tangent space to the geometric quotient at $[\mathbf V]$ is naturally identified with the middle cohomology
		\[
		T_{[\mathbf V]}\mathcal M^\theta(\overline Q,\mathbf d)
		\cong
		\frac{\ker\delta_1}{\operatorname{Im}\delta_0}
		=
		H^1(\overline{\mathcal H om}^\bullet(\mathbf V)).
		\]
	\end{proof}
	\begin{corollary}
	The dimension of the moduli space $\mathcal{M}^{\theta}(\overline{Q},\mathbf{d})$ is $\dim(H^1(\overline{\mathcal{H}om}^\bullet(\mathbf{V})))=2-2\langle\mathbf{d},\mathbf{d}\rangle_Q$.
	\end{corollary}
	\section{$\mathbb{C}^*$-action on Nakajima quiver varieties}\label{torusActionSection}
	\subsection{The $\mathbb{S}^1$ and $\mathbb{C}^*$ actions on the moduli space} The quotient $\mathcal{M}^{\theta}(\overline{Q},\mathbf{d})$ has a natural $\mathbb{S}^1$-action given by
	$$
	(t,(V,X,Y))\mapsto (V,X,tY).
	$$  
	
	\begin{proposition}  \cite[Theorem 5.1]{Nakajima} The circle action satisfies the following properties: 
		\begin{enumerate}
			\item It preserves the complex structure $\mathbf{I}$ and the Hyperkähler metric.
			\item The associated moment map 
			$$
			F(\mathbf{V})=\sum_{\alpha\in Q_1}\|Y_{\overline{\alpha}}\|^2
			$$
			is proper. 
			\item It extends to a holomorphic (with respect to the complex structure $\mathbf{I}$) $\mathbb{C}^*$-action.
		\end{enumerate}
	\end{proposition}
	Now we study the critical set for the $\mathbb{S}^1$-action or, equivalently, the set of fixed points for the circle action which coincide with the fixed points for the associated $\mathbb{C}^*$-action.
	\\
	\\
	Let $[\mathbf{V}]\in \mathcal{M}^{\theta}(\overline{Q},\mathbf{d})$ be a fixed point for the $\mathbb{C}^*$-action. Then, there exists a one-parameter subgroup $g:\mathbb{C}^*\to \text{GL}(\mathbf{d})$ such that 
	\begin{equation}\label{oneParameterSubgroupTorusaction}
		g(t)\cdot \mathbf{V} = t\cdot\mathbf{V}=(V,X,(tY_{\overline{\alpha}})_{\alpha\in Q_1})
	\end{equation}
	for all $t\in \mathbb{C}^*$. Thus, for all $i\in Q_0$, $V_i$ becomes a finite-dimensional representation of $\mathbb{C}^*$ so it decomposes into weight spaces
	$$V_i=\bigoplus_{\lambda\in \mathbb{Z}}V_{i,\lambda}$$
	where  $V_{i,\lambda}=\{v\in V_i\ |\ g_i(t)v=t^\lambda v\text{ for all }t\in\mathbb C^*\}$ \cite[Theorem 12.12]{MilneAlgebraicGroups}.
	\\
	\\
	For $\alpha\in Q_1$, Equation (\ref{oneParameterSubgroupTorusaction}) gives that
	$
	g_{h\alpha}(t)X_\alpha g_{t\alpha}(t)^{-1}
	=
	X_\alpha
	$
	which implies that
	$X_\alpha(V_{t\alpha,\lambda})
	\subseteq
	V_{h\alpha,\lambda}$. Similarly, for the reversed arrow $\overline\alpha$, we have
	$
	g_{t\alpha}(t)Y_{\overline\alpha}g_{h\alpha}(t)^{-1}
	=
	tY_{\overline\alpha}
	$
	which gives
	$
	Y_{\overline\alpha}(V_{h\alpha,\lambda})
	\subseteq
	V_{t\alpha,\lambda+1}.
	$
	Thus the representation decomposes into weight spaces in such a way that the
	maps along arrows of $Q$ preserve the weight, whereas the maps along reversed
	arrows increase the weight by one as seen in Figure \ref{fixedPoints}.
	\begin{figure}[H]
		\begin{center}
			\begin{tikzcd}[scale cd=0.7]
				&\cdots V_{t\alpha_1,\lambda+1}\arrow[dr,"X_{\alpha_1,\lambda+1}"]\arrow[from=5-3,magenta,"Y_{\overline{\alpha}_1,\lambda}",',bend left=12]&&&&& V_{t\beta_1,\lambda+1}\cdots \\
				&\iddots&V_{t\alpha,\lambda+1}\arrow[to=2-5,"X_{\alpha,\lambda+1}"]&&V_{h\alpha,\lambda+1}\arrow[from=1-7,"X_{\beta_1,\lambda+1}",']&\iddots&\\
				\cdots V_{h\alpha_k,\lambda+1}\arrow[from=2-3,"X_{\alpha_k,\lambda+1}",',crossing over]&&&&&V_{t\beta_j,\lambda+1}\arrow[from=ul,"X_{\beta_j,\lambda +1 }"]\cdots &\\
				
				&\cdots V_{t\alpha_1,\lambda}\arrow[dr,"X_{\alpha_1,\lambda}",']&&&&& V_{t\beta_1,\lambda}\cdots \\
				&\iddots&V_{t\alpha,\lambda}\arrow[to=6-1,"X_{\alpha_k,\lambda}"]\arrow[to=5-5,"X_{\alpha,\lambda}"]&&V_{h\alpha,\lambda}\arrow[orange,to=2-3,"Y_{\overline{\alpha},\lambda}",crossing over]\arrow[from=4-7,"X_{\beta_1,\lambda}"]\arrow[to=1-7,magenta,"Y_{\overline{\beta}_1,\lambda}",bend right=20,']&  \iddots&\\
				\cdots V_{h\alpha_k,\lambda}\arrow[to=2-3, bend left=20 ,blue,"Y_{\overline{\alpha}_k,\lambda}",crossing over ]&&&&& V_{t\beta_j,\lambda} \arrow[from=ul,"X_{\beta_j,\lambda}"]\cdots\arrow[to=2-5,blue,"Y_{\overline{\beta}_j,\lambda}",bend right=12,',crossing over]& 
			\end{tikzcd}
		\end{center}
		\vspace{-2mm}
		\caption{\label{fixedPoints}Fixed point for the $\mathbb{S}^1(\mathbb{C}^*)$-action}
	\end{figure}
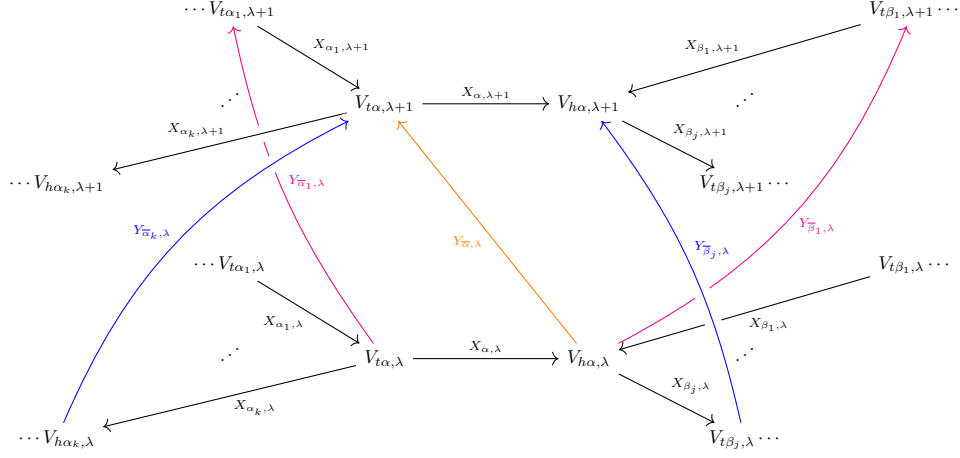
	Conversely, a quiver representation that admits such a decomposition is a fixed point for the $\mathbb{C}^*$-action. Indeed, let $\mathbf{V}$ be such a representation, then for all $t\in\mathbb{C}^*$, $t\cdot\mathbf{V}\cong \mathbf{V}$ via the automorphism of quiver representations given by
	\begin{equation}\label{automorphismRepresentationAdmitingDecompositionIsFixedPoint}
		t^n\text{Id}:V_{i,n}\to V_{i,n}
	\end{equation}
	We summarize the previous discussion in the following theorem:
	\begin{theorem}\label{characterizationFixedPoints}
		A representation $\mathbf{V}\in \mu^{-1}_{\mathbb{C}}(0)^{\theta\text{-}s}$ is a fixed point for the $\mathbb{C}^*$-action if and only if there exists $\ell\in \mathbb{N}_{>0}$ such that: 
		\begin{itemize}
			\item For each $i\in Q_0$, there is a splitting $V_i=\bigoplus_{n=1}^\ell V_{i,n}$.
			\item For each $\alpha\in Q_1$, $X_\alpha(V_{t\alpha,n})\subseteq V_{h\alpha,n}$ for $n=1,\ldots,\ell$ and $Y_{\overline{\alpha}}(V_{h\alpha,n})\subseteq V_{t\alpha,n+1}$ for $n=1,\ldots,\ell-1$. We will denote $X_{\alpha,n}:=X_{\alpha}|_{V_{t\alpha,n}}$ and $Y_{\overline{\alpha},n}:=Y_{\overline{\alpha}}|_{V_{h\alpha,n}} $.
		\end{itemize}
	\end{theorem}
	\begin{remark}\label{complexMomentMapConditionOnChains}
		For each $i\in Q_0$ and $n=1,\ldots,\ell-1$, the complex moment map condition gives $$\sum_{\substack{h\alpha=i\\\alpha\in Q_1}}X_{\alpha,n+1}Y_{\overline{\alpha},n}-\sum_{\substack{t\alpha=i\\\alpha\in Q_1}}Y_{\overline{\alpha},n}X_{\alpha,n}:V_{i,n}\to V_{i,n+1}\equiv 0.$$
	\end{remark}
	\begin{remark}\label{noEmptyLevels}
		There are no completely empty
		levels. More precisely,
		$$
		\sum_{i\in Q_0}\dim(V_{i,n})>0
		$$
		for every $n=1,\ldots,\ell$. Indeed, suppose that $V_{i,n}=0$ for every $i\in Q_0$ and some
		$1<n<\ell$. Since the ordinary arrows preserve the level and the opposite
		arrows increase it by one, the representations
		\[
		\mathbf V^{<n}
		:=
		\bigoplus_{r<n}\mathbf V_r
		\qquad\text{and}\qquad
		\mathbf V^{>n}
		:=
		\bigoplus_{r>n}\mathbf V_r
		\]
		are subrepresentations of $\mathbf V$. Consequently,
		\[
		\mathbf V=\mathbf V^{<n}\oplus\mathbf V^{>n}.
		\]
		Both summands are nonzero contradicting the indecomposability of the stable representation
		$\mathbf V$. In fact the same argument shows that there cannot be a level $n=1,\ldots,\ell-1$ for which $Y_{\overline{\alpha},n}=0 $ for all $\alpha\in Q_1$. 
	\end{remark}
	\begin{remark}\label{quiverChainsAreQuiverReps}
		Since we are assuming $Q$ acyclic, a quiver chain can be regarded as a representation with relations (see Section \ref{representationsWIthRelations}) of the acyclic quiver $Q^\ell$ given by the following data: $Q^\ell_0=\{(i,n)| i\in Q_0,\ n=1,\ldots,\ell\}$, $Q_1^\ell=\{(\alpha,n)| \alpha\in Q_1,\ n=1,\ldots,\ell\}\sqcup\{(\overline{\alpha},n)|\overline{\alpha}\in \overline{Q}_1\setminus Q_1,n=1,\ldots,\ell-1\}$ and for $\alpha\in Q_1$ we have that $h(\alpha,n)=(h\alpha,n),t(\alpha,n)=(t\alpha,n),h(\overline{\alpha},n)=(h\overline{\alpha},n+1)$ and $t(\overline{\alpha},n)=(t\overline{\alpha},n)$ (see, for instance, Figure \ref{ExampleOfQEll}). We denote by $\mathbf{d}^\ell$ the dimension vector of the representation of $Q^\ell$ obtained from a $\mathbf{d}$-dimensional fixed point. Notice that if $\mathbf{d}^\ell=(d_{i,n})_{\substack{i\in Q_0\\n=1,\ldots,\ell}}$, then  $\sum_{n=1}^\ell d_{i,n}=\text{dim}(V_{i})=d_i$ for all $i\in Q_0$. Notice also that the relations of the representation of $Q^\ell$ obtained can be read from the equality in Remark \ref{complexMomentMapConditionOnChains} and these are: 
		$$
		\mathcal{R}^\ell=\{\sum_{\substack{h\alpha=i\\ \alpha\in Q_1}}(\alpha,n+1) (\overline{\alpha},n)-\sum_{\substack{t\alpha=i\\ \alpha\in Q_1}}(\overline{\alpha},n)(\alpha,n) \mid i\in Q_0, n=1,\ldots,\ell-1 \}.
		$$
		We call such a representation of $Q^\ell$ a $Q$-chain or simply quiver chain and we denote it by $(V_{\bullet\bullet},X_{\bullet\bullet},Y_{\bullet\bullet})$. We will study these in more detail in Section \ref{quiverChainsSection}.
	\end{remark}
	\begin{figure}[H]
		\begin{center}
			\begin{tikzcd}[scale cd=1,row sep=large,column sep= large]
				&&&&\bullet_{(1,3)}\arrow[r,"\alpha_3"]&\bullet_{(2,3)}\arrow[r,"\beta_3"]&\bullet_{(3,3)}\\
				\bullet_1\arrow[r, shift left=1ex,"\alpha"]&\bullet_2\arrow[r, shift left=1ex,"\beta"]\arrow[l,magenta,shift left=1ex,"\overline{\alpha}"]&\bullet_3\arrow[l,magenta,shift left=1ex,"\overline{\beta}"]&\textcolor{blue}{\Longrightarrow}&\bullet_{(1,2)}\arrow[r,"\alpha_2"]&\bullet_{(2,2)}\arrow[r,"\beta_2"]\arrow[lu,magenta,"\overline{\alpha}_2",']&\bullet_{(3,2)}\arrow[lu,magenta,"\overline{\beta}_2",']\\
				&&&&\bullet_{(1,1)}\arrow[r,"\alpha_1"]&\bullet_{(2,1)}\arrow[r,"\beta_1"]\arrow[lu,magenta,"\overline{\alpha}_1",']&\bullet_{(3,1)}\arrow[lu,magenta,"\overline{\beta}_1",']
			\end{tikzcd}
		\end{center}
		\vspace{-2mm}
		\caption{\label{ExampleOfQEll}$Q^\ell$ for $Q=A_3$ and $\ell=3$}
	\end{figure}
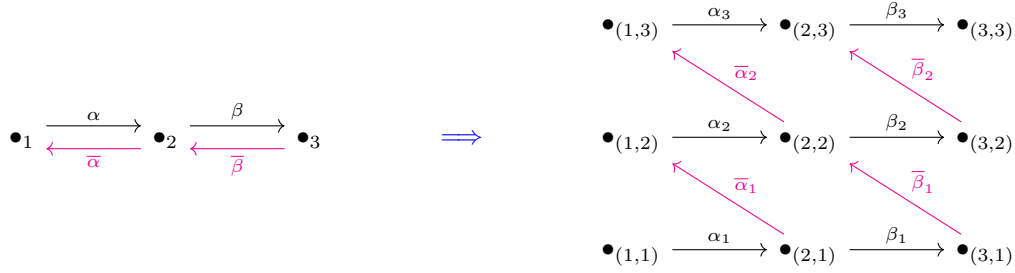
	\begin{remark}\label{underlyingRepOfQSplits}
		If $\mathbf{V}$ is fixed by the $\mathbb{C}^*$-action then the previous theorem implies that $\texttt{Forget}(\mathbf{V})\cong (V_1,X_1)\oplus\ldots\oplus(V_\ell,X_\ell)$ where $V_k=(V_{i,k})_{i\in Q_0}$,  $X_k:=(X_{\alpha,k})_{\alpha\in Q_1}$ and \texttt{Forget} is the forgetful functor introduced in Section \ref{repsDoubleQuiverSection}. Notice, however, that neither $\texttt{Forget}(\mathbf{V})$ nor the direct summands appearing in the decomposition need to be (semi)stable as representations of $Q$.  
	\end{remark}
	The real moment map condition 
	$$
	\mu_\mathbb{R}(\mathbf{V})_i=\theta_{i}\text{Id}_{V_i}
	$$
	restricted to $V_{i,n}$, for $1\leq n \leq \ell$, gives
	\begin{align*}
		\mu_{\mathbb{R}}(\mathbf{V})_{i,n}&=
		\sum_{\substack{h\alpha=i\\\alpha\in Q_1}}X_{\alpha,n}X_{\alpha,n}^*+\sum_{\substack{t\alpha=i\\\alpha\in Q_1}}Y_{\overline{\alpha},n-1}Y_{\overline{\alpha},n-1}^* \\ &- \sum_{\substack{t\alpha=i\\\alpha\in Q_1}} X_{\alpha,n}^*X_{\alpha,n}- \sum_{\substack{h\alpha=i\\\alpha\in Q_1}}Y_{\overline{\alpha},n}^*Y_{\overline{\alpha},n}=\theta_{i}\text{Id}_{V_{i,n}}
	\end{align*}
	Summing over all $i\in Q_0$ and taking trace on both sides of the equality gives 
	$$
	\sum_{\alpha\in Q_1}(\|Y_{\overline{\alpha},n-1}\|^2-\|Y_{\overline{\alpha},n}\|^2)=\sum_{i\in Q_0}\theta_{i}d_{i,n}.
	$$
	Hence, at a critical point $[\mathbf{V}]\in \mathcal{M}^{\theta}(\overline{Q},\mathbf{d})$, the value of the moment map will be 
	$$
	F([\mathbf{V}])=\sum_{\alpha\in Q_1}\|Y_{\overline{\alpha}}\|^2=\sum_{\alpha\in Q_1}\sum_{n=1}^{\ell-1}\|Y_{\overline{\alpha},n}\|^2=\sum_{n=1}^{\ell-1}\sum_{k=1}^{n}\sum_{i\in Q_0}-\theta_{i}d_{i,k}.
	$$
	This implies the following:
	\begin{proposition}
		The moment map $F:\mathcal{M}^{\theta}(\overline{Q},\mathbf{d})\to \mathbb{R}$ has finitely many critical values. 
	\end{proposition}
	
	\subsection{Existence of a universal family and liftings of the $\mathbb{C}^*$-action}\label{universalFamilyNakajimaQuiverVarietiesAndLiftingsTorusAction}
	Later on, the existence of a universal family will be helpful for obtaining some of the results. We will follow \cite[Section 5]{KingModuli} and specialize it to the situation at hand. 
	\begin{definition}
		A \emph{family of Nakajima quiver representations} over a scheme $S$ is given by a tuple $\mathcal{V}=((\mathcal{V}_i)_{i\in Q_0},(\mathcal{X}_{\alpha})_{\alpha\in Q_1},(\mathcal{Y_{\overline{\alpha}}})_{\alpha\in Q_1})$ where $\mathcal{V}_i$ is a locally free sheaf over $S$ and $\mathcal{X}_{\alpha}:\mathcal{V}_{t\alpha}\to \mathcal{V}_{h\alpha},\ \mathcal{Y}_{\overline\alpha}:\mathcal{V}_{h\alpha}\to \mathcal{V}_{t\alpha}$ are sheaf morphisms for each $\alpha\in Q_1$ such that 
		$$
		\sum_{\substack{h\alpha=i\\\alpha\in Q_1}}\mathcal{X}_\alpha\mathcal{Y}_{\overline{\alpha}}-\sum_{\substack{t\alpha=i\\\alpha\in Q_1}}\mathcal{Y}_{\overline{\alpha}}\mathcal{X}_\alpha\equiv 0
		$$
		for all $i\in Q_0$. Two such families \(\mathcal V\) and \(\mathcal V'\) are said to be
		\emph{equivalent} if there exist a line bundle \(\mathcal L\) on \(S\)
		and isomorphisms
		$$
		\varphi_i:\mathcal V_i'
		\xrightarrow{\cong}
		\mathcal V_i\otimes\mathcal L,
		$$
		for all $i\in Q_0$
		compatible with all the arrow maps. That is, for every
		$\alpha\in Q_1$, the diagrams
		$$
		\begin{tikzcd}[column sep=large,row sep=large]
			\mathcal V'_{t\alpha}
			\arrow[r,"\mathcal X'_\alpha"]
			\arrow[d,"\varphi_{t\alpha}"']
			&
			\mathcal V'_{h\alpha}
			\arrow[d,"\varphi_{h\alpha}"]
			\\
			\mathcal V_{t\alpha}\otimes\mathcal L
			\arrow[r,"\mathcal X_\alpha\otimes
			\operatorname{Id}_{\mathcal L}"']
			&
			\mathcal V_{h\alpha}\otimes\mathcal L
		\end{tikzcd}
		\text{and}
		\begin{tikzcd}[column sep=large,row sep=large]
			\mathcal V'_{h\alpha}
			\arrow[r,"\mathcal Y'_{\overline\alpha}"]
			\arrow[d,"\varphi_{h\alpha}"']
			&
			\mathcal V'_{t\alpha}
			\arrow[d,"\varphi_{t\alpha}"]
			\\
			\mathcal V_{h\alpha}\otimes\mathcal L
			\arrow[r,"\mathcal Y_{\overline\alpha}\otimes
			\operatorname{Id}_{\mathcal L}"']
			&
			\mathcal V_{t\alpha}\otimes\mathcal L
		\end{tikzcd}
		$$
		commute.
	\end{definition}
	We define the moduli functor
	\begin{equation}\label{moduliFunctorNakajimaQuiverRepresentations}
	\mathcal F_{\mathbf d}^{\theta}:(\operatorname{Sch}/\mathbb C)^{\operatorname{op}}\longrightarrow \operatorname{Sets}
	\end{equation}
	by
	$$
	\mathcal F_{\mathbf d}^{\theta}(S)
	=
	\left\{
	\begin{array}{c}
		\text{\(\theta\)-stable families of Nakajima quiver}\\
		\text{representations of dimension vector \(\mathbf d\) over \(S\)}
	\end{array}
	\right\}
	\Big/\sim.
	$$
	And, for a morphism $f:T\to S$, the corresponding map is defined by
	$$
	\begin{matrix}
			\mathcal F_{\mathbf d}^{\theta}(f):&	\mathcal F_{\mathbf d}^{\theta}(S)&\longrightarrow& \mathcal F_{\mathbf d}^{\theta}(T)\\
		&	[\mathcal V]&\longmapsto&[f^*\mathcal V]
	\end{matrix}
	$$
	which is well defined since
	$f^*(\mathcal V\otimes\mathcal L)\cong (f^*\mathcal V\otimes f^*\mathcal L)\sim f^*\mathcal{V}$.
	
	\begin{proposition}\label{existenceUniversalFamilyNakajimaQuiverVarieties}
		For indivisible dimension vector $\mathbf{d}$ and generic stability parameter $\theta$, $\mathcal{M}^{\theta}(\overline{Q},\mathbf{d})$ is a fine moduli space for families of stable Nakajima quiver representations and represents the functor $\mathcal{F}^{\theta}_{\mathbf{d}}$. 
	\end{proposition}
	\begin{proof}
		For all $i\in Q_0$, we let $\widetilde{\mathcal{V}_i}\to \mu^{-1}_{\mathbb{C}}(0)^{\theta\text{-}s}$ be the trivial bundle of rank $d_i$. For $g\in \text{GL}(\mathbf{d})$, $g$ acts on $(\mathbf{V},v)\in (\widetilde{\mathcal{V}}_i)_{\{\mathbf{V}\}}$ as $g\cdot (\mathbf{V},v)=(g\cdot\mathbf{V},g_iv)$. For each edge $\alpha\in Q_1$ we define the morphism $\widetilde{X}_{\alpha}:\widetilde{\mathcal{V}}_{t\alpha}\to \widetilde{\mathcal{V}}_{h\alpha}$ as $(\mathbf{V},v)\mapsto (\mathbf{V},X_{\alpha}(v))$ and similarly we define $\widetilde{Y}_{\overline\alpha}:\widetilde{\mathcal{V}}_{h\alpha}\to \widetilde{\mathcal{V}}_{t\alpha}$. Notice that these morphisms are $\text{GL}(\mathbf{d})$-equivariant. 
		\\
		\\
		Now, we recall that the stabilizer of each representation in $\mu^{-1}_{\mathbb{C}}(0)^{\theta\text{-}s}$ is just $\mathbb{C}^*$. So for the tautological family that we have just constructed to descend to $\mathcal{M}^{\theta}(\overline{Q},\mathbf{d})$ we need the stabilizer to act trivially in the fiber direction. We achieve this by twisting each bundle $\widetilde{\mathcal{V}}_i$ by the trivial line bundle on which $\text{GL}(\mathbf{d})$ acts via a character $\chi:\text{GL}(\mathbf{d})\to \mathbb{C}^*$ acting with weight $-1$ when restricted to $\mathbb{C}^*$. For such a character to exist is a necessary and sufficient condition for the dimension vector $\mathbf{d}$ to be indivisible. This construction gives the desired universal family over $\mathcal{M}^{\theta}(\overline{Q},\mathbf{d})$ which we denote $\mathbb{M}^{\theta}(\overline{Q},\mathbf{d})=((\mathbb{V}_i)_{i\in Q_0},(\mathbb{X}_{\alpha})_{\alpha\in Q_1},(\mathbb{Y}_{\overline\alpha})_{\alpha\in Q_1})$.
	\end{proof}
	We now turn to $\mathbb{C}^*$-equivariant structures on families taking inspiration from Hausel and Thaddeus' work on Higgs bundles \cite[(4.3) and (4.4)]{HauselThaddeus}: 
	\begin{definition}\label{equivariantStructures}
		Let $S$ be a scheme with a $\mathbb{C}^*$-action, and let
		\[
		\mathcal V=
		\left(
		(\mathcal V_i)_{i\in Q_0},
		(\mathcal X_\alpha)_{\alpha\in Q_1},
		(\mathcal Y_{\overline\alpha})_{\alpha\in Q_1}
		\right)
		\]
		be a family of Nakajima quiver representations over $S$. A
		\emph{$\mathbb{C}^*$-equivariant structure} on $\mathcal V$ is a collection of
		isomorphisms
		\[
		\rho_i(t,s):(\mathcal V_i)_s\longrightarrow (\mathcal V_i)_{t\cdot s},
		\qquad i\in Q_0,
		\]
		depending algebraically on $t\in\mathbb{C}^*$ and $s\in S$, satisfying: 
		\begin{itemize}
			\item The equivariant condition
			$$
			\rho_i(t_1t_2,s)
			=
			\rho_i(t_1,t_2\cdot s)\circ \rho_i(t_2,s)
			$$
			for all $t_1,t_2\in\mathbb C^*$ and $s\in S$.
			\item Compatibility with the arrow maps: for ordinary arrows $\alpha\in Q_1$, one has
			\begin{equation}\label{compatibilityEqStructureWithArrows}
				\mathcal X_{\alpha}\circ \rho_{t\alpha}(t,s)
				=
				\rho_{h\alpha}(t,s)\circ \mathcal X_{\alpha},
			\end{equation}
			whereas for opposite arrows $\overline\alpha$, one has
			\begin{equation}\label{compatibilityEqStructureWithOppositeArrows}
				\mathcal Y_{\overline\alpha}\circ \rho_{h\alpha}(t,s)
				=
				t\rho_{t\alpha}(t,s)\circ \mathcal Y_{\overline\alpha}.
			\end{equation}
			Thus ordinary arrows have weight $0$, while opposite arrows have weight $1$.
		\end{itemize}
		
	\end{definition}
	\begin{lemma}\label{twoEquivariantStructuresDifferByALineBundle}
		Let $\mathcal V$
		be a family of stable Nakajima quiver representations over a base
		scheme $S$ with $\mathbb{C}^*$-action. Suppose that \(\rho\) and \(\rho'\) are two
		$\mathbb C^*$-equivariant structures on
		$\mathcal V$, then $\rho$ and $\rho'$ differ
		by tensoring with a $\mathbb C^*$-linearized trivial line bundle.
	\end{lemma}
	\begin{proof}
		We compare both equivariant structures by setting
		\[
		g_i(t,s):=\rho_i(t,s)^{-1}\circ \rho_i'(t,s):
		(\mathcal V_i)_{\{s\}}\longrightarrow(\mathcal V_i)_{\{s\}}.
		\]
		Then $(g_i(t,s))_{i\in Q_0}$ defines an automorphism of the stable
		representation $\mathcal V_s$ and therefore there exists a scalar
		$c(t,s)\in\mathbb C^*$ depending regularly on $t$ and $s$ and independent of $i\in Q_0$, such that
		$$
		g_i(t,s)=c(t,s)\operatorname{Id}_{(\mathcal V_i)_{\{s\}}}.
		$$
		Notice that
		$$
		c(t_1t_2,s)=c(t_1,t_2\cdot s)c(t_2,s)
		$$
		for all $t_1,t_2\in\mathbb C^*$ and $s\in S$. Therefore $c$ defines a
		$\mathbb C^*$-linearization of the trivial line bundle $\mathscr O_S$. Tensoring every vertex bundle $\mathcal V_i$ with this linearized trivial line
		bundle changes the equivariant structure $\rho$ into $\rho'$.
	\end{proof}
	
	\begin{proposition}\label{equivariantStructureUniversalFamily}
		There is a $\mathbb{C}^*$-equivariant structure on the universal family $\mathbb{M}^{\theta}(\overline{Q},\mathbf{d})$ acting with weight 0 on ordinary arrows
		$\alpha\in Q_1$ and weight one on opposite arrows $\overline\alpha$. 
	\end{proposition}
	\begin{proof}
		We keep the notation from Proposition \ref{existenceUniversalFamilyNakajimaQuiverVarieties}. Define a $\mathbb C^*$-action on each tautological vertex bundle $\widetilde{\mathcal{V}}_i$ by
		$$
		t\cdot(\mathbf{V},v):=(t\cdot \mathbf{V},v)
		$$
		and notice that this covers the $\mathbb C^*$-action on $\mu_{\mathbb{C}}^{-1}(0)^{\theta\text{-}s}$. We now see how this $\mathbb{C}^*$-action interacts with the tautological arrow maps. Let
		$\alpha$ be an ordinary arrow, then
		\begin{equation}\label{equivarianceArrowsBeforeDescending}
			\widetilde{\mathcal{X}}_\alpha(t\cdot(\mathbf{V},v))=(t\cdot\mathbf{V},X_{\alpha}(v))=t\cdot(\mathbf{V},X_{\alpha}(v))=t\cdot \widetilde{\mathcal{X}}_{\alpha}(\mathbf{V},v).
		\end{equation}
		
		For an opposite arrow $\overline{\alpha}$ we get
		\begin{equation}\label{equivarianceInverseArrowsBeforeDescending}
			\widetilde{\mathcal{Y}}_{\overline\alpha}(t\cdot(\mathbf{V},v))= (t\cdot \mathbf{V},tY_{\overline{\alpha}}(v))
			=t(t\cdot\widetilde{\mathcal{Y}}_{\overline{\alpha}}(\mathbf{V},v)).
		\end{equation}
		It remains to see that this \(\mathbb C^*\)-equivariant structure descends
		to the universal family $\mathbb{M}^{\theta}(\overline{Q},\mathbf{d})$. Recall, from Proposition \ref{existenceUniversalFamilyNakajimaQuiverVarieties}, that the $\text{GL}(\mathbf{d})$-action on $\widetilde{\mathcal{V}}_i$ is given by
		$$
		g\cdot(\mathbf{V},v)=(g\cdot \mathbf{V},\chi(g)g_i v)
		$$
		for some character $\chi$. Then
		$$
		t\cdot(g\cdot(\mathbf{V},v))=t\cdot(g\cdot \mathbf{V},\chi(g)g_i v)=
		(t\cdot(g\cdot \mathbf{V}),\chi(g)g_i v),
		$$
		while
		$$
		g\cdot(t\cdot(\mathbf{V},v))=
		g\cdot(t\cdot \mathbf{V},v)=
		(g\cdot(t\cdot \mathbf{V}),\chi(g)g_i v).
		$$
		Since the $\text{GL}(\mathbf{d})$-action on $\mu^{-1}_{\mathbb{C}}(0)^{\theta\text{-}s}$ commutes with the
		$\mathbb C^*$-action, this descends to the universal family. 
		Consequently, for every $t\in\mathbb{C}^*$ and every
		$[\mathbf{V}]\in \mathcal{M}^{\theta}(\overline{Q},\mathbf{d})$, we obtain isomorphisms
		$$
		\rho_i(t):
		(\mathbb{V}_i)_{\{[\mathbf{V}]\}}
		\to
		(\mathbb{V}_i)_{\{t\cdot[\mathbf{V}]\}}
		$$
		for all $i\in Q_0$. The relations in Equations (\ref{equivarianceArrowsBeforeDescending}) and (\ref{equivarianceInverseArrowsBeforeDescending}) now become respectively:
		$$
		\mathbb{X}_{\alpha}|_{t\cdot[\mathbf{V}]}
		\circ \rho_{t\alpha}(t)=\rho_{h\alpha}(t)\circ \mathbb{X}_{\alpha}|_{[\mathbf{V}]},
		$$
		and
		$$
		\mathbb{Y}_{\overline\alpha}|_{t\cdot[\mathbf{V}]}
		\circ \rho_{h\alpha}(t)=
		t\ 
		\rho_{t\alpha}(t)
		\circ \mathbb{Y}_{\overline\alpha}|_{[\mathbf{V}]}.
		$$
		Thus the universal family carries the desired
		\(\mathbb C^*\)-equivariant structure.
	\end{proof}
	\subsection{More on quiver chains and their moduli}\label{quiverChainsSection}
	As we pointed out in Remark \ref{quiverChainsAreQuiverReps}, fixed points for the $\mathbb{C}^*$-action on $\mathcal{M}^{\theta}(\overline{Q},\mathbf{d})$ can be seen as quiver chains which are representations with relations of a new quiver, $Q^\ell$, which depends on a positive integer $\ell$ and the data of the original quiver $Q$. The goal of this section is to show that the classification problem of the fixed points is equivalent to that of quiver chains. First, we define the collapsing morphism which will play an important role in the section:
	$$
	\resizebox{\textwidth}{!}{$
	\begin{matrix}
	\operatorname{Rep}(Q^\ell,\mathbf{d}^\ell)&\overset{\tilde{\iota}_{\mathbf{d}^\ell}}{\longrightarrow}&\operatorname{Rep}(\overline{Q},\mathbf{d}) \\
		(V_{\bullet\bullet},X_{\bullet\bullet},Y_{\bullet\bullet}) & \longmapsto& \left ((V_i:=\bigoplus\limits_{k=1}^\ell V_{k,i})_{i\in Q_0},(X_\alpha:=\bigoplus\limits_{n=1}^{\ell}X_{\alpha,n}:V_{t\alpha}\to V_{h\alpha},Y_{\overline\alpha}:=
		\bigoplus\limits_{n=1}^{\ell-1}Y_{\overline\alpha,n}:V_{h\alpha}\rightarrow V_{t\alpha})_{\alpha\in Q_1}\right). 
	\end{matrix} $}
	$$
	\begin{proposition}\label{semistableRepresentationIffSemistableChain}
		A quiver chain $\mathcal{C}=(V_{\bullet\bullet},X_{\bullet\bullet},Y_{\bullet\bullet})$ is $\theta^{\ell}=(\theta_{i,n})_{(i,n)\in Q_0^\ell}$-(semi)stable if and only if the Nakajima quiver representation $\mathbf{V}=\tilde{\iota}_{\mathbf{d}^\ell}(\mathcal{C})$ obtained by collapsing is $\theta$-(semi)stable. Here $\theta_{i,n}=\theta_i$ for all $i\in Q_0$ and $n=1,\ldots,\ell$.
	\end{proposition}
	\begin{proof}
		We follow a similar strategy to that of Álvarez-Cónsul and García-Prada in the context of Higgs bundles \cite[Proposition 3.5]{AlvarezCGarciaPDimensionalReduction}.  Suppose that $\mathbf{V}$ is $\theta$-(semi)stable and let $\mathcal{C}'$ be a subchain of $\mathcal{C}$. The representation $\mathbf{W}$, of the doubled quiver $\overline{Q}$, associated to the subchain $\mathcal{C}'$ and obtained by collapsing the vertices is a subrepresentation of $\mathbf{V}$, then $\theta\cdot\mathbf{d}_{\mathbf{W}}(\leq)<0$ by the (semi)stability of $\mathbf{V}$.
		\\
		\\
		Conversely, assume that the quiver chain $\mathcal{C}$ is $\theta^\ell$-(semi)stable We have to show that $\mathbf{V}$ is $\theta$-(semi)stable or, in other words, that $\theta\cdot\mathbf{d}_{\mathbf{W}}(\leq)< 0$ for all subrepresentation $\mathbf{W}$ of $\mathbf{V}$. Recall from Remark \ref{underlyingRepOfQSplits} that $\texttt{Forget}(\mathbf{V})=(V_1,X_1)\oplus\cdots\oplus (V_\ell,X_\ell)$ so let us denote by $\pi_n:\texttt{Forget}(\mathbf{V})\to (V_n,X_n)$ the canonical projection which is a morphism of representations of the quiver $Q$. We define recursively, for $n=1,\ldots,\ell$,
		$$
		\mathbf{W}^{\geq n}=\begin{cases}
			\texttt{Forget}(\mathbf{W}) & \text{if }n=1, \\
			\ker(\pi_{n-1}|_{\mathbf{W}^{\geq n-1}}) & \text{if }2\leq n\leq \ell\\
		\end{cases}, \ \mathbf{W}^n=\text{im}(\pi_n|_{\mathbf{W}^{\geq n}}).
		$$
		There is, therefore, for all $n=1,\ldots,\ell$, an exact sequence of representations of the quiver $Q$:
		\begin{center}
			\begin{tikzcd}[column sep =normal, row sep =normal]
				0\arrow[r]&\mathbf{W}^{\geq n+1}\arrow[r,""]& \mathbf{W}^{\geq n} \arrow[r] &\mathbf{W}^{n}\arrow[r]&0.
			\end{tikzcd}
		\end{center}
		Now, note that the representations $\mathbf{W}^n$ together with the morphisms $Y_{\overline{\alpha}}$ define a quiver subchain $\mathcal{C}'$ of $\mathcal{C}$. Thus, by the $\theta^\ell$-(semi)stability of $\mathcal{C}$ we have that 
		$$
		\theta^\ell\cdot \dim(\mathcal{C}')(\leq)<0.
		$$ 
		But, 
		$$
		\theta^\ell\cdot\dim(\mathcal{C}')=\sum_{n=1}^\ell\sum_{i\in Q_0}\theta_{i,n}\dim(W^n_i)=\sum_{i\in Q_0}\theta_i\dim(W_i)=\theta\cdot\mathbf{d}_{\mathbf{W}}
		$$
		hence the conclusion. 
	\end{proof}
	\begin{remark}
		An alternative proof, allowing to conclude polystability rather than semistability, can be given using the real moment map condition \cite[Proposition 4.9]{NumpaqueTensorProducts}. Indeed, in op. cit. it is shown analytically that the quiver representation obtained by collapsing any number of vertices on a polystable representation is polystable as long as the corresponding stability parameters of the collapsed vertices agree. In a sense Proposition \ref{semistableRepresentationIffSemistableChain} is an algebraic proof of this fact, as the representation $\mathbf{V}$ of $\overline{Q}$ is obtained from the quiver chain $\mathcal{C}$ from collapsing all the vertices $i_n$, for $n=1,\ldots,\ell$, into one. 
	\end{remark}
	Following Section \ref{representationsWIthRelations} and Remark \ref{quiverChainsAreQuiverReps} , for a positive integer $\ell$, we define the \emph{moduli space of quiver chains} to be
	$$
	\mathcal{M}^{\theta^\ell}(Q^\ell,\mathbf{d}^\ell,\mathcal{R}^\ell)= \operatorname{Rep}(Q^\ell,\mathbf{d}^\ell,\mathcal{R}^\ell)^{\theta^\ell\text{-}s}\sslash\operatorname{GL}(\mathbf{d^\ell}).
	$$
	Since $\mathbf{d}$ is indivisible and $\theta$ generic, $\mathbf{d}^\ell$ will be indivisible and $\theta^\ell$ generic for this dimension vector. This means that semistability coincides with stability and that this moduli space admits a universal family that represents the functor $
	\mathcal F_{\mathbf d^\ell}^{\theta^\ell}:(\operatorname{Sch}/\mathbb C)^{\operatorname{op}}
	\rightarrow \operatorname{Sets}$ classifying equivalence classes of families of $\theta^\ell$-stable quiver chains of dimension vector $\mathbf{d}^\ell$. This is done in a similar fashion to what we did for Nakajima quiver varieties in the beginning of Section \ref{universalFamilyNakajimaQuiverVarietiesAndLiftingsTorusAction}. 

	\begin{theorem}\label{quiverChainsClosedImmersion}
		Let 
	$$
				\resizebox{\textwidth}{!}{$\displaystyle
					\mathscr{D}_\theta(\mathbf{d})=\{\mathbf{d}^\ell \ |\ \sum_{n=1}^{\ell}d_{i,n}=d_i\ \forall i\in Q_0,\ \sum_{i\in Q_0}d_{i,n}>0\ \forall n=1,\ldots,\ell{\normalfont \text{ and }}\mathcal{M}^{\theta^\ell}(Q^\ell,\mathbf{d}^\ell,\mathcal{R}^\ell)\neq\varnothing \}.$}
	$$
	Denote $Y:=\mathcal M^\theta(\overline Q,\mathbf d)$
		and
		$X_{\mathbf d^\ell}:=\mathcal M^{\theta^\ell}(Q^\ell,\mathbf d^\ell,\mathcal R^\ell)$. Then the collapsing morphisms $\tilde{\iota}_{\mathbf{d}^\ell}$ induce closed immersions
		$\iota_{\mathbf d^\ell}:X_{\mathbf d^\ell}\rightarrow Y$ with image contained in $Y^{\mathbb{C}^*}$. These morphisms, in turn, give an isomorphism 
		$$
		Y^{\mathbb{C}^*}\cong\bigsqcup_{\mathbf{d}^\ell\in \mathscr{D}_{\theta}(\mathbf{d})}X_{\mathbf{d}^\ell}.
		$$
	\end{theorem}
	
	\begin{proof}
		First we comment on why the map $\iota_{\mathbf{d}^\ell}$ is well-defined. By Proposition
		\ref{semistableRepresentationIffSemistableChain}, stability of the quiver chain is
		equivalent to stability of the associated representation of \(\overline Q\). Since the relations
		$\mathcal R^\ell$ are precisely the graded components of the equation $\mu_{\mathbb C}(\mathbf{V})=0$, the image of a $\theta^\ell$-stable chain under $\tilde{\iota}_{\mathbf{d}^\ell}$ defines a $\theta$-stable representation of $\overline Q$ satisfying the complex moment map condition. Moreover, the image of $\tilde{\iota}_{\mathbf{d}^\ell}$ is contained in the fixed locus $Y^{\mathbb{C}^*}$ essentially because the family of automorphisms described in Equation (\ref{automorphismRepresentationAdmitingDecompositionIsFixedPoint}). Thus the restriction of the collapsing map $\tilde{\iota}_{\mathbf{d}^\ell}$ to the respective stable loci is well-defined. Furthermore, this map is equivariant with respect to the inclusion $\text{GL}(\mathbf{d}^\ell)\hookrightarrow\text{GL}(\mathbf{d})$ so it descends to the respective quotients giving the desired well-definedness of $\iota_{\mathbf{d}^\ell}$.
		\\
		\\
		Notice that 
		$$
		Y^{\mathbb{C}^*}=\bigsqcup_{\mathbf{d}^\ell\in \mathscr{D}_{\theta}(\mathbf{d})} Y^{\mathbb{C}^*}_{\mathbf{d}^\ell}
		$$
		where $Y_{\mathbf{d}^\ell}^{\mathbb{C}^*}\subseteq Y^{\mathbb{C}^*}$ is the locus of representations that admit a weight decomposition as in Theorem \ref{characterizationFixedPoints}. We now describe the scheme structure of this subset. We start by recalling that the fixed point locus $Y^{\mathbb{C}^*}$ is a closed subscheme of $Y$ \cite[Theorem 7.1]{MilneAlgebraicGroups}. We claim that $Y^{\mathbb{C}^*}_{\mathbf{d}^\ell}$ is a closed subscheme of the fixed point scheme $Y^{\mathbb{C}^*}$ given by the union of connected components of it. To see this we will be needing the following lemma:
		\begin{lemma}\label{weightDecompositionDimensionIsLocallyConstant}
			The map $\mathbf{V}\mapsto \mathbf{d}^{\ell(\mathbf{V})}$ sending a representation fixed by the $\mathbb{C}^*$-action to the dimension vector of the associated quiver chain is locally constant. 
		\end{lemma} 
		\begin{proof}[Proof of Lemma \ref{weightDecompositionDimensionIsLocallyConstant}]
			Recall that in Proposition \ref{equivariantStructureUniversalFamily} we showed that the universal family $\mathbb{M}^{\theta}(\overline{Q},\mathbf{d})$, constructed in Proposition \ref{existenceUniversalFamilyNakajimaQuiverVarieties}, has a $\mathbb{C}^*$-equivariant structure. Therefore, so it does the restriction of this family to the fixed point scheme $Y^{\mathbb{C}^*}$. The action on this scheme being trivial means that there is a splitting $\mathbb{V}_{i}|_{Y^{\mathbb{C}^*}}=\bigoplus_{n\in\mathbb{Z}} \mathbb{V}_{i,n}$, $\mathbb{V}_{i,n}=\{v\in \mathbb{V}_{i}|_{Y^{\mathbb{C}^*}}\ | \ \rho_i(t)(v)=t^{-n}v, \ \forall t\in\mathbb{C}^*\}$, induced by the automorphisms $\rho_i(t)$ defined in Proposition \ref{equivariantStructureUniversalFamily}. A way to see this is by trivializing $\mathbb{V}_i$, seeing that there is a such a splitting locally and then checking that these local splittings glue together into vector bundles over $Y^{\mathbb{C}^*}$. Since the rank of each $\mathbb{V}_{i,n}$ is locally constant, it only remains to show that $\operatorname{rk}(\mathbb{V}_{i,n})=\dim(V_{i,n})$. The fiber $(\mathbb{V}_{i})_{\{[\mathbf{V}]\}}$ is, by construction, the vector space $V_i$. By restricting $\rho_i$ to this fiber we see that $(\mathbb{V}_{i,n})_{\{[\mathbf{V}]\}}$ is precisely $V_{i,n}$ which finishes the proof. 
		\end{proof}
		The previous lemma implies our claim for the following two reasons: locally constant maps are constant on connected components and, on the other hand, the map in the lemma has finite image, meaning that the preimage of a point is a clopen subset of the fixed point scheme. 
		\\
		\\
		The proof is reduced to see that $X_{\mathbf d^\ell}\cong Y^{\mathbb C^*}_{\mathbf d^\ell}$ as schemes via $\iota_{\mathbf{d}^\ell}$ and to see this we will show that $Y^{\mathbb{C}^*}_{\mathbf{d}^\ell}$ represents the functor $\mathcal{F}_{\mathbf{d}_\ell}^{\theta^\ell}$. 
		\\
		\\
		Note that the splittings of the universal bundles $\mathbb{V}_i$ seen in Lemma \ref{weightDecompositionDimensionIsLocallyConstant} are compatible with the arrow maps in the sense that $
		\mathbb{X}_\alpha(\mathbb V_{t\alpha,n})
		\subseteq
		\mathbb V_{h\alpha,n}
		$
		and
		$
		\mathbb{Y}_{\overline\alpha}(\mathbb V_{h\alpha,n})
		\subseteq
		\mathbb V_{t\alpha,n+1}
		$ and that the relations inherited from the complex moment map hold. Thus, the restriction of the universal family to $Y^{\mathbb{C}^*}_{\mathbf{d}^\ell}$ defines a family of $\theta^\ell$-stable quiver chains and therefore, so will do the pullback of this family through a morphism $f:T\to Y^{\mathbb{C}^*}_{\mathbf{d}^\ell}$. This construction
		gives a map
		$\operatorname{Hom}(T,Y^{\mathbb{C}^*}_{\mathbf{d}^\ell})
		\rightarrow
		\mathcal F^{\theta^\ell}_{\mathbf d^\ell}(T)
		$.
		\\
		\\
		Conversely, let $(\mathcal{V}_{\bullet\bullet},\mathcal{X}_{\bullet\bullet},\mathcal{Y}_{\bullet\bullet})$ be a family of $\theta^\ell$-stable quiver chains of dimension vector $\mathbf d^\ell$ over a scheme $T$. The collapsed family $(\mathcal{V}_i:=\bigoplus_{n=1}^{\ell}\mathcal{V}_{i,n},\mathcal{X}_\alpha:=\bigoplus_{n=1}^{\ell}\mathcal{X}_{\alpha,n},\mathcal{Y}_{\overline{\alpha}}:=\bigoplus_{n=1}^{\ell-1}\mathcal{Y}_{\overline{\alpha},n})$ gives a family of Nakajima quiver representations over $T$. Since $Y$ is a fine moduli space for Nakajima quiver representations, this family determines an unique morphism $T\rightarrow Y$. This morphism factors through $Y^{\mathbb{C}^*}_{\mathbf{d}^\ell}$. Indeed, from the decomposition
		$\mathcal{V}_i=\bigoplus_{n=1}^{\ell}\mathcal{V}_{i,n}$ we define, for every $t\in\mathbb C^*$, bundle automorphisms
		$g_i(t):\mathcal{V}_i\rightarrow \mathcal{V}_i$ given by
		$g_i(t)|_{\mathcal{V}_{i,n}}=t^n\operatorname{Id}_{\mathcal{V}_{i,n}}$. We claim that the family of automorphisms $(g_i(t))_{i\in Q_0}$ identify the collapsed family with its $\mathbb C^*$-translate. Note that $g_{h\alpha}(t)\mathcal{X}_\alpha g_{t\alpha}(t)^{-1}=\mathcal{X}_\alpha$ and that
		$g_{t\alpha}(t)\mathcal{Y}_{\overline\alpha}g_{h\alpha}(t)^{-1}=	t\mathcal{Y}_{\overline\alpha}$. Hence the associated morphism \(T\to Y\) factors through the fixed-point locus $Y^{\mathbb C^*}_{\mathbf{d}^\ell}$. This construction gives a natural map $\mathcal F^{\theta^\ell}_{\mathbf d^\ell}(T)\rightarrow \operatorname{Hom}(T,Y^{\mathbb{C}^*}_{\mathbf{d}^\ell})$.
		\\
		\\
		The two constructions just described are inverse to one another. Thus $Y^{\mathbb{C}^*}_{\mathbf{d}^\ell}$ represents the same functor as $X_{\mathbf d^\ell}$ and, by Yoneda's
		lemma, there is a canonical isomorphism $X_{\mathbf d^\ell}\cong Y^{\mathbb C^*}_{\mathbf d^\ell}$. This isomorphism is actually induced by $\iota_{\mathbf d^\ell}$ by our preceding discussion. By the universal property of the coproduct \cite[Section 7.3.8]{VakilAG} these isomorphisms glue into 
		$$
		\bigsqcup_{\mathbf{d}^\ell\in \mathscr{D}_{\theta}(\mathbf{d})}X_{\mathbf{d}^\ell}\overset{\simeq}{\longrightarrow} Y^{\mathbb{C}^*}.
		$$
	\end{proof}
	
	\section{Bia\l ynicki-Birula stratification and its geometry}\label{BialynickiBirulaSection}
	\subsection{Bia\l ynicki-Birula stratification}\label{BialynickiBirulaStratification}
	
	Let $\{F_j\}_{j\in J}$ be the set of connected components of the fixed point set for the $\mathbb{C}^*$-action on $\mathcal{M}^{\theta}(\overline{Q},\mathbf{d})$.By Theorem \ref{quiverChainsClosedImmersion}, there is an isomorphism
	\[
	\mathcal M^\theta(\overline Q,\mathbf d)^{\mathbb C^*}
	\cong
	\bigsqcup_{\mathbf d^\ell\in\mathscr D_\theta(\mathbf d)}
	\mathcal M^{\theta^\ell}
	(Q^\ell,\mathbf d^\ell,\mathcal R^\ell).
	\]
	Consequently, every \(F_j\) is a connected component of one of the
	quiver-chain moduli spaces appearing on the right-hand side, and each such
	moduli space may contribute more than one connected component. Since the
	fixed-point locus is of finite type, the set \(J\) is finite. Moreover, the $F_j$ are smooth projective subvarieties of $\mathcal{M}^{\theta}(\overline{Q},\mathbf{d})$. For all $[\mathbf{V}]\in \mathcal{M}^{\theta}(\overline{Q},\mathbf{d})$, $\lim_{t\to 0}t\cdot [\mathbf{V}]$ exists and by this we mean that there is a $\mathbb{C}^*$-equivariant map $f:\mathbb{A}^1_\mathbb{C}\to \mathcal{M}^{\theta}(\overline{Q},\mathbf{d})$, with respect to the standard $\mathbb{C}^*$-action on the affine line, such that $f(1)=[\mathbf{V}]$. For proofs of these facts about the $\mathbb{C}^*$-action, we refer the reader to the works of Crawley-Boevey, Van den Bergh and Hoskins \cite{CrawleyBoeveyVanDenBergh} \cite{HoskinsSurvey}. 
	\\
	\\
	A $\mathbb{C}^*$-action with the properties just described is called \emph{semi-projective} \cite{HauselRodriguezVillegas} and one of its key features is that it gives rise to a so-called \emph{Bia\l ynicki-Birula stratification} of the moduli space $\mathcal{M}^{\theta}(\overline{Q},\mathbf{d})$. By this we mean that 
	$$
	\mathcal{M}^{\theta}(\overline{Q},\mathbf{d})=\bigsqcup_{j\in J}U_j ^+
	$$
	where $$U_j^+=\{[\mathbf{V}]\in\mathcal{M}^{\theta}(\overline{Q},\mathbf{d}) |  \lim_{t\to 0}t\cdot [\mathbf{V}]\in F_j\}.$$ Furthermore, note that $U_j^+=\bigsqcup_{[\mathbf{W}]\in F_j}U_{[\mathbf{W}]}^+$ for $U_{[\mathbf{W}]}^+=\{[\mathbf{V}]\in\mathcal{M}^{\theta}(\overline{Q},\mathbf{d}) |  \lim_{t\to 0}t\cdot [\mathbf{V}]=[\mathbf{W}]\}$ which is known in the literature as the \emph{upward flow} from $[\mathbf{W}]$. Similarly one can define, for every $j\in J$, 
	$$
	U_j^-=\{[\mathbf{V}]\in \mathcal{M}^{\theta}(\overline{Q},\mathbf{d})| \lim_{t\to \infty}t\cdot[\mathbf{V}]\in F_j\}
	$$
	but contrary to what happens above, the limit on the right hand side of the above equality does not exist for every point in the moduli space. However, Nakajima \cite[Theorem 5.8]{Nakajima} showed that 
	\begin{equation}\label{decompositionNilpotentCone}
	\mathcal{N}=\bigsqcup_{j\in J}U_j^-.
	\end{equation}

	Recall that $\mathcal{N}$ is the nilpotent cone as defined in Section \ref{GITDescriptionNakajimaQuiverVarieties}. 
	Also, we have $U_j^-=\bigsqcup_{[\mathbf{W}]\in F_j}U_{[\mathbf{W}]}^-$ where $U_{[\mathbf{W}]}^-=\{[\mathbf{V}]\in\mathcal{M}^{\theta}(\overline{Q},\mathbf{d}) |  \lim_{t\to \infty}t\cdot [\mathbf{V}]=[\mathbf{W}]\}$ is the so-called \emph{downward flow} from $[\mathbf{W}]$.  Both  $U_j^\pm$ are locally closed subschemes of $\mathcal{M}^{\theta}(\overline{Q},\mathbf{d})$ for which the corresponding limit map $U_j^\pm\to F_j$ defines a Zariski locally trivial affine space fibration. In the next section we will discuss in more detail the dimension of these fibrations. Meanwhile, we will finish this section giving some criteria that describe the upward and downward flows $U_{[\mathbf{W}]}^+$ and $U_{[\mathbf{W}]}^-$. This is a quiver analogue of the Higgs bundle case studied by Hausel and Hitchin \cite[Proposition 3.4 and Proposition 3.11]{HauselHitchin}:
	\begin{proposition}\label{reesConstructionLimitToZero}
		Let $[\mathbf{V}=(V,X,Y)]\in \mathcal{M}^{\theta}(\overline{Q},\mathbf{d})$ then $\lim_{t\to 0}t\cdot[\mathbf{V}]=[\mathbf{V}']$ if and only if 
		\begin{itemize}
			\item there exists a unique filtration 
			{\normalfont
				$$
				0=\mathbf{V}_0\subset \mathbf{V}_1\subset\cdots\subset \mathbf{V}_k=\texttt{Forget}(\mathbf{V})
				$$}
			by subrepresentations such that 
			\item for all $n=1,\ldots,k-1$ and $\alpha\in Q_1$, $$Y_{\overline{\alpha}}(V_{h\alpha,n})\subseteq V_{t\alpha,n+1}$$
			\item and the induced maps 
			{\normalfont$$
				\text{gr}^{n}Y_{\overline{\alpha}}:V_{h\alpha,n}/V_{h\alpha,n-1}\to V_{t\alpha,n+1}/V_{t\alpha,n}
				$$}
			satisfy 
			{\normalfont
				$$
				(\mathbf{V}_1/\mathbf{V}_0\oplus \mathbf{V}_2/\mathbf{V}_1\oplus\cdots \oplus \mathbf{V}_k/\mathbf{V}_{k-1},(\text{gr}^nY_{\overline{\alpha}})_{\substack{\alpha\in Q_1\\ n=1,\ldots,k-1}})\cong \mathbf{V}'.
				$$}
		\end{itemize}
		
	\end{proposition}
	\begin{proof}
		Let us assume there exists a filtration of $\texttt{Forget}(\mathbf{V})$ as in the statement of the proposition for which the corresponding graded representation of $\overline{Q}$ is stable. For every vertex $i\in Q_0$, we consider the $\mathbb{Z}$-graded $\mathbb{C}[x]$-module given by 
		\begin{equation}\label{reesModule}
			\mathcal{V}_i:=\bigoplus_{n\in \mathbb{Z}}x^{-n}V_{i,-n}
		\end{equation}
		for which $x$ acts via the embedding $x^{-n}V_{i,-n}\hookrightarrow x^{-n+1}V_{i,-n+1}$. Here $V_{i,n}=V_{i,k}$ for all $n\geq k$ and $V_{i,n}=0$ for all $n\leq 0$. We claim that $\mathcal{V}_i$ is a free $\mathbb{C}[x]$-module hence giving a vector bundle over $\mathbb{A}^1_{\mathbb{C}}$. First notice that $\mathcal{V}_i$ is finitely generated as it is a submodule of the finitely generated $\mathbb{C}[x]$-module $V_{i,k}[x]=\bigoplus_{n\in\mathbb{Z}}x^{-n}V_{i,k}$. Moreover, $\mathcal{V}_i$ is torsion-free because $V_{i,k}[x]$ is and hence the conclusion. 
		\\
		\\
		For all $\alpha\in Q_1$ we consider the morphism of free $\mathbb{C}[x]$-modules given by 
		$$
		\begin{matrix}
			\mathcal{X}_\alpha:&\mathcal{V}_{t\alpha} & \longrightarrow & \mathcal{V}_{h\alpha} \\
			&\sum_nx^{-n}v_{t\alpha,-n}&\longmapsto&\sum_nx^{-n}X_{\alpha}(v_{t\alpha,-n})
		\end{matrix}
		$$
		and similarly we define 
		$$\begin{matrix}
			\mathcal{Y}_{\overline{\alpha}}:&\mathcal{V}_{h\alpha}&\longrightarrow & \mathcal{V}_{t\alpha} \\
			&\sum_nx^{-n}v_{h\alpha,-n}&\longmapsto&\sum_nx^{-n+1}Y_{\overline{\alpha}}(v_{h\alpha,-n})
		\end{matrix}$$ 
		Therefore, $\mathcal{V}=((\mathcal{V}_i)_{i\in Q_0},(\mathcal{X}_{\alpha},\mathcal{Y}_{\overline{\alpha}})_{\alpha\in Q_1})$ is a family of Nakajima quiver representations over $\mathbb{A}^1_\mathbb{C}$. 
		\\
		\\
		For all closed point $s\in \mathbb{A}^1_\mathbb{C}$, 
		$$
		(\mathcal{V}_i)_{\{s\}}\cong \mathcal{V}_i\otimes_{\mathbb{C}[x]}\mathbb{C}[x]/(x-s)\cong \mathcal{V}_i/(x-s)\mathcal{V}_i, 
		$$
		therefore $\mathcal{V}_{\{s\}}\cong s\cdot \mathbf{V}$ for $s\in \mathbb{C}^*$ and $\mathcal{V}_{\{0\}}\cong \mathbf{V}'$ so our family interpolates between $\mathbf{V}$ and $\mathbf{V}'$. We have a $\mathbb C^*$-equivariant structure on the family which is induced by
		the grading (see, for instance, \cite[Example 39.12.3]{stacksProject}): on
		$\mathbb A_\mathbb{C}^1=\text{Spec}(\mathbb{C}[x])$ we consider the standard action of \(\mathbb C^*\), so that on points $t\cdot s=t s$. Equivalently, on the coordinate ring, the coordinate $x$ has weight
		$-1$. Thus $x^{-n}$ has weight $n$. Therefore, the action
		on the homogeneous summand $x^{-n}V_{i,-n}\subset \mathcal V_i$ is given by
		$t\cdot(x^{-n}v)=t^n x^{-n}v$. This gives a $\mathbb C^*$-linearization of each vector bundle $\mathcal V_i$, covering the standard action on $\mathbb A_{\mathbb{C}}^1$. Equivalently, for every $t\in\mathbb C^*$ and every point
		$s\neq 0\in\mathbb A^1_{\mathbb{C}}$, the induced map on fibres, $\rho_i(t,s):(\mathcal V_i)_{\{s\}}\rightarrow(\mathcal V_i)_{\{t s\}}$, is defined by
		$$
		\rho_i(t,s)(x^{-n}v|_{x=s})
		= t^n x^{-n}v|_{x=t s}=t^n(t s)^{-n}v=s^{-n}v.
		$$
		Thus over \(\mathbb C^*\), the equivariant transport identifies the
		nonzero fibres without changing the underlying vector \(v\in V_i\). The behavior at the special fiber will be discussed below when we treat the converse direction. 
		The arrow maps can easily be seen to be compatible with this equivariant structure (see Equations (\ref{compatibilityEqStructureWithArrows}) and (\ref{compatibilityEqStructureWithOppositeArrows})).
		Since $\mathcal{M}^{\theta}(\overline{Q},\mathbf{d})$ is a fine moduli space, this family induces a $\mathbb{C}^*$-equivariant morphism $\mathbb{A}^{1}_{\mathbb{C}}\to \mathcal{M}^\theta(\overline{Q},\mathbf{d})$ which is equivalent to the desired conclusion. 
		\\
		\\
		Now we show the converse statement. Suppose that $\lim_{t\to 0}t\cdot [\mathbf{V}]=[\mathbf{V}']$ which is the same as having a $\mathbb{C}^*$-equivariant map $\mathbb{A}^1_{\mathbb{C}}\to \mathcal{M}^{\theta}(\overline{Q},\mathbf{d})$. Since $\mathcal{M}^\theta({\overline{Q},\mathbf{d}})$ is a fine moduli space, such a map gives a family $\mathcal{V}=((\mathcal{V}_i)_{i\in Q_0},(\mathcal{X}_{\alpha},\mathcal{Y}_{\overline{\alpha}})_{\alpha\in Q_1})$ over the affine line such that $\mathcal{V}_{\{s\}}\cong s\cdot \mathbf{V}$, $\mathcal{V}_{\{0\}}\cong \mathbf{V}'$  and such that the projection map onto the affine line is $\mathbb{C}^*$-equivariant. The equivariant structure gives, for every $t\in\mathbb C^*$, $s\in \mathbb{C}$ and every vertex \(i\in Q_0\), isomorphisms $\rho_i(t,s):(\mathcal{V}_i)_{\{s\}}\rightarrow (\mathcal{V}_{i})_{\{t s\}}$ as in Definition \ref{equivariantStructures}.
		\\
		\\
		We now define, for every vertex \(i\in Q_0\) and every $n\in \mathbb{Z}_{\geq0}$,
		\begin{equation}\label{piecesFiltrationFromFamily}
			V_{i,n}=\{v\in V_i\cong (\mathcal{V}_i)_{\{1\}}\ |
			\ \lim_{t\to 0}t^n\rho_i(t,1)(v)\text{ exists}\}.
		\end{equation}
		These are vector subspaces of $V_i$ and also can be understood as the set of those vectors $v$ for which the equivariant section $t\overset{\sigma}{\mapsto} t^n\rho_i(t,1)(v)$ extends to a regular section over $\mathbb{A}^1_{\mathbb{C}}$. Moreover, notice that $V_{i,n}\subseteq V_{i,n+1}$ so we have a filtration
		\[
		0=V_{i,0}\subset V_{i,1}\subset\cdots\subset V_{i,k}=V_i.
		\]
		In principle, there could exist non-empty subspaces $V_{i,n}$ for negative integers $n$ but one can twist the equivariant structure by a suitable character for which the filtration, for every vertex $i\in Q_0$, starts at step 0. This is justified by Lemma \ref{twoEquivariantStructuresDifferByALineBundle} and subsequent discussion.  
		\\
		\\
		We claim that these filtrations induce a filtration by subrepresentations of
		$\texttt{Forget}(\mathbf V)$. Let \(\alpha\in Q_1\) and take
		\(v\in V_{t\alpha,n}\). Now observe the following:
		\begin{equation}\label{evaluatingMorphismInArrowsCommutesWithLimits}
			\mathcal{X}_{\alpha}(\lim_{t\to 0}t^n\rho_{t\alpha}(t,1)(v))=\lim_{t\to 0}
			\mathcal{X}_{\alpha}(t^n\rho_{t\alpha}(t,1)(v))=\lim_{t\to 0} t^n\rho_{h\alpha}(t,1)(\mathcal{X}_\alpha(v)).
		\end{equation}
		Equation (\ref{compatibilityEqStructureWithArrows}) gives the right-hand side equality and the left-hand side one is because computing the value of $\mathcal{X}_\alpha$ at $\sigma(0)$ is the same as saying that the section $\mathcal{X}_{\alpha}\circ \sigma$ can be extended to the affine line and then computing its value at $t=0$. Since the limit on the left-hand side of the equation exists, we get that $\mathcal{X}_\alpha(V_{t\alpha,n})\subseteq V_{h\alpha,n}$. A similar strategy, using now Equation (\ref{compatibilityEqStructureWithOppositeArrows}), shows that $\mathcal{Y}_{\overline\alpha}(V_{h\alpha,n})\subseteq V_{t\alpha,n+1}$ for an opposite arrow $\overline{\alpha}$.
		\\
		\\
		Next, we identify the special fibre with the associated graded
		representation. For every vertex $i\in Q_0$ and every $n=1,\ldots,k$, define
		$$
		\begin{matrix}
			\widetilde\gamma_{i,n}:&V_{i,n}&\longrightarrow &(\mathcal V_{i})_{\{0\}} \\
			&v&\longmapsto& \lim_{t\to 0}
			t^n\rho_i(t,1)(v).
		\end{matrix}
		$$
		This map is essentially the evaluation map of the equivariant section $\sigma$ at $t=0$. By construction, $V_{i,n-1}\subseteq \ker(\widetilde\gamma_{i,n})$ and, conversely, if
		$\widetilde\gamma_{i,n}(v)=0$, then
		$\lim_{t\to 0}t^n\rho_i(t,1)(v)=0$
		which means that $\sigma(0)=0$ so that, in particular, $t$ divides $\sigma$. The latter is because $\sigma$ can be written as a tuple of regular functions (polynomials) on $t$ so saying that $\sigma(0)=0$ is the same as saying that $t$ divides each of these functions. Hence $\ker(\widetilde\gamma_{i,n})=V_{i,n-1}$ and $\widetilde\gamma_{i,n}$ induces an injection
		$$
		\gamma_{i,n}:
		V_{i,n}/V_{i,n-1}
		\hookrightarrow
		(\mathcal{V}_i)_{\{0\}}.
		$$
		Now we describe the image of this map. For $\lambda\in \mathbb{C}^*$ we compute 
		$$
		\rho_i(\lambda,0)(\lim_{t\to 0}t^{n}\rho_i(t,1)(v))=\lim_{t\to 0}t^{n}\rho_i(\lambda,t)\rho_i(t,1)(v)=\lambda^{-n}\lim_{t\to 0}(\lambda t)^{n}\rho_i(\lambda t,1)(v) = \lambda^{-n}\lim_{t\to 0} t^n\rho_i(t,1)(v)
		$$
		so the image of $\gamma_{i,n}$ is contained in the weight $-n$ space of the representation $\rho_i(\bullet,0):\mathbb{C}^*\to \text{GL}((\mathcal{V}_i)_{\{0\}})$. The map
		$$
		\oplus_{n}\gamma_{i,n}:V_{i,1}/V_{i,0}\oplus \ldots\oplus V_{i,k}/V_{i,k-1}\to (\mathcal{V}_i)_{\{0\}}
		$$
		is then an injective linear map between vector spaces of the same dimension, hence an isomorphism. 
		It remains to check that these isomorphisms are compatible with the arrow maps. For an ordinary arrow \(\alpha\in Q_1\) and \(v\in V_{t\alpha,n}\), we have
		$$
		\mathcal{X}_{\alpha}(\gamma_{t\alpha,n}([v]))
		=
		\mathcal{X}_{\alpha}\bigl(t^n\rho_{t\alpha}(t,1)(v)\bigr)|_{t=0}
		=
		t^n\rho_{h\alpha}(t,1)(X_\alpha(v))|_{t=0}
		=
		\gamma_{h\alpha,n}([X_\alpha(v)]).
		$$
		Thus the ordinary arrow on the special fibre is identified with the
		degree-zero associated graded map. For an opposite arrow \(\overline\alpha\) and \(v\in V_{h\alpha,n}\),
		$$
		\mathcal{Y}_{\overline\alpha}(\gamma_{h\alpha,n}([v]))
		=
		\mathcal{Y}_{\overline\alpha}\bigl(t^n\rho_{h\alpha}(t,1)(v)\bigr)|_{t=0}
		=
		\left.t^{n+1}\rho_{t\alpha}(t,1)(Y_{\overline\alpha}(v))\right|_{t=0}
		=
		\gamma_{t\alpha,n+1}([Y_{\overline\alpha}(v)]).
		$$
		Hence the opposite arrow is identified with the graded map
		$V_{h\alpha,n}/V_{h\alpha,n-1}\to V_{t\alpha,n+1}/V_{t\alpha,n}.
		$
		\\
		\\
		Finally, we show the uniqueness of the filtration. Our strategy will be the following: we start with a filtration as in the statement of the proposition, then we apply to it the construction in Equation (\ref{reesModule}) and onwards. Then, from the family obtained we construct a filtration by subrepresentations of $\texttt{Forget}(\mathbf{V})$ following Equation (\ref{piecesFiltrationFromFamily}) and subsequent discussion. Finally, we check that the filtration obtained agrees with the one we started with. Let
		$
		0=\mathbf V_0\subset \mathbf V_1\subset\cdots\subset \mathbf V_k
		=\texttt{Forget}(\mathbf V)
		$
		be a filtration by subrepresentations of $Q$ satisfying the conditions of the proposition. Then we obtain a $\mathbb{C}^*$ equivariant family of representations  $\mathcal{V}=((\mathcal{V}_i)_{i\in Q_0},(\mathcal{X}_{\alpha},\mathcal{Y}_{\overline{\alpha}})_{\alpha\in Q_1})$ over $\mathbb{A}^1_{\mathbb{C}}$ whose fiber at $s=1$ is $\mathbf{V}$ and at $s=0$ is the associated graded representation of the filtration. 
		Now we define $\widetilde V_{i,n}$ as in Equation (\ref{piecesFiltrationFromFamily}) and we claim that $\widetilde V_{i,n}=V_{i,n}$ for every vertex $i\in Q_0$ and every $n=1,\ldots,k$. Note that
		$$
		v\in \widetilde V_{i,n}
		\quad\Longleftrightarrow\quad
		x^n v\in \mathcal V_i.
		$$
		So, if \(v\in V_{i,n}\), then clearly $v\in\widetilde V_{i,n}$. Conversely, if $v\in\widetilde V_{i,n}$, then $x^n v\in\mathcal V_i.$ and thus we may write $x^n v=\sum_m f_m(x)x^{-m}v_{-m},$
		with \(f_m(x)\in\mathbb C[x]\) and \(v_m\in V_{i,-m}\). Looking at the
		coefficient of \(x^n\), a term $f_m(x)x^{-m}v_{-m}$ can contribute only if $-m+j=n$
		for some $j\geq0$. Hence $-m\leq n$ and  since the filtration is increasing, $V_{i,-m}\subseteq V_{i,n}$.
		Therefore every contribution to the coefficient $v$ lies in $V_{i,n}$, and so $v\in V_{i,n}$. This concludes the proof.
	\end{proof}

	We conclude by recording the analogue statement of Proposition \ref{reesConstructionLimitToZero} for the downward flow. The proof is obtained by adapting the argument given there.
	\begin{proposition}
		Let $[\mathbf{V}=(V,X,Y)]\in \mathcal{M}^{\theta}(\overline{Q},\mathbf{d})$ then $\lim_{t\to \infty}t\cdot[\mathbf{V}]=[\mathbf{V}']$ if and only if 
		\begin{itemize}
			\item there exists a unique filtration 
			{\normalfont
				$$
				0=\mathbf{V}_0\subset \mathbf{V}_1\subset\cdots\subset \mathbf{V}_k=\texttt{Forget}(\mathbf{V})
				$$}
			by subrepresentations such that 
			\item for all $n=1,\ldots,k-1$ and $\alpha\in Q_1$, 
			$$Y_{\overline{\alpha}}(V_{h\alpha,n+1})\subseteq V_{t\alpha,n}$$
			\item and the induced maps 
			{\normalfont$$
				\text{gr}^{n}_\infty(Y_{\overline{\alpha}}):V_{h\alpha,n}/V_{h\alpha,n-1}\to V_{t\alpha,n-1}/V_{t\alpha,n-2}
				$$}
			satisfy 
			{\normalfont
				$$
				(\mathbf{V}_1/\mathbf{V}_0\oplus \mathbf{V}_2/\mathbf{V}_1\oplus\cdots \oplus \mathbf{V}_k/\mathbf{V}_{k-1},(\text{gr}_\infty(Y_{\overline{\alpha}}))_{\alpha\in Q_1})\cong \mathbf{V}'.
				$$}
		\end{itemize}
		
	\end{proposition}
	
	\subsection{Deformation theory of the $\mathbb{C}^*$-action}
	Let $[\mathbf{V}=(V,X,Y)]\in \mathcal{M}^{\theta}(\overline{Q},\mathbf{d})$  be a fixed point for the $\mathbb{C}^*$-action. Then, we have an induced representation $\mathbb{C}^*\to \text{GL}(T_{[\mathbf{V}]}\mathcal{M}^{\theta}(\overline{Q},\mathbf{d}))$ and hence a decomposition into weight spaces
	$$
	T_{[\mathbf{V}]}\mathcal{M}^\theta(\overline{Q},\mathbf{d})\cong \bigoplus_{k\in \mathbb{Z}} T_{[\mathbf{V}]}\mathcal{M}^\theta(\overline{Q},\mathbf{d})_k
	$$
	where $T_{[\mathbf{V}]}\mathcal{M}^\theta(\overline{Q},\mathbf{d})_k$ is the component where $t\in \mathbb{C}^*$ acts with weight $k$. Following Hausel and Hitchin \cite[Section 2.1]{HauselHitchin} we denote $T_{[\mathbf{V}]}\mathcal{M}^\theta(\overline{Q},\mathbf{d})^+,\ T_{[\mathbf{V}]}\mathcal{M}^\theta(\overline{Q},\mathbf{d})^-$ and $T_{[\mathbf{V}]}\mathcal{M}^\theta(\overline{Q},\mathbf{d})_0$ the positive, negative and zero weight components of this decomposition so that $T_{[\mathbf{V}]}\mathcal{M}^\theta(\overline{Q},\mathbf{d})\cong T_{[\mathbf{V}]}\mathcal{M}^\theta(\overline{Q},\mathbf{d})^+\oplus T_{[\mathbf{V}]}\mathcal{M}^\theta(\overline{Q},\mathbf{d})_0\oplus T_{[\mathbf{V}]}\mathcal{M}^\theta(\overline{Q},\mathbf{d})^-$. The following proposition relates the upward and downward flows with the weight decomposition just introduced:
	\begin{proposition}\label{upAndDownFlowsIsoWithWeightedTangentSpaces}
		\cite[Proposition 2.1]{HauselHitchin}  For a fixed point $[\mathbf{V}]$ the upward and downward flows $U_{[\mathbf{V}]}^\pm$ are $\mathbb{C}^*$-invariant and $U_{[\mathbf{V}]}^\pm\cong T_{[\mathbf{V}]}\mathcal{M}^\theta(\overline{Q},\mathbf{d})^\pm$ as varieties with $\mathbb{C}^*$-action.
	\end{proposition}
	
	Now, using a decomposition of the deformation complex in Equation (\ref{deformationComplexNakajimaQuiverVarieties}) we can compute explicitly the weight spaces in which the tangent space decomposes.  Let $[\mathbf{V}]=[(V,X,Y)]\in \mathcal{M}^\theta(\overline{Q},\mathbf{d})$ be a fixed point for the $\mathbb{C}^*$-action and let $$\mathfrak{gl}_k(\mathbf{d})=\bigoplus\limits_{\substack{i\in Q_0 \\ m-n=k}}\text{Hom}(V_{i,n},V_{i,m}).$$ The \emph{weight $k$ deformation complex} is:
	\begin{equation}\label{weightKDeformationComplexNakajimaQuiverVarieties}
		\overline{\mathcal{H}om}^\bullet_k(\mathbf{V}):\mathfrak{gl}_k(\mathbf{d})\xlongrightarrow{\delta_0^k}\bigoplus_{\substack{\alpha\in Q_1\\ m-n=k}}\text{Hom}(V_{t\alpha,n},V_{h\alpha,m})\oplus\bigoplus_{\substack{\alpha\in Q_1\\ m-n=k+1}}\text{Hom}(V_{t\overline{\alpha},n},V_{h\overline{\alpha},m})\xlongrightarrow{\delta_1^k}\mathfrak{gl}_{k+1}(\mathbf{d})
	\end{equation}
	where $\delta_0$ and $\delta_1$ are defined in the same way as for the complex in Equation (\ref{deformationComplexNakajimaQuiverVarieties}). As a matter of fact 
	\begin{equation}\label{splittingDeformationComplex}
	\overline{\mathcal{H}om}^\bullet(\mathbf{V})=\bigoplus_{k\in\mathbb{Z}}\overline{\mathcal{H}om}^\bullet_k(\mathbf{V}).
	\end{equation}
	A computation similar to that carried out in the proof of Lemma \ref{eulerCharacteristicDeformationComplexNakajimaQuiverVarieties} gives the following interesting identity:
	\begin{lemma}\label{eulerCharacteristicWeightedDeformationComplex}
		For $d^n=(d_{i,n})_{i\in Q_0}$,
		$$\chi(\overline{\mathcal{H}om}_k^\bullet(\mathbf{V}))=\sum_{m-n=k}\langle d^n,d^m\rangle_Q+\sum_{m-n=k+1}\langle d^m,d^n\rangle_Q$$ 
	\end{lemma}
	The reason the weight $k$ deformation complex is of interest to us is the following:
	\begin{proposition}
		Let $[\mathbf{V}]$ be a fixed point for the $\mathbb{C}^*$-action, then $T_{[\mathbf{V}]}\mathcal{M}^\theta(\overline{Q},\mathbf{d})_k\cong H^1(\overline{\mathcal{H}om}_{-k}^\bullet(\mathbf{V}))$
	\end{proposition}
	\begin{proof}
		Let $t\in \mathbb{C}^*$ and $g(t)\in \text{GL}(\mathbf{d})$ the automorphism such that $t\cdot \mathbf{V}=g(t)\cdot\mathbf{V}$. Our problem at hand amounts to determining the weight $k$ piece of the following composition of morphisms of complexes:
		\begin{center}
			\begin{tikzcd}
				\overline{\mathcal{H}om}^\bullet(\mathbf{V}):\arrow[d]&\mathfrak{gl}(\mathbf{d})\arrow[d,"\text{Id}",']\arrow[r,"\delta_0"]&\text{Rep}(\overline{Q},\mathbf{d})\arrow[r,"\delta_1"]\arrow[d,"t\cdot\_",']&\mathfrak{gl}(\mathbf{d})\arrow[d,"t \text{Id}",'] \\
				\overline{\mathcal{H}om}^\bullet(t\cdot\mathbf{V}):\arrow[d]&\mathfrak{gl}(\mathbf{d})\arrow[d,"\text{Ad}(g(t)^{-1})",']\arrow[r,"\delta_0"]&\text{Rep}(\overline{Q},\mathbf{d})\arrow[r,"\delta_1"]\arrow[d,"g(t)^{-1}\cdot\_",']& \mathfrak{gl}(\mathbf{d})\arrow[d," \text{Ad}(g(t)^{-1})",']\\
				\overline{\mathcal{H}om}^\bullet(\mathbf{V}):&\mathfrak{gl}(\mathbf{d})\arrow[r,"\delta_0"]&\text{Rep}(\overline{Q},\mathbf{d})\arrow[r,"\delta_1"]&\mathfrak{gl}(\mathbf{d}).
			\end{tikzcd}
		\end{center}
		Now, looking at the restriction to the degree $k$ pieces of the morphisms in the first column, we get:
		$$
		\text{Ad}(g(t)^{-1})\circ\text{Id}(f_{i,n,m})_{\substack{i\in Q_0 \\ m-n=k}}=(g(t)_i^{-1}f_{i,n,m}g(t)_i)_{\substack{i\in Q_0 \\ m-n=k}}=(t^{n-m}f_{i,n,m})_{\substack{i\in Q_0\\m-n=k}}=t^{-k}(f_{i,n,m})_{\substack{i\in Q_0 \\ m-n=k}}.
		$$
		This is because for all $i\in Q_0$, $g(t)_i$ acts with weight $n$ on $V_{i,n}$ (see Equation (\ref{oneParameterSubgroupTorusaction}) and subsequent discussion). A similar computation gives the analogous result for the third column. Likewise, for the middle column:
		\begin{align*}
			&g(t)^{-1}\cdot(t\cdot ((X_{\alpha,n,m})_{\substack{\alpha\in Q_1\\m-n=k }},(Y_{\overline{\alpha},n,m})_{\substack{\alpha\in Q_1\\m-n=k+1 }}))\\&=((g(t)_{h\alpha}^{-1}X_{\alpha,n,m}g(t)_{t\alpha})_{\substack{\alpha\in Q_1\\m-n=k }},(g(t)_{t\alpha}^{-1}tY_{\overline{\alpha},n,m}g(t)_{h\alpha})_{\substack{\alpha\in Q_1\\m-n=k+1}})\\&=((t^{n-m}X_{\alpha,n,m})_{\substack{\alpha\in Q_1\\ m-n=k}},(t^{n-m+1}Y_{\overline{\alpha},n,m})_{\substack{\alpha\in Q_1\\ m-n=k+1}})=t^{-k}((X_{\alpha,n,m})_{\substack{\alpha\in Q_1\\m-n=k }},(Y_{\overline{\alpha},n,m})_{\substack{\alpha\in Q_1\\m-n=k+1 }}).
		\end{align*}
		Our preceding calculations imply that the induced action on
		$$
		H^1(
		\overline{\mathcal{H}om}^{\bullet}_k(\mathbf V))
		=
		\frac{\ker(\delta_1^k)}
		{\operatorname{Im}(\delta_0^k)}
		$$
		is also multiplication by $t^{-k}$.
		On the other hand, recall that the cohomology functor is additive
		\cite[Exercise 1.2.1]{WeibelHomologicalAlgebra} so Proposition \ref{cohomologiesDeformationComplexNakajimaQuiverVarieties} and Equation
		(\ref{splittingDeformationComplex}) give
		$$
		T_{[\mathbf{V}]}\mathcal{M}^{\theta}(\overline{Q},\mathbf{d})\cong H^1(\overline{\mathcal{H}om}^{\bullet}(\mathbf V))
		\cong
		\bigoplus_{k\in\mathbb Z}
		H^1(
		\overline{\mathcal{H}om}^{\bullet}_k(\mathbf V)).
		$$
		and hence the conclusion. 
	\end{proof}
	We end the section with the following identity: 
	\begin{proposition}\label{motivicBialynickiBirulaDecomposition}
	In $K_0(\operatorname{Var}_\mathbb{C})$:
	\begin{equation}
	[\mathcal{M}^{\theta}(\overline{Q},\mathbf{d})]=\sum_{\mathbf{d}^\ell\in \mathscr{D}_\theta(\mathbf{d})}[\mathcal{M}^{\theta^\ell}(Q^\ell,\mathbf{d}^\ell,\mathcal{R}^\ell)]\cdot \mathbb{L}^{N_{\mathbf{d}^\ell}},
	\end{equation}
	where $	\mathscr{D}_\theta(\mathbf{d})$ is as in Theorem \ref{quiverChainsClosedImmersion} and $N_{\mathbf{d}^\ell}$ is a nonnegative integer.
	\end{proposition}
	\begin{proof}
	This is a direct consequence of the Bia\l ynicki-Birula decomposition of the Nakajima quiver variety $\mathcal{M}^{\theta}(\overline{Q},\mathbf{d})$ and Section \ref{GrothendieckRingOfVarieties}. 

	\end{proof}
	\begin{remark}
	 We can compute explicitly $N_{\mathbf{d}^\ell}$. Indeed, because of Proposition \ref{upAndDownFlowsIsoWithWeightedTangentSpaces}, 
$$
N_{\mathbf{d}^\ell}=\dim(T_{[\mathbf{V}]}\mathcal{M}^{\theta}(\overline{Q},\mathbf{d})^+)=\sum_{k> 0}\dim(H^1(\overline{\mathcal{H}om}_{-k}^\bullet(\mathbf{V})))
$$
where $[\mathbf{V}]\in \mathcal{M}^{\theta}(\overline{Q},\mathbf{d})$ is a fixed point for the $\mathbb{C}^*$-action lying in the component $\mathcal{M}^{\theta^\ell}(Q^\ell,\mathbf{d}^\ell,\mathcal{R}^\ell)$. 
Because of the stability condition 
$$
H^0(\overline{\mathcal{H}om}_{k}^\bullet(\mathbf{V}))=\begin{cases}
	\mathbb{C} & \text{ if }k=0,\\
	0 & \text{otherwise}. 
\end{cases}
$$
Finally, one can use the trace pairing in Equation (\ref{tracePairing}) to conclude that 
$$
\overline{\mathcal{H}om}_{k}^\bullet(\mathbf{V})^\vee = \overline{\mathcal{H}om}_{-k-1}^\bullet(\mathbf{V})
$$
and therefore 
\begin{equation}\label{dualityWeightedCohomologyGroups}
H^j(\overline{\mathcal{H}om}_{k}^\bullet(\mathbf{V}))^\vee\cong H^{2-j}(\overline{\mathcal{H}om}_{-k-1}^\bullet(\mathbf{V})), \ j=0,1,2.
\end{equation}
Taking $j=1$ and using Lemma \ref{eulerCharacteristicWeightedDeformationComplex} we get
$$
N_{\mathbf{d}^\ell}=1-\sum_{k<0}\left(\sum_{m-n=k}\langle d^n,d^m\rangle_Q+\sum_{m-n=k+1}\langle d^m,d^n\rangle_Q\right).
$$
	\end{remark}

\begin{remark}
There is yet another way of computing the value of $N_{\mathbf{d}^\ell}$ using holomorphic symplectic form on $T^*\operatorname{Rep}(Q,\mathbf{d})\cong \operatorname{Rep}(\overline{Q},\mathbf{d})$ is given by 
$$
\omega_\mathbb{C}((X,Y),(X',Y'))=\sum_{\alpha\in Q_1}\operatorname{tr}(X_{\alpha}Y'_{\overline{\alpha}}-X'_{\alpha}Y_{\overline{\alpha}}).
$$
This descends to the moduli space $\mathcal{M}^{\theta\text{-}s}(\overline{Q},\mathbf{d})$ and an easy computation shows that $\mathbb{C}^*$ acts with weight one on it. By this we mean that 
\begin{equation}\label{scalingSymplecticForm}
t^*(\omega_{\mathbb{C}})=t\omega_{\mathbb{C}}.
\end{equation}
The condition in Equation (\ref{scalingSymplecticForm}) means, for $[V]$ a fixed point for the $\mathbb{C}^*$-action, that $\omega_{\mathbb{C}}(v,w)=0$ for $v\in T_{[\mathbf{V}]}\mathcal{M}^\theta(\overline{Q},\mathbf{d})_k$ and $w\in T_{[\mathbf{V}]}\mathcal{M}^\theta(\overline{Q},\mathbf{d})_l$ whenever $k+l\neq1$. If $k+l=1$ we recover the duality in Equation (\ref{dualityWeightedCohomologyGroups}) in degree one:
$$
T_{[\mathbf{V}]}\mathcal{M}^\theta(\overline{Q},\mathbf{d})_k^\vee\cong T_{[\mathbf{V}]}\mathcal{M}^\theta(\overline{Q},\mathbf{d})_{1-k}.
$$
Then $(\bigoplus_{k>1}T_{[\mathbf{V}]}\mathcal{M}^\theta(\overline{Q},\mathbf{d})_k)^\vee\cong T_{[\mathbf{V}]}\mathcal{M}^\theta(\overline{Q},\mathbf{d})^-$ and
\begin{align*}
\dim(\mathcal{M}^\theta(\overline{Q},\mathbf{d}))&=\dim(T_{[\mathbf{V}]}\mathcal{M}^\theta(\overline{Q},\mathbf{d})^+)+\dim(T_{[\mathbf{V}]}\mathcal{M}^\theta(\overline{Q},\mathbf{d})_0)+\dim(T_{[\mathbf{V}]}\mathcal{M}^\theta(\overline{Q},\mathbf{d})^-)\\
&=2\dim(T_{[\mathbf{V}]}\mathcal{M}^\theta(\overline{Q},\mathbf{d})^+)
\end{align*}
Therefore, the variety $U_{[\mathbf{V}]}^+$ is Lagrangian and 
$$
N_{\mathbf{d}^\ell}=\frac{1}{2}\dim(\mathcal{M}^\theta(\overline{Q},\mathbf{d}))=1-\langle\mathbf{d},\mathbf{d}\rangle_Q
$$
showing that the dimensions of the Zariski-locally trivial fibrations in which the moduli space $\mathcal{M}^\theta(\overline{Q},\mathbf{d})$ stratifies do not depend at all on the moduli space of quiver chains serving as base to these.  Using a similar technique and using the description in Equation \ref{decompositionNilpotentCone} we can also show that the nilpotent cone $\mathcal{N}$ is Lagrangian subvariety.
\end{remark}

	\section{The case of the star-shaped quiver}\label{theCaseOfTheStarShapedQuiver}
	In this section we deal with the case of the so-called \emph{star-shaped quiver:}
	
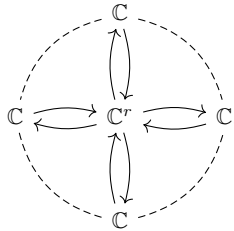
\vspace{-1mm}	
	\begin{center}
		\begin{tikzcd}[scale cd=0.85,row sep=large, column sep=normal]
			& \mathbb{C}\arrow[dl,dashed,no head,bend right=30]\arrow[dr,dashed,no head,bend left=30]\arrow[d,bend left = 15]&  \\
			\mathbb{C}\arrow[r,bend left=15]& \mathbb{C}^r\arrow[u,bend left = 15]\arrow[d,bend left=15]\arrow[l,bend left=15] \arrow[r,bend left=15]&\mathbb{C}. \arrow[l,bend left=15]  \\
			& \mathbb{C}\arrow[u, bend left=15]\arrow[ur,dashed,no head,bend right=30]\arrow[ul,dashed,no head,bend left=30]& 
		\end{tikzcd}
	\end{center}
	with $n$ branches. We will assume that $n> r$.
	The stability parameter that we will be considering is $\theta=(\theta_0:=\sum_{i=1}^n\theta_i,-r\theta_1,\ldots,-r\theta_n)$ where $\theta_i$ is a positive integer. Following Theorem \ref{characterizationFixedPoints} and its subsequent remark, the fixed points for the $\mathbb{C}^*$-action on the moduli space are given by quiver chains which in this case are representations
	
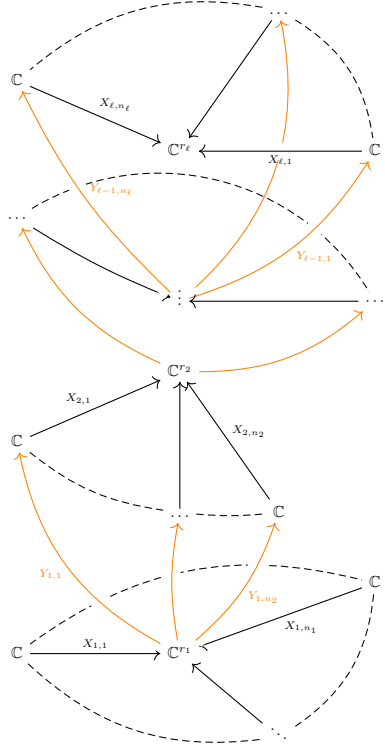
\begin{center}
		\begin{tikzcd}[scale cd=0.6] 
			&&&\ldots\arrow[to=3-3]\arrow[to=3-5,dashed,no head,bend left=25]\arrow[to=2-1,dashed,no head,bend right=25]\arrow[orange,from=5-3,bend right=30,crossing over]&\\
			\mathbb{C}\arrow[to=3-3,"X_{\ell,n_\ell}"]&&&&\\
			&&\mathbb{C}^{r_\ell}&&\mathbb{C}\arrow[to=3-3,crossing over,"X_{\ell,1}"]\\
			\cdots \arrow[to=5-5,dashed,no head,bend left=45]\arrow[to=5-3,bend right=5]&&&&\\
			&&\vdots \arrow[orange,to=3-5,bend right=20,crossing over,"Y_{\ell-1,1}",']\arrow[orange,to=2-1,bend left=10, crossing over,"Y_{\ell-1,n_\ell}",']&& \cdots\arrow[to=5-3] \\
			&&\mathbb{C}^{r_2}\arrow[orange,to=4-1,bend left = 20]\arrow[orange,to=5-5,bend right = 20]&&\\
			\mathbb{C}\arrow[to=8-3,dashed,no head,bend right=15]\arrow[to=6-3,"X_{2,1}"]&&&& \\
			&&\ldots\arrow[to=6-3]\arrow[to=8-4,dashed,no head,bend right=5]&\mathbb{C}\arrow[to=6-3,"X_{2,n_2}",']&\\
			&&&&\mathbb{C}\arrow[to=10-3,"X_{1,n_1}"]\\
			\mathbb{C}\arrow[to=11-4,dashed,no head,bend right=25]\arrow[to=10-3,"X_{1,1}"]\arrow[to=9-5,dashed,no head,bend left=25]&&\mathbb{C}^{r_1}\arrow[orange,to=8-3,bend left=10,crossing over]\arrow[orange,to=8-4,bend right=15,crossing over,,"Y_{1,n_2}",']\arrow[orange,to=7-1,bend left=25, crossing over,"Y_{1,1}"]&& \\
			&&&\ddots\arrow[to=10-3]\arrow[to=9-5,dashed,no head,bend right=25]& 
		\end{tikzcd}
	\end{center}
	for which $\ell$ is the number of ``levels'' of the chain, $r_1+\ldots+r_\ell=r$ and $n_1+\ldots+n_\ell=n$ with $n_k=|I_k|$ the number of branches in level $k$ of the chain and $I_1\sqcup\cdots \sqcup I_\ell=[n]$. Moreover, recall that the fixed points are representations with relations which in this case are given by
	\begin{equation}\label{relationsFixedPointsStarShapedQuiver}
		\sum_{j\in I_{k}}X_{k,j}Y_{k-1,j}=0
	\end{equation}
	for all $k=2,\ldots,\ell$ and that we will measure stability of these chains with respect to the stability parameter $\theta^\ell$ as defined in Proposition \ref{semistableRepresentationIffSemistableChain}. 
	\begin{remark}
	Recall that for $k=1,\ldots,\ell$, we have $(r_k,n_k)\neq(0,0)$, this is just Remark \ref{noEmptyLevels} rewritten for the star-shaped situation. This will be important for the next proposition. 
	\end{remark}
	\begin{remark}\label{anotherPointOfViewStarShapedChains}
		In this context, a quiver chain can also be seen as a tuple $(A_1,B_1,\ldots ,B_{\ell-1},A_\ell)$ where $A_k:\mathbb{C}^{n_k}\to \mathbb{C}^{r_k}$, $B_{k}:\mathbb{C}^{r_{k}}\to \mathbb{C}^{n_{k+1}}$ are linear maps such that $A_{k+1}B_k=0$ for all $k=1,\ldots,\ell-1$. Notice that the latter relation comes from the complex moment map condition in Equation (\ref{relationsFixedPointsStarShapedQuiver}).
	\end{remark}
	Using the real moment map (see Equation (\ref{realMomentMapDef})) we can be more precise describing these fixed points:
	\begin{proposition}\label{propertiesFixedPointsStarShapedQuiver}
		Under the previous notations and hypotheses:
		\begin{enumerate}
			\item $r_k>0$ for all $k=1,\ldots,\ell$.
			\item $n_k>0$ for all $k=1,\ldots,\ell$.
			\item $X_{k,j}\neq 0$  for all $k=1,\ldots,\ell$ and $j\in I_k$.
			\item $\sum_{j\in I_{k+1}}\|Y_{k,j}\|^2=r\sum_{j\in I_1\sqcup\cdots\sqcup I_k}\theta_j-\theta_0\sum_{j=1}^{k}r_j$ for all $k=1,\ldots,\ell$. Moreover, for all $k=1,\ldots,\ell-1$, the norm in the left-hand side is strictly positive. 
		\end{enumerate}
	\end{proposition}
	\begin{proof}
		We start with the first item and we suppose that $r_k=0$. For $k=1,\ldots,\ell$ and $j\in I_k$, the real moment map equation in the branches of level $k$ reads $\|Y_{k-1,j}\|^2-\|X_{k,j}\|^2=\|Y_{k-1,j}\|^2=\theta_{k,j}=-r\theta_j< 0$ which is a contradiction. For (2) assume that $n_k=0$ then the real moment map condition in the central vertex of level $k$ implies that $-\sum_{j\in I_{k+1}}Y_{k,j}^*Y_{k,j}=\theta_{0,k}\text{Id}_{r_k}=\theta_0\text{Id}_{r_k}$. Taking the trace on both sides of the equation we get that $-\sum_{j\in I_{k+1}}\|Y_{k,j}\|^2=r_k\theta_0>0$ yielding a contradiction. For the third item of the proposition we argue similarly as with (1). For item (4) we write the moment map equation for the center vertex in level $k$ of the fixed point: $\sum_{j\in I_k}X_{k,j}X_{k,j}^*-\sum_{j\in I_{k+1}}Y_{k,j}^*Y_{k,j}=\theta_{0,k}\text{Id}_{r_k}$. After taking traces and reorganizing the terms we get $$\sum_{j\in I_{k+1}}\|Y_{k,j}\|^2=\sum_{j\in I_k}\|X_{k,j}\|^2-\theta_{0,k}r_k.$$ Using the moment map for each one of the branches in the level $k$ we obtain that $\sum_{j\in I_k}\|X_{k,j}\|^2=\sum_{j\in I_k}\|Y_{k-1,j}\|^2-\theta_{k,j}$ so that 
		$$
		\sum_{j\in I_{k+1}}\|Y_{k,j}\|^2=\sum_{j\in I_k}\|Y_{k-1,j}\|^2-\sum_{j\in I_k}\theta_{k,j}-\theta_{0,k}r_k.
		$$
		If we keep going recursively we thus obtain the desired formula
		$$
		\sum_{j\in I_{k+1}}\|Y_{k,j}\|^2=r\sum_{j\in I_1\sqcup\cdots\sqcup I_k}\theta_j-\theta_0\sum_{j=1}^{k}r_j.
		$$
		Finally, the fact that the norm is strictly positive for $k=1,\ldots,\ell-1$, follows from Remark \ref{noEmptyLevels}.
	\end{proof}
	\begin{remark}
		Equality (4) is particularly useful because it gives bounds on the dimensions of the vector spaces at the central vertices, as well as on the number of branches at each level. In particular, for all $k=1,\ldots,\ell-1$ we have that 
		$$
		\frac{n_1+\ldots+n_k}{r_1+\ldots+r_k}> \frac{n}{rL_k}
		$$
		where $L_k=\max\limits_{j\in I_1\sqcup\ldots\sqcup I_{k}}\theta_j$.
	\end{remark}
	\subsection{Classification of all type $(1,\ldots,1)$ fixed points/quiver chains}
	The goal of this section is to classify all type $(1,\ldots,1)$ quiver chains, that is, those quiver chains for which $\ell=r$, or equivalently, representations of $Q^\ell$ of dimension vector $\mathbf{d}^\ell$ such that the rank at the center vertex in each level is one. Our first task will be to describe the moduli space $\mathcal{M}^{\theta^\ell}(Q^\ell,\mathbf{d}^\ell)$.  This is given by (see, for instance, \cite[p.597]{ReinekeSurvey}):
	\begin{equation}\label{definitionModuliSpace}
		\mathcal{M}^{\theta^\ell}(Q^\ell,\mathbf{d}^\ell)=\text{Proj}(\bigoplus_{\eta\geq 0}\mathcal{O}_{V}^{\chi^\eta_{\theta^\ell}})
	\end{equation}
	for $V=\text{Rep}(Q^\ell,\mathbf{d}^\ell)$ the affine space parameterising all $\mathbf{d}^\ell$-dimensional representations,
	$$
	\mathcal{O}_V^{\chi^\eta_{\theta^\ell}}=\{f\in \mathcal{O}_V\cong \mathbb{C}[X_{k,j},Y_{k-1,j}]_{\substack{k=1,\ldots,\ell\\j\in I_k}}|f(t\cdot (X,Y))=\chi^\eta_{\theta^\ell}(t)f(X,Y)\ \forall\ (X,Y)\in V,\ t\in \text{GL}(\mathbf{d}^\ell)\}
	$$
	and $\chi_{\theta^\ell}:\text{GL}(\mathbf{d}^\ell)\to \mathbb{G}_m$ the character given by 
	$$
	\chi_{\theta^\ell}(t)=\prod_{k=1}^{\ell}\frac{t_{k,0}^{\theta_0}}{\prod_{j\in I_k}t_{k,j}^{r\theta_j}}.
	$$
	Now, notice that the graded ring in Equation (\ref{definitionModuliSpace}) is generated by the monomials
	$$
	x_{s}:=\prod_{j\in I_1}X_{1,j}^{r\theta_j}\prod_{j\in I_2}Y_{1,j}^{s_{2,j}}X_{2,j}^{r\theta_j+s_{2,j}}\cdots\prod_{j\in I_\ell}Y_{\ell-1,j}^{s_{\ell,j}}X_{\ell,j}^{r\theta_j+s_{\ell,j}}
	$$
	where $s=(s_{k,j})_{\substack{k=2,\ldots,\ell \\ j\in I_k}}\in \mathbb{N}^{n_2+\ldots+n_\ell}$ is such that 
	\begin{equation}\label{equationExponentsFixedPoint}
		\sum_{j\in I_k}s_{k,j}=\sum_{j\in I_1\sqcup\ldots\sqcup I_{k-1}}r\theta_j-(k-1)\theta_0:=\alpha_k
	\end{equation}
	for all  $k=2,\ldots,\ell$. A priori, it is not clear if this set of equalities can hold. However, equality (4) in Proposition \ref{propertiesFixedPointsStarShapedQuiver} implies that the right hand side in Equation (\ref{equationExponentsFixedPoint}) is positive for all $k=2,\ldots,\ell$ and hence solutions to it do exist. As a matter of fact, for a fixed $k$, there will be 
	$
	N_k:=$ $ {\alpha_k+n_k-1}\choose{n_k-1}$
	solutions, which is the number of ways we can write the right hand side in Equation (\ref{equationExponentsFixedPoint}) as a sum of $n_k$ natural numbers. We denote the set of solutions $s\in \mathbb{N}^{n-n_1}$ to the equalities in Equation (\ref{equationExponentsFixedPoint}) as $\mathscr{A}$.
	\\
	\\
	The key observation is that $\mathcal{M}^{\theta^\ell}(Q^\ell,\mathbf{d}^\ell)$ is the toric variety associated to the monomial map determined by the finite set of lattice points $\mathscr{A}\subseteq \mathbb{Z}^{n-n_1}$ and given by:
	$$
	\begin{matrix}
		\Phi_\mathscr{A}:(\mathbb{C}^*)^{n_1+2n_2+\ldots+2n_{\ell-1}+2n_\ell}&\longrightarrow&\mathbb{P}_\mathbb{C}^{|\mathscr{A}|-1} \\ (X_{k,j},Y_{k-1,j}) & \longmapsto & (x_s)_{s\in \mathscr{A}}.
	\end{matrix} 
	$$ 
	The general theory of toric varieties tells us then that the toric ideal defining $\mathcal{M}^{\theta^\ell}(Q^\ell,\mathbf{d}^\ell)$ is
	$
	\mathcal{I}_{\mathscr{A}}=\langle x^\alpha-x^\beta\ | \ \alpha,\beta\in\mathbb{N}^{|\mathscr{A}|}\text{ and } \alpha-\beta\in L\rangle
	$
	where $L$ fits in the exact sequence
	\vspace{-1mm}
	\begin{center}
		\begin{tikzcd}[column sep =normal, row sep =normal]
			0\arrow[r]&L\arrow[r,""]& \mathbb{Z}^{|\mathscr{A}|} \arrow[r] & \mathbb{Z}^{n-n_1}.
		\end{tikzcd}
	\end{center}
	The ideal $\mathcal{I}_{\mathscr{A}}$ is essentially encoding the relations between the variables $x_s$ for all $s\in\mathscr{A}$. It is, however, difficult to give a more concrete description of $\mathcal{M}^{\theta^\ell}(Q^\ell,\mathbf{d}^\ell)$ from this ideal so we instead describe the lattice polytope associated to this toric variety. This is 
	$$
	P=\underbrace{\{(r\theta_j)_{j\in I_1}\}}_{P_1}\times\underbrace{(\iota_2(\alpha_2\Delta_{{n_2-1}})+v_2)}_{P_2}\times\cdots\times \underbrace{(\iota_\ell(\alpha_\ell\Delta_{{n_\ell-1}})+v_\ell)}_{P_\ell}
	$$
	where $v_k=(0,(r\theta_j)_{j\in I_k})$, $\Delta_{m}=\{(a_0,a_1,a_2,\ldots,a_{m})\in \mathbb{R}^{m+1}|\sum_{j=0}^{m}a_j=1,\ a_j\geq0\}$ is the standard $m$-simplex and $\iota_k:\mathbb{R}^{n_k}\hookrightarrow\mathbb{R}^{n_k}\times \mathbb{R}^{n_k}$ is the diagonal map. Now, we use the fact that the toric variety associated to a product of very ample polytopes is the product of the projective toric varieties associated to each factor to conclude that $\mathcal{M}^{\theta^\ell}(Q^\ell,\mathbf{d}^\ell)\cong X_{P_1}\times \ldots \times X_{P_\ell}$. Here $X_{P_k}$ is the toric variety associated to the polytope $P_k$. In particular:
	$$
	X_{P_k}\cong \begin{cases}
		\{\text{Point}\} & \text{ for }k=1,\\
		\mathbb{P}_{\mathbb{C}}^{n_k-1} \text{embedded in }\mathbb{P}_{\mathbb{C}}^{N_k-1} \text{ via the line bundle }\mathcal{O}_{\mathbb{P}_\mathbb{C}^{n_k-1}}(\alpha_k) & \text{ for }k=2,\ldots,\ell.
	\end{cases}
	$$
	So $\mathcal{M}^{\theta^\ell}(Q^\ell,\mathbf{d}^\ell)$ is the projective variety inside $\mathbb{P}_{\mathbb{C}}^{|\mathscr{A}|-1}$ given by the Segre embedding of $\ell-1$ projective spaces, each embedded in higher dimensional projective spaces via the Veronese embedding of degree $\alpha_k$. 
	\\
	\\
	Finally, we need to cut $\mathcal{M}^{\theta^\ell}(Q^\ell,\mathbf{d}^\ell)$ by the hypersurfaces defined by the relations described in Equation (\ref{relationsFixedPointsStarShapedQuiver}) to obtain $\mathcal{M}^{\theta^\ell}(Q^\ell,\mathbf{d}^\ell,\mathcal{R}^\ell)$ which is the desired moduli space of type $(1,\ldots,1)$ quiver chains. For each $k=2,\ldots,\ell$, the relation in Equation (\ref{relationsFixedPointsStarShapedQuiver}) describes a linear hyperplane in the corresponding projective space $\mathbb{P}_\mathbb{C}^{n_k-1}$ yielding the following theorem:
	\begin{theorem}\label{classificationType1...1QuiverChains}
	Suppose $\ell=r$, then 
	$$
	\mathcal{M}^{\theta^\ell}(Q^\ell,\mathbf{d}^\ell,\mathcal{R}^\ell)\cong 	\mathbb{P}^{n_2-2}_{\mathbb{C}}\times \cdots \times \mathbb{P}^{n_\ell-2}_{\mathbb{C}}.
	$$	
	In particular, if $n_k<2$ for some $k=2,\ldots,\ell$ then $	\mathcal{M}^{\theta^\ell}(Q^\ell,\mathbf{d}^\ell,\mathcal{R}^\ell)=\varnothing$. 
	\end{theorem}
	
	\subsection{Motivic class of higher rank fixed point types}\label{motivicClassHigherRankStarShapedQuiver}
	An explicit description of the fixed points for higher rank quiver chains, as in the previous section, is no longer possible. We are able, however, to compute their classes in the Grothendieck ring of varieties. Some useful facts regarding the construction of this ring and the motivic class of some varieties necessary for the computation are recollected in Section \ref{GrothendieckRingOfVarieties}.
	\\
	\\
	To compute the motivic class of the locus of $\theta$-stable quiver chains we use the following motivic analogue of a result due to Fei \cite[Lemma 3.5]{FeiQuiversWithRelations} which in turn is based on Reineke's seminal work \cite{ReinekePointCounting} in which he uses point counting techniques to compute the Betti numbers of moduli spaces of quiver representations without relations.  Following the conventions from Sections \ref{quiverFlagVarieties} and \ref{GrothendieckRingOfVarieties} we have:
	\begin{lemma}\label{MotivicFeiIdentity}
		In $K_0^{\operatorname{loc}}(\operatorname{Var}_{\mathbb{C}})$, 
		$$
		[\mathcal{M}^{\theta^\ell}(Q^\ell,\mathbf{d}^\ell,\mathcal{R}^\ell)]=\frac{\mathbb{L}-1}{[\operatorname{GL}(\mathbf{d}^\ell)]}\sum_{*}(-1)^{s-1}[{\normalfont\text{Frep}_{\mathbf{d}^\ell_1\cdots\mathbf{d}^\ell_s}}(Q^\ell,\mathcal{R}^\ell)]
		$$
		where the sum runs over all decompositions $\mathbf{d}^\ell_1+\cdots+\mathbf{d}^\ell_s=\mathbf{d}^\ell$ of $\mathbf{d}^\ell$ into non-zero dimension vectors such that $\mu_{\theta^\ell}(\sum_{j=1}^k\mathbf{d}^\ell_j)>0$ for all $k=1,\ldots,s-1$.
	\end{lemma}
	\begin{proof}
	The derivation of this identity uses motivic Hall algebras and is carried out in the Appendix, Section \ref{motivicHallAlgebrasSection}. 
	\end{proof}
	\begin{remark}
	One can determine all the decompositions $\mathbf{d}_1^\ell+\cdots+\mathbf{d}_s^\ell=\mathbf{d}^\ell$ satisfying the conditions in the theorem computationally as follows: consider a directed graph with set of vertices all the dimension vectors $\mathbf{e}<\mathbf{d}^\ell$ (see Section \ref{hyperKReduction} for an explanation of this notation) such that $\mu_{\theta^\ell}(\mathbf{e})>0$ plus the zero dimension vector and $\mathbf{d}^\ell$. In this graph there will be an arrow from $\mathbf{e}_1$ to $\mathbf{e}_2$ if $\mathbf{e}_1<\mathbf{e}_2$. Finding all the desired decompositions is then equivalent to compute all the paths in the graph from the zero vertex to the vertex corresponding to $\mathbf{d}^\ell$. 
	\end{remark}
	So our task will be to compute the motivic class of the quiver flag varieties $\text{Frep}_{\mathbf{d}^\ell_1\cdots\mathbf{d}^\ell_s}(Q^\ell,\mathcal{R}^\ell)$. 
	Fix an admissible decomposition $\mathbf{d}_1^\ell+\cdots+\mathbf{d}_s^\ell=\mathbf{d}^\ell$ where $\mathbf{d}_j^\ell=((\varepsilon_{b,j,k})_{b\in I_k},\beta_{j,k})_{k=1}^\ell$, $\varepsilon_{b,j,k}\in\{0,1\} $ depending on whether the branch $b\in I_k$ is added at that step $j$ of the filtration or not (so that in particular $\sum_{j=1}^{s}\varepsilon_{b,j,k}=1$). Notice also that $\sum_{b\in I_k}\varepsilon_{b,j,k}=\alpha_{j,k}$, $\sum_{j=1}^s\alpha_{j,k}=n_k$ and $\sum_{j=1}^s\beta_{j,k}=r_k$. Following the notation from Section \ref{quiverFlagVarieties} we denote
	$$
	\operatorname{Fl}_{\mathbf{d}^\ell_\bullet }=
	\prod_{k=1}^{\ell}(
	\{\operatorname{Coord}_{\bullet,k}^{\varepsilon}\}
	\times
	\operatorname{Fl}_{\beta_{\bullet,k}}(\mathbb C^{r_k})).
	$$
	Here $\operatorname{Coord}_{\bullet,k}^{\varepsilon}$ is a fixed coordinate flag given by the data
	$$
	\operatorname{Coord}_{p,k}^{\varepsilon}=\langle e_b \ | \ \sum_{j=1}^p\varepsilon_{b,j,k}=1,\ b\in I_k \rangle,\ p=1,\ldots,s, 
	$$
	and $\operatorname{Fl}_{\beta_{\bullet,k}}(\mathbb C^{r_k})$ denotes the flag variety of type $\beta_{\bullet,k}$ (see Section \ref{flagVarietiesSection}).  
	\\
	\\
	In Section \ref{quiverFlagVarieties} we claimed that the map $\text{Frep}_{\mathbf{d}^\ell_1\cdots\mathbf{d}^\ell_s}(Q^\ell,\mathcal{R}^\ell)\overset{p_2}{\longrightarrow}\operatorname{Fl}_{\mathbf{d}^\ell_\bullet }$ is a Zariski-locally trivial fibration. We will use this explicitly for the star-shaped quiver case we are looking at and this will lead the way to compute the desired motivic class. There are some points, however, that need to be justified. The first one is why in the base $\operatorname{Fl}_{\mathbf{d}^\ell_\bullet }$ we have coordinate flags. In Remark \ref{anotherPointOfViewStarShapedChains} we pointed out that one can see a quiver chain as a tuple of matrices whose product is zero. This point of view is helpful for the motivic computation that we want to carry out. But there is one caveat,  the matrix description identifies the direct sum of the one-dimensional branch spaces at level $k$ with $\mathbb C^{n_k}$, but the branches remain distinct vertices. A subrepresentation therefore selects either \(0\) or the entire space at each branch vertex. Under this identification, selecting a branch  is the same as choosing a coordinate subspace of $\mathbb C^{n_k}$, and a filtration by subrepresentations must give a coordinate flag branchwise. Arbitrary flags, such as one containing a diagonal subspace cannot correspond to branchwise subrepresentations. In contrast, the central space $\mathbb C^{r_k}$ belongs to a single vertex and admits arbitrary flags.
	\\
	\\
	We denote by
	$$
	0\subseteq \mathcal F^{n_k}_1\subseteq\cdots\subseteq
	\mathcal F^{n_k}_s=
	\mathbb C^{n_k}\times\operatorname{Fl}_{\mathbf{d}^\ell_\bullet },\ 
	0\subseteq \mathcal F^{r_k}_1\subseteq\cdots\subseteq
	\mathcal F^{r_k}_s=
	\mathbb C^{r_k}\times\operatorname{Fl}_{\mathbf{d}^\ell_\bullet }
	$$
	the pullbacks to $\operatorname{Fl}_{\mathbf{d}^\ell_\bullet }$ of the tautological flags (see Equation (\ref{tautologicalFlagOverFlagVariety})) on each one of the factors of $\operatorname{Fl}_{\mathbf{d}^\ell_\bullet}$. Notice that $\operatorname{rk}(\mathcal{F}_j^{n_k})=\sum_{n=1}^j\alpha_{n,k}$ and $\operatorname{rk}(\mathcal{F}_j^{r_k})=\sum_{n=1}^j\beta_{n,k}$. For $k=2,\ldots,\ell$, fix a rank profile
	\begin{equation}\label{rankProfileDef}
	\Gamma_{\bullet,k}=(\gamma_{1,k},\ldots,\gamma_{s,k}), \  \Gamma_{j,k}:=\gamma_{1,k}+\cdots+\gamma_{j,k}\text{ and } 0\leq \gamma_{j,k}\leq \min\{\alpha_{j,k},\operatorname{rk}(\mathcal{F}_j^{r_k})-\Gamma_{j-1,k}\}.
	\end{equation}
	Define the relative rank-profile stratum
	$$
	\mathcal{H}_{\Gamma_{\bullet,k}}:=
	\mathcal H om_{\mathcal F\ell}^{\Gamma_{\bullet,k}}
	(\mathcal F^{n_k}_\bullet,\mathcal F^{r_k}_\bullet)
	\longrightarrow \operatorname{Fl}_{\mathbf{d}^\ell_\bullet }
	$$
	which we know to be a Zariski-locally trivial fibration by Lemma \ref{usefulLemmaInvolvingVBAndZariskiLTFibrations}. 
	A point of $\mathcal{H}_{\Gamma_{\bullet,k}}$ consists of the coordinate flag
	$\operatorname{Coord}_{\bullet,k}^{\varepsilon}\subseteq \mathbb{C}^{n_k}$ and a flag $ W_{\bullet,k}\subseteq \mathbb C^{r_k},$
	together with a filtration-preserving morphism $A_k:\mathbb C^{n_k}\to \mathbb C^{r_k}$ such that
	$A_k(\operatorname{Coord}_{j,k}^{\varepsilon})\subseteq W_{j,k}, \ \operatorname{rk}(A_k|_{\operatorname{Coord}_{j,k}^{\varepsilon}})=\Gamma_{j,k}$ for every $j$. Choosing the $A_k$, for $k=1,\ldots,\ell$, simultaneously is the same as considering 
	$$
	\mathcal{H}_{\Gamma_\bullet}
	:=
	\mathcal{H}om_{\mathcal{F}\ell}(\mathcal{F}^{n_1}_{\bullet},\mathcal{F}^{r_1}_\bullet)
	\times_{\operatorname{Fl}_{\mathbf{d}^\ell_\bullet }}
	\mathcal{H}_{\Gamma_{\bullet,2}}
	\times_{\operatorname{Fl}_{\mathbf{d}^\ell_\bullet }}\cdots\times_{\operatorname{Fl}_{\mathbf{d}^\ell_\bullet }}
	\mathcal{H}_{\Gamma_{\bullet,\ell}}.
	$$
	Let $p:\mathcal{H}_{\Gamma_\bullet}\to \operatorname{Fl}_{\mathbf{d}^\ell_\bullet }$
	be the natural projection. For all $k=2,\ldots,\ell$ there are tautological morphisms
	$$
	\mathcal A_k:
	p^*\mathcal F^{n_k}_s
	\longrightarrow
	p^*\mathcal F^{r_k}_s
	$$
	whose restriction to $p^*\mathcal F^{n_k}_j$
	has constant rank \(\Gamma_{j,k}\). Thus, the kernel
	$$
	\mathcal K_{j,k}:=
	\ker\left(\mathcal A_k|_{p^*\mathcal F^{n_k}_j}:p^*\mathcal F^{n_k}_j\longrightarrow
	p^*\mathcal F^{r_k}_j\right)
	$$
	is a vector subbundle of $p^*\mathcal F^{n_k}_j$ and for every
	$k$, we obtain a filtration of vector bundles
	$$
	0\subseteq \mathcal K_{1,k}\subseteq\cdots\subseteq \mathcal K_{s,k}\subseteq p^*\mathcal F^{n_k}_s.
	$$
	The relation in the quiver chain is $A_{k+1}B_k=0$ for $k=1,\ldots,\ell-1$. Therefore, after the maps $A_{k+1}$ have been chosen, the possible maps
	$B_k:\mathbb C^{r_k}\to \mathbb C^{n_{k+1}}$
	are precisely the filtration-preserving morphisms whose image lies in the kernel filtration of $A_{k+1}$. Thus, over $\mathcal{H}_{\Gamma_\bullet}$, the choices of $B_k$ are parameterised by the vector bundle
	$\mathcal H om_{\mathcal F\ell}\left(p^*\mathcal F^{r_k}_\bullet,\mathcal K_{\bullet,k+1}\right)$.
	Consequently, the stratum of
	$\operatorname{Frep}_{\mathbf d^\ell_1\cdots\mathbf d^\ell_s}
	(Q^\ell,\mathcal R^\ell)$ with rank profile $\Gamma_\bullet$ is the total
	space of the vector bundle
	$$
	\bigoplus_{k=1}^{\ell-1}
	\mathcal H om_{\mathcal F\ell}
	\left(
	p^*\mathcal F^{r_k}_\bullet,
	\mathcal K_{\bullet,k+1}
	\right)
	\longrightarrow
	\mathcal{H}_{\Gamma_\bullet}.
	$$
	We denote this stratum by $\operatorname{Frep}^{\Gamma_\bullet}_{\mathbf{d}^\ell_\bullet}$. 
	\\
	\\
	Putting all together, a point of $\operatorname{Frep}^{\Gamma_\bullet}_{\mathbf{d}^\ell_\bullet}$ consists of flags
	$$
	\operatorname{Coord}_{\bullet,k}^{\varepsilon}\subseteq \mathbb C^{n_k},
	\ W_{\bullet,k}\subseteq \mathbb C^{r_k},
	$$
	linear maps
	$$
	A_k:\mathbb C^{n_k}\to \mathbb C^{r_k},\ A_k(\operatorname{Coord}_{j,k}^{\varepsilon})\subseteq W_{j,k}, \ k=1,\ldots,\ell,
	$$
	$$
	B_k:\mathbb C^{r_k}\to \mathbb C^{n_{k+1}}, \ B_k(W_{j,k})\subseteq \ker(A_{k+1}|_{\operatorname{Coord}_{j,k+1}^{\varepsilon}}), \ k=1,\ldots,\ell-1
	$$
	where
	$$
	\operatorname{rk}(A_k|_{\operatorname{Coord}_{j,k}^{\varepsilon}})=\Gamma_{j,k}, \ k=2,\ldots,\ell,
	$$
	for every $j=1,\ldots,s$. Notice that the second condition means, in particular, that $A_{k+1}B_k=0$. 
	\\
	\\
	Thus, by Lemma \ref{computationsMotivicClasses}, in the Grothendieck ring of varieties $K_0(\text{Var}_{\mathbb{C}})$,
	$$
	[\operatorname{Frep}^{\Gamma_\bullet}_{\mathbf{d}^\ell_\bullet}]
	=
	[\mathcal{H}_{\Gamma_\bullet}]\cdot
	\mathbb L^{N_{\Gamma_\bullet}},
	$$
	where
	$$
	N_{\Gamma_\bullet}=\sum_{k=1}^{\ell-1}
	\operatorname{rk}
	(\mathcal H om_{\mathcal F\ell}
	\left(
	p^*\mathcal F^{r_k}_\bullet,
	\mathcal K_{\bullet,k+1}
	\right))=\sum_{k=1}^{\ell-1}\sum_{j=1}^{s}\beta_{j,k}\left(\sum_{i=1}^{j}\alpha_{i,k+1}-\Gamma_{j,k+1}\right).
	$$
	Since there is a stratification by locally closed subsets:
	$$
	\operatorname{Frep}_{\mathbf d^\ell_1\cdots\mathbf d^\ell_s}
	(Q^\ell,\mathcal R^\ell)
	=
	\bigsqcup_{\Gamma_\bullet}
	\operatorname{Frep}^{\Gamma_\bullet}_{\mathbf{d}_\bullet^\ell},
	$$
	we obtain in $K_0(\operatorname{Var}_\mathbb{C})$ that
	$$
	[
	\operatorname{Frep}_{\mathbf d^\ell_1\cdots\mathbf d^\ell_s}
	(Q^\ell,\mathcal R^\ell)
	]
	=
	\sum_{\Gamma_\bullet}
	[\operatorname{Frep}^{\Gamma_\bullet}_{\mathbf{d}^\ell_\bullet}]=\sum_{\Gamma_\bullet}[\mathcal{H}_{\Gamma_\bullet}]\cdot
	\mathbb L^{N_{\Gamma_\bullet}}
	$$
	where $\Gamma_\bullet$ ranges over all rank profiles
	$\Gamma_\bullet=(\Gamma_{j,k})$ for $\Gamma_{j,k}$ as in Equation \ref{rankProfileDef}. So it only remains to compute the motivic class of the varieties $\mathcal{H}_{\Gamma_\bullet}$. 
	\\
	\\
	By Lemma \ref{fiberProductOfZariskiLocallyTrivialFibrationsIsZariskiLocallyTrivial} we have that 
	$$
	[\mathcal{H}_{\Gamma_\bullet}]=[\operatorname{Fl}_{\mathbf{d}^\ell_\bullet }]\cdot [F_1]\cdots[F_\ell]
	$$
	where $F_1$ is the typical fiber of the vector bundle $\mathcal{H}om_{\mathcal{F}\ell}(\mathcal{F}^{n_1}_{\bullet},\mathcal{F}^{r_1}_\bullet)\to \operatorname{Fl}_{\mathbf{d}^\ell_\bullet }$ and $F_k$ is that of the Zariski-locally trivial fibration $\mathcal{H}_{\Gamma_{\bullet,k}} \to \operatorname{Fl}_{\mathbf{d}^\ell_\bullet }$ for $k=2,\ldots,\ell$. For $k=1$ we have
	$$
	[F_1]=\mathbb{L}^{N_1}, \ N_1=\sum_{j=1}^{s}\alpha_{j,1}\sum_{i=1}^j\beta_{i,1}
	$$ Indeed, the typical fibers $F_k$ correspond to the variety described in Equation (\ref{varietyOfRankProfileFilteredMorphisms}) for which we compute now its motivic class. 
	\begin{lemma}
		Let
		$$
		0\subseteq V_{1}\subseteq V_{2}\subseteq\cdots\subseteq  V_{s}=\mathbb{C}^{n}, \ \dim V_{j}=\sum_{i=1}^{j}\alpha_{i},
		$$
		$$
		0\subseteq W_{1}\subseteq W_{2}\subseteq\cdots\subseteq  W_{s}=\mathbb{C}^{r}, \ \dim W_{j}=\sum_{i=1}^{j}\beta_{i}.
		$$ 
		be flags.
		Fix a rank profile $\Gamma=(\gamma_1,\ldots,\gamma_s)$, $\Gamma_k=\sum_{j=1}^k\gamma_j$, and let 
		$$
		F=\left\{
		\varphi:\mathbb{C}^n\to \mathbb{C}^r
		\ | \
		\varphi(V_k)\subseteq W_k
		\text{ and }
		\operatorname{rk}(\varphi|_{V_k})=\Gamma_k
		\text{ for every }k
		\right\}.
		$$
		Then 
		$$
		[F]=\prod_{k=1}^s[\operatorname{GL}(\gamma_k)]\mathbb{L}^{\alpha_{k}\Gamma_{k-1}}[\operatorname{Gr}(\gamma_k,\dim W_k-\Gamma_{k-1})][\operatorname{Gr}(\alpha_k-\gamma_k,\alpha_k)].
		$$
	\end{lemma}
	\begin{proof}
		Let $\text{Hom}^{\gamma_1}(V_{1},W_{1})$ be the variety parameterising rank $\gamma_1$ linear maps $V_{1}\to W_{1}$. Let $g_1:\text{Hom}^{\gamma_1}(V_{1},W_{1})\to \operatorname{Gr}(\gamma_1,W_{2})$ be the morphism given by the composition
		$$
		\begin{matrix}
			\text{Hom}^{\gamma_1}(V_{1},W_{1})&\longrightarrow& \operatorname{Gr}(\gamma_1,W_{1})& \longhookrightarrow & \operatorname{Gr}(\gamma_1,W_{2})\\
			A_1&\longmapsto &\text{im}(A_1)&\longmapsto& \text{im}(A_1)
		\end{matrix}
		$$
		where the right-most morphism is induced by the inclusion $W_{1}\subseteq W_{2}$. To see why this map is algebraic refer to Equation (\ref{mapSendingMatrixToKerAndImage}). Recall that over $\operatorname{Gr}(\gamma_1,W_{2})$ there is a tautological exact sequence of vector bundles
		\begin{center}
			\begin{tikzcd}[column sep =normal, row sep =normal]
				0\arrow[r]&\mathcal{S}_1\arrow[r]&\operatorname{Gr}(\gamma_1,W_{2})\times W_{2} \arrow[r]&\mathcal{Q}_1\arrow[r]&0.
			\end{tikzcd}
		\end{center}
		Now, define 
		$$
		X_{\Gamma_2}=\mathcal{H}om(V_{2}/V_{1},\operatorname{Hom}^{\gamma_1}(V_{1},W_{1})\times W_{2})\times_{\mathcal{H}om(V_{2}/V_{1},g_1^*\mathcal{Q}_1)}\mathcal{H}om^{\gamma_2}(V_{2}/V_{1},g_1^*\mathcal{Q}_1).
		$$
		Here $\mathcal{H}om^{\gamma_2}(V_{2}/V_{1},g_1^*\mathcal{Q}_1)\to \operatorname{Hom}^{\gamma_1}(V_{1},W_{1})$ is the fiber bundle parameterising rank $\gamma_2$ morphisms $V_{2}/V_{1}\to g_1^*\mathcal{Q}_1|_{\{A_1\}}=W_{2}/\operatorname{im}(A_1)$ as defined in (2) of Lemma \ref{usefulLemmaInvolvingVBAndZariskiLTFibrations}. Notice, furthermore, that $X_{\Gamma_2}$ is the pullback of the vector bundle $\mathcal{H}om(V_{2}/V_{1},\operatorname{Hom}^{\gamma_1}(V_{1},W_{1})\times W_2)\to \mathcal{H}om(V_{2}/V_{1},g_1^*\mathcal{Q}_1)$ to $ \mathcal{H}om^{\gamma_2}(V_{2}/V_{1},g_1^*\mathcal{Q}_1)$. Thus
		$$
		X_{\Gamma_2}\overset{p_1}{\longrightarrow} \mathcal{H}om^{\gamma_2}(V_{2}/V_{1},g_1^*\mathcal{Q}_1)\overset{q_1}{\longrightarrow} \text{Hom}^{\gamma_1}(V_{1},W_{1})
		$$ 
		is a tower of Zariski-locally trivial fibrations
		which is parameterising all rank $\Gamma_2$ morphisms $A_{\leq2}:V_{2}\to W_2$ such that $A_{\leq 2}(V_{1})\subseteq W_{1}$ and $A_{\leq 2}|_{V_{1}}$ has rank $\gamma_1$. Now, let $\operatorname{Gr}(\Gamma_2,\operatorname{Hom}^{\gamma_1}(V_{1},W_{1})\times W_{3})$ be the rank $\Gamma_2$ relative Grassmannian of the vector bundle $\operatorname{Hom}^{\gamma_1}(V_{1},W_{1})\times W_{3}\to \operatorname{Hom}^{\gamma_1}(V_{1},W_{1})$ and $g_2:X_{\Gamma_2}\to \operatorname{Gr}(\Gamma_2,\operatorname{Hom}^{\gamma_1}(V_{1},W_{1})\times W_{3})$ be given by the composition 
		$$
		\begin{matrix}
			X_{\Gamma_2}&\longrightarrow& \operatorname{Gr}(\Gamma_2,\operatorname{Hom}^{\gamma_1}(V_{1},W_{1})\times W_{2})& \longhookrightarrow & \operatorname{Gr}(\Gamma_2,\operatorname{Hom}^{\gamma_1}(V_{1},W_{1})\times W_{3})\\
			A_{\leq2}&\longmapsto &\text{im}(A_{\leq 2})&\longmapsto& \text{im}(A_{\leq2}).
		\end{matrix}
		$$

		This is a relative version of $g_1$ defined above and one can see that it is algebraic in a similar manner we did for the map in Equation (\ref{mapSendingMatrixToKerAndImage}). Similarly, we have an exact sequence 
		\begin{center}
			\begin{tikzcd}[column sep =normal, row sep =normal]
				0\arrow[r]&\mathcal{S}_2\arrow[r]&p_{\operatorname{Gr}}^*(\operatorname{Hom}^{\gamma_1}(V_{1},W_{1})\times W_{3}) \arrow[r]&\mathcal{Q}_2\arrow[r]&0
			\end{tikzcd}
		\end{center}
		where $p_{\operatorname{Gr}}:\operatorname{Gr}(\Gamma_2,\operatorname{Hom}^{\gamma_1}(V_{1},W_{1})\times W_{3})\to \operatorname{Hom}^{\gamma_1}(V_{1},W_{1})$ is the canonical projection. We also define 
		$$
		X_{\Gamma_3}=\mathcal{H}om(V_{3}/V_{2},\pi_1^*(\operatorname{Hom}^{\gamma_1}(V_{1},W_{1})\times W_{3}))\times_{\mathcal{H}om(V_{3}/V_{2},g_2^*\mathcal{Q}_2)}\mathcal{H}om^{\gamma_3}(V_{3}/V_{2},g_2^*\mathcal{Q}_2)
		$$
		where $\pi_1=q_1\circ p_1$. This construction yields  the tower of Zariski-locally trivial fibrations 
		$$
		X_{\Gamma_3}\longrightarrow \mathcal{H}om^{\gamma_3}(V_{3}/V_{2},g_2^*\mathcal{Q}_2)\longrightarrow X_{\Gamma_2}\overset{p_1}{\longrightarrow} \mathcal{H}om^{\gamma_2}(V_{2}/V_{1},g_1^*\mathcal{Q}_1)\overset{q_1}{\longrightarrow} \operatorname{Hom}^{\gamma_1}(V_{1},W_{1})
		$$
		which is now parameterising all rank $\Gamma_3$ morphisms $A_{\leq 3}:V_{3}\to W_{3}$ such that, for $j=1,2$, $A_{\leq 3}(V_{j})\subseteq W_{j}$, rank of $A_{\leq 3}|_{V_{1}}$ is $\gamma_1$ and rank of $A_{\leq 3}|_{V_{2}}$ is $\Gamma_2$. If we carry out this construction for all the steps of the flag we will end up with a tower of Zariski-locally trivial fibrations 
		\begin{center}
			\begin{tikzcd}[scale cd=1,row sep=normal, column sep=large]
				X_{\Gamma_s} \rar & \mathcal{H}om^{\gamma_s}(\mathbb{C}^{n}/V_{s-1},g_{s-1}^*\mathcal{Q}_{s-1}) \rar
				\ar[draw=none]{d}[name=X, anchor=center]{}
				& X_{\Gamma_{s-1}} \ar[rounded corners,
				to path={ -- ([xshift=2ex]\tikztostart.east)
					|- (X.center) \tikztonodes
					-| ([xshift=-2ex]\tikztotarget.west)
					-- (\tikztotarget)}]{dll}[at end]{} \\      
				\cdots \rar & \mathcal{H}om^{\gamma_2}(V_{2}/V_{1},g_1^*\mathcal{Q}_1) \rar & \operatorname{Hom}^{\gamma_1}(V_{1},W_{1})
			\end{tikzcd}
		\end{center}
		where $X_{\Gamma_s}$ is parameterising rank $\Gamma_s$ morphisms $A:\mathbb{C}^{n}\to \mathbb{C}^{r}$ such that $A(V_{j})\subseteq W_{j}$ and the rank of $A|_{V_{j}}$ is $\Gamma_j$.
		\\
		\\
		Lemma \ref{computationsMotivicClasses} now gives that 
		$$
		[F]= \mathbb{L}^{\alpha_s\Gamma_{s-1}}\cdot[\operatorname{Hom}^{\gamma_s}(\mathbb{C}^{\alpha_s},\mathbb{C}^{\dim W_s-\Gamma_{s-1}})]\cdots \mathbb{L}^{\alpha_2\gamma_1}\cdot [\operatorname{Hom}^{\gamma_2}(\mathbb{C}^{\alpha_2},\mathbb{C}^{\dim W_2 - \gamma_1})]\cdot [\operatorname{Hom}^{\gamma_1}(V_1,W_1)]
		$$
		Expanding each term using again Lemma \ref{computationsMotivicClasses} gives the desired formula. 
	\end{proof}
	The combination of all our computations so far yields
	\begin{align*}
		&[\operatorname{Frep}_{\mathbf d^\ell_1\cdots\mathbf d^\ell_s}
		(Q^\ell,\mathcal R^\ell)]=\sum_{\Gamma_\bullet}[\operatorname{Frep}^{\Gamma_\bullet}_{\mathbf{d}^\ell_\bullet}]=\sum_{\Gamma_\bullet}[\mathcal{H}_{\Gamma_\bullet}]\cdot\mathbb L^{N_{\Gamma_\bullet}}=\sum_{\Gamma_\bullet}[\operatorname{Fl}_{\mathbf{d}^\ell_\bullet }]\cdot [F_1]\cdots[F_\ell]\cdot \mathbb{L}^{N_{\Gamma_\bullet}}\\
		&=[\operatorname{Fl}_{\mathbf{d}_\bullet^\ell}]\sum_{\Gamma_\bullet} \mathbb{L}^{N_1+N_{\Gamma_\bullet}}\prod_{k=2}^\ell\prod_{j=1}^{s} [\operatorname{GL}(\gamma_{j,k})]\mathbb{L}^{\alpha_{j,k}\Gamma_{j-1,k}}[\operatorname{Gr}(\gamma_{j,k},r_{j,k}-\Gamma_{j-1,k})][\operatorname{Gr}(\alpha_{j,k}-\gamma_{j,k},\alpha_{j,k})]
	\end{align*}
	where $r_{j,k}=\sum_{i=1}^j\beta_{i,k}$. The remaining ingredient is the computation of the motivic class of $\operatorname{Fl}_{\mathbf{d}_\bullet^\ell}$. Lemma \ref{computationsMotivicClasses} in the appendix gives 
	$$
	[\operatorname{Fl}_{\beta_{\bullet,k}}(\mathbb C^{r_k})]=\prod_{j=1}^{s-1}[\operatorname{Gr}(r_{j,k},\mathbb{C}^{r_{j+1,k}})]
	$$
	Thus, 
	$$
	[\operatorname{Fl}_{\mathbf{d}_\bullet^\ell}]=\prod_{k=1}^\ell\prod_{j=1}^{s-1}[\operatorname{Gr}(r_{j,k},\mathbb{C}^{r_{j+1,k}})].
	$$
	We put all together in the following:
	\begin{theorem}\label{motivicClassQuiverModuliWithRelations}
	Under our previous notations and conventions, in $K_0^{\text{loc}}(\operatorname{Var}_{\mathbb{C}})$,
	$$
			\resizebox{\textwidth}{!}{$ \displaystyle
		[\mathcal{M}^{\theta^\ell}(Q^\ell,\mathbf{d}^\ell,\mathcal{R}^\ell)]=\frac{\mathbb{L}-1}{[\operatorname{GL}(\mathbf{d}^\ell)]}\sum_{*}(-1)^{s-1}\prod_{k=1}^\ell{r_{k}\choose \beta_{1,k},\cdots ,\beta_{s,k}}_{\mathbb{L}}\sum_{\Gamma_\bullet} \mathbb{L}^{N_1+N_{\Gamma_\bullet}}\prod_{k=2}^\ell\prod_{j=1}^{s} [\operatorname{GL}(\gamma_{j,k})]\mathbb{L}^{\alpha_{j,k}\Gamma_{j-1,k}}{r_{j,k}-\Gamma_{j-1,k}\choose \gamma_{j,k}}_{\mathbb{L}}{\alpha_{j,k}\choose \alpha_{j,k}-\gamma_{j,k}}_{\mathbb{L}}.$}$$
	\end{theorem}
	
	Combining this result with Proposition \ref{motivicBialynickiBirulaDecomposition} gives the following closed formula for the motive of the full moduli space $\mathcal{M}^{\theta}(\overline{Q},\mathbf{d})$:
	\begin{theorem}\label{motivicClassOfNakajimaQuiverVariety}
	In $K_0^{\operatorname{loc}}(\operatorname{Var}_\mathbb{C})$,
	$$
		\resizebox{\textwidth}{!}{$ \displaystyle
		[\mathcal{M}^{\theta}(\overline{Q},\mathbf{d})]=\mathbb{L}^{1-\langle\mathbf{d},\mathbf{d}\rangle_Q}\sum_{\mathbf{d}^\ell\in \mathscr{D}_\theta(\mathbf{d})}\frac{\mathbb{L}-1}{[\operatorname{GL}(\mathbf{d}^\ell)]}\sum_{*}(-1)^{s-1}\prod_{k=1}^\ell{r_{k}\choose \beta_{1,k},\cdots ,\beta_{s,k}}_{\mathbb{L}}\sum_{\Gamma_\bullet} \mathbb{L}^{N_1+N_{\Gamma_\bullet}}\prod_{k=2}^\ell\prod_{j=1}^{s} [\operatorname{GL}(\gamma_{j,k})]\mathbb{L}^{\alpha_{j,k}\Gamma_{j-1,k}}{r_{j,k}-\Gamma_{j-1,k}\choose \gamma_{j,k}}_{\mathbb{L}}{\alpha_{j,k}\choose \alpha_{j,k}-\gamma_{j,k}}_{\mathbb{L}}.$}$$
	\end{theorem}
	
	\section{Appendix}\label{theAppendix}
	\subsection{Quiver representations with relations}\label{representationsWIthRelations}
	Let $Q$ be a quiver. A \emph{relation} in $Q$ is a formal sum of the form 
	$$
	\sum_{\substack{p\text{ path, }|p|\geq 1\\ t(p)=i,h(p)=j}}c_p\cdot p,\ c_p\in \mathbb{C}
	$$
	where all but finitely many terms are zero. Let $\mathcal{R}$ be a set of relations of the quiver $Q$. For all $r\in \mathcal{R}$, let $t(r),\ h(r)\in Q_0$ denote, respectively, the common tail and head of all the paths appearing in $r$.  Every such $r\in \mathcal{R}$, determines a morphism of affine varieties
	$$
	\nu_r:\text{Rep}(Q,\mathbf{d})\to \text{Hom}(V_{t(r)},V_{h(r)})
	$$
	given by composing and adding linear maps according to the information given by the relation $r$. The $\mathbf{d}$-dimensional representations of $Q$ satisfying the relations in $\mathcal{R}$ are then parameterised by the closed affine subvariety 
	$$
	\text{Rep}(Q,\mathbf{d},\mathcal{R}):=\bigcap_{r\in \mathcal{R}}\nu^{-1}_r(0)\subseteq\text{Rep}(Q,\mathbf{d}).
	$$
	An easy computation shows that the variety $\text{Rep}(Q,\mathbf{d},\mathcal{R})$ is $\text{GL}(\mathbf{d})$-invariant under the action of $\text{GL}(\mathbf{d})$. For $\theta$ a generic stability parameter and $\mathbf{d}$ an indivisible dimension vector we set $\text{Rep}(Q,\mathbf{d},\mathcal{R})^{\theta\text{-}s}:=\text{Rep}(Q,\mathbf{d},\mathcal{R})\cap \text{Rep}(Q,\mathbf{d})^{\theta\text{-}s}$ and define:
	\begin{definition}
		The GIT quotient 
		$$
		\mathcal{M}^{\theta}(Q,\mathbf{d},\mathcal{R}):=\text{Rep}(Q,\mathbf{d},\mathcal{R})^{\theta\text{-}s}\sslash\text{GL}(\mathbf{d})
		$$
		is the \emph{moduli space of $\theta$-stable representations of $Q$ satisfying the relations in $\mathcal{R}$}.
	\end{definition}
	\subsection{Flag varieties}\label{flagVarietiesSection} We follow closely the expositions from Brion \cite[Section 1.2]{BrionFlagVarieties}, Anderson and Fulton \cite[Section 4.2]{AndersonFultonEquivariantCohomology}. Let $d\in \mathbb{N}$ and $(d_1,\ldots,d_s)$ a sequence of non-negative integers with sum $d$. A \emph{flag} in $\mathbb{C}^d$ of type $d_\bullet$ is an increasing sequence of linear subspaces 
	$$
	0\subseteq V_1\subseteq V_2\subseteq \cdots\subseteq V_{s-1}\subseteq V_s=\mathbb{C}^d
	$$
	such that $\dim(V_k)=\delta_k$ where $\delta_k=d_1+\ldots+d_k$. 
	\\
	\\
	We denote the set of all flags of type $d_\bullet$ as  $\operatorname{Fl}_{d_\bullet}(\mathbb{C}^d)$ and we call it the \emph{flag variety of type }$d_\bullet$. We now discuss the variety structure of $\operatorname{Fl}_{d_\bullet}(\mathbb{C}^d)$. Let $X_{s-1}:=\operatorname{Gr}(\delta_{s-1},\mathbb{C}^d)$ be the Grassmannian of $\delta_{s-1}$-dimensional planes of $\mathbb{C}^d$. Over this Grassmannian there is a tautological exact sequence of vector bundles 
	\begin{equation}\label{tautologicalExactSequenceGrassmannian}
		\begin{tikzcd}[column sep =normal, row sep =normal]
			0\arrow[r]&\mathcal{S}_{s-1}\arrow[r]&\operatorname{Gr}(\delta_{s-1},\mathbb{C}^d)\times \mathbb{C}^d \arrow[r]&\mathcal{Q}_s\arrow[r]&0
		\end{tikzcd}
	\end{equation}
	where $\mathcal{S}_{s-1}|_{\{W\}}=W$ for $W\in \operatorname{Gr}(\delta_{s-1},\mathbb{C}^d)$ and the tautological quotient $\mathcal{Q}_s\simeq (\operatorname{Gr}(\delta_{s-1},\mathbb{C}^d)\times \mathbb{C}^d)/\mathcal{S}_{s-1}$ has rank $d_s$. Consider now the \emph{relative Grassmannian} $X_{s-2}:=\mathbf{Gr}_{X_{s-1}}(\delta_{s-2},\mathcal{S}_{s-1})$ (see also \cite[Section 5.1.5]{NitsureHilbertQuotSchemes}) which can be constructed using the frame bundle of $\mathcal{S}_{s-1}$ and the left action of $\operatorname{GL}(\delta_{s-1})$ on the Grassmannian $\operatorname{Gr}(\delta_{s-2},\mathbb{C}^{\delta_{s-1}})$. The relative Grassmannian parameterises pairs $(V_{s-1},V_{s-2}\subseteq V_{s-1})$ such that $\dim(V_{s-1})=\delta_{s-1}$ and $\dim(V_{s-2})=\delta_{s-2}$.  Over the relative Grassmannian just described there is an analogous exact sequence of vector bundles to the one that we have in Equation (\ref{tautologicalExactSequenceGrassmannian}):
	\begin{equation}\label{tautologicalExactSequenceRelativeGrassmannian}
		\begin{tikzcd}[column sep =normal, row sep =normal]
			0\arrow[r]&\mathcal{S}_{s-2}\arrow[r]&p_{s-2}^*\mathcal{S}_{s-1} \arrow[r]&\mathcal{Q}_{s-1}\arrow[r]&0.
		\end{tikzcd}
	\end{equation}
	Here $p_{s-2}:\mathbf{Gr}_{X_{s-1}}(\delta_{s-2},\mathcal{S}_{s-1})\to \operatorname{Gr}(\delta_{s-1},\mathbb{C}^d)$ is the canonical projection and $\mathcal{S}_{s-2}|_{(V_{s-1},V_{s-2})}=V_{s-2}$. The relative Grassmannian $X_{s-3}:=\mathbf{Gr}_{X_{s-2}}(\delta_{s-3},\mathcal{S}_{s-2})$ is now parameterising tuples $(V_{s-1},V_{s-2}\subseteq V_{s-1},V_{s-3}\subseteq V_{s-2}\subseteq V_{s-1})$ such that $\dim(V_{s-k})=\delta_{s-k}$, $k=1,2,3$. Iterating this construction gives a tower of relative Grassmannians
	$$
	X_1\overset{p_1}{\longrightarrow}X_2\longrightarrow\cdots\longrightarrow X_{s-2} \overset{p_{s-2}}{\longrightarrow} X_{s-1}
	$$
	 where $p_k:X_{k}\to X_{k+1}$ are the canonical projections and such that $X_1=\mathbf{Gr}_{X_2}(d_1,\mathcal{S}_2)\simeq \operatorname{Fl}_{d_\bullet}(\mathbb{C}^d)$. This point of view on the flag variety will be useful to compute its motivic class (see Lemma \ref{computationsMotivicClasses}). 
	\\
	\\
	Notice, by construction, that over the flag variety $\operatorname{Fl}_{d_\bullet}(\mathbb{C}^d)$ there is a tautological flag
	\begin{equation}\label{tautologicalFlagOverFlagVariety}
		0\subseteq \mathscr{S}_{1}\subseteq\mathscr{S}_{2}\subseteq\cdots\subseteq \mathscr{S}_{s}=\operatorname{Fl}_{d_{\bullet}}(\mathbb{C}^{d})\times\mathbb{C}^{d}, \ \mathscr{S}_k:=(p_{k-1}\circ \ldots\circ p_{1})^*\mathcal{S}_{k} 
	\end{equation}
	and quotients 
	\begin{equation}\label{tautologicalQuotientsOverFlagVariety}
		\operatorname{Fl}_{d_{\bullet}}(\mathbb{C}^{d})\times\mathbb{C}^{d} \twoheadrightarrow \mathscr{Q}_{1}\twoheadrightarrow \mathscr{Q}_{2}\twoheadrightarrow\cdots\twoheadrightarrow \mathscr{Q}_{s-1}, \ \mathscr{Q}_{k}\simeq \mathscr{S}_{s}/\mathscr{S}_{k},
	\end{equation}
	of vector bundles such that $\operatorname{rk}(\mathscr{S}_{k})=\delta_{k}$ and $\mathscr{S}_{k}|_{F_{\bullet}}=F_{k}$ for every $F_{\bullet}\in\operatorname{Fl}_{d_{\bullet}}(\mathbb{C}^{d})$. 
	\\
	\\
	Finally, we discuss the so-called \emph{coordinate flag varieties}. Let $\{e_1,\ldots,e_d\}$ be the standard basis of $\mathbb{C}^d$. We say that $V_{\bullet}\in \operatorname{Fl}_{d_\bullet}(\mathbb{C}^d)$ is a \emph{coordinate flag} if for some subsets
	$$
	I_1\subseteq I_2\subseteq \cdots \subseteq I_{s-1}\subseteq I_s={1,\ldots,d}, \ |I_k|=\delta_{k}
	$$ 
	we have that 
	$$V_{k}=\langle e_i\rangle_{i\in I_k}.$$
	We denote  the subset of the flag variety parameterising all coordinate flags as $\operatorname{Fl}^{\text{coord}}_{d_\bullet}(\mathbb{C}^d)$. Notice that $\operatorname{Fl}^{\text{coord}}_{d_\bullet}(\mathbb{C}^d)$ is parameterised by the coset space $\mathfrak{S}_d/(\mathfrak{S}_{d_1}\times \cdots \times\mathfrak{S}_{d_s})$ where $\mathfrak{S}_d$ denotes the symmetric group in $d$ elements. One can show that $\operatorname{Fl}^{\text{coord}}_{d_\bullet}(\mathbb{C}^d)=(\operatorname{Fl}_{d_\bullet}(\mathbb{C}^d))^T$ where $T=(\mathbb{C}^*)^d$ and the action of $T$ is induced from that of $\operatorname{GL}(d)$ on the flag variety $\operatorname{Fl}_{d_\bullet}(\mathbb{C}^d)$ given by $g\cdot(V_\bullet)=(gV_\bullet)$ \cite[Section 15.3]{AndersonFultonEquivariantCohomology}. $\operatorname{Fl}^{\text{coord}}_{d_\bullet}(\mathbb{C}^d)$ is then a finite closed subscheme \cite[Theorem 7.1]{MilneAlgebraicGroups} of the corresponding flag variety with $\binom{d}{d_{1},\ldots,d_{s}}$ points.
	
	\subsection{Quiver Flag Varieties}\label{quiverFlagVarieties}
	Let $Q$ be a quiver, $\mathcal{R}$ a set of relations and $\mathbf{d}=\sum_{k=1}^s\mathbf{d}_k$ a decomposition of the dimension vector $\mathbf{d}\in \mathbb{N}^{|Q_0|}$. We define $$\operatorname{Fl}_{\mathbf{d}_\bullet}=\prod_{i\in Q_0}\operatorname{Fl}_{d_{\bullet,i}}(\mathbb{C}^{d_i})$$
	where $\operatorname{Fl}_{d_{\bullet,i}}(\mathbb{C}^{d_i})$ is the flag variety defined in the previous section of the appendix.
	
	\begin{definition}
		The \emph{quiver flag variety} is the incidence variety whose underlying set is
		$$
		\resizebox{\textwidth}{!}{$
		\operatorname{Frep}_{\mathbf{d}_{\bullet}}(Q,\mathcal{R}):=\{((X_{\alpha})_{\alpha\in Q_1},(V_{\bullet,i})_{i\in Q_0})\in \text{Rep}(Q,\mathbf{d},\mathcal{R})\times \operatorname{Fl}_{\mathbf{d}_\bullet}| \ X_{\alpha}(V_{k,t\alpha})\subseteq V_{k,h\alpha}\ \forall k=1,\ldots,s-1\text{ and } \alpha\in Q_1\}.$}
		$$
	\end{definition}
	Now we describe its scheme structure. Recall, from equations (\ref{tautologicalFlagOverFlagVariety}) and (\ref{tautologicalQuotientsOverFlagVariety}), that over each factor $\operatorname{Fl}_{d_{\bullet,i}}(\mathbb{C}^{d_i})$ there is a tautological flag $0\subseteq\mathcal{S}_{1,i}\subseteq\cdots\subseteq\mathcal{S}_{s,i}=\operatorname{Fl}_{d_{\bullet,i}}(\mathbb{C}^{d_i})\times \mathbb{C}^{d_i}$ and quotients 
	$
	\operatorname{Fl}_{d_{\bullet,i}}(\mathbb{C}^{d_i})\times\mathbb{C}^{d_i} \twoheadrightarrow \mathcal{Q}_{1,i}\twoheadrightarrow \cdots\twoheadrightarrow \mathcal{Q}_{s-1,i}, \ \mathcal{Q}_{k,i}\simeq \mathcal{S}_{s,i}/\mathcal{S}_{k,i},
	$
	of vector bundles such that $\operatorname{rk}(\mathcal{S}_{k,i})=\sum_{j=1}^kd_{j,i}$ and $\mathcal{S}_{k,i}|_{F_{\bullet}}=F_{k}$ for every $F_{\bullet}\in\operatorname{Fl}_{d_{\bullet,i}}(\mathbb{C}^{d_i})$. Let us, abusively, use the same notation for the pull-back of these to $\operatorname{Rep}(Q,\mathbf{d},\mathcal{R})\times \operatorname{Fl}_{\mathbf{d}_\bullet}$. For all $\alpha\in Q_1$, there is a tautological morphism of vector bundles $\mathcal{X}_{\alpha}:\mathcal{S}_{s,t\alpha}\to \mathcal{S}_{s,h\alpha}$ which is the pullback of that defined at the beginning of Proposition \ref{existenceUniversalFamilyNakajimaQuiverVarieties}. Then, we want for all $k=1,\ldots,s-1$ and $\alpha\in Q_1$ the morphism, or equivalently the associated section, 
	$$
	\sigma_{\alpha,k}:=\mathcal{S}_{k,t\alpha}\hookrightarrow \mathcal{S}_{s,t\alpha}\overset{\mathcal{X}_{\alpha}}{\longrightarrow}\mathcal{S}_{s,h\alpha}\twoheadrightarrow \mathcal{Q}_{k,h\alpha}
	$$
	to vanish. The quiver flag variety is therefore the closed subscheme of $\operatorname{Rep}(Q,\mathbf{d},\mathcal{R})\times \operatorname{Fl}_{\mathbf{d}_\bullet}$ given by the scheme-theoretic intersection 
	$$
	\operatorname{Frep}_{\mathbf{d}_{\bullet}}(Q,\mathcal{R}) = \bigcap_{\substack{\alpha\in Q_1 \\ k=1,\ldots,s-1}} Z(\sigma_{\alpha,k})
	$$
	where $Z(\sigma_{\alpha,k})\subseteq \operatorname{Rep}(Q,\mathbf{d},\mathcal{R})\times \operatorname{Fl}_{\mathbf{d}_\bullet}$ denotes the zero subscheme of the global section $\sigma_{\alpha,k}\in H^0( \operatorname{Rep}(Q,\mathbf{d},\mathcal{R})\times \operatorname{Fl}_{\mathbf{d}_\bullet},\mathcal{H}om(\mathcal{S}_{k,t\alpha},\mathcal{Q}_{k,h\alpha}))$.
	\\
	\\
	Finally, we would like to comment on the fact that there are two canonical projections 
	\begin{center}
		\begin{tikzcd}
			&\operatorname{Frep}_{\mathbf{d}_{\bullet}}(Q,\mathcal{R})\arrow[dr,"p_2"]\arrow[dl,"p_1",']& \\
			\operatorname{Rep}(Q,\mathbf{d},\mathcal{R})&& \operatorname{Fl}_{\mathbf{d}_\bullet}
		\end{tikzcd}
	\end{center}
	and that $p_2$ defines a Zariski-locally trivial fibration whose typical fiber is the closed subvariety of $\operatorname{Rep}(Q,\mathbf{d},\mathcal{R})$ of flag-preserving representations. To see this, we fix a flag $F_\bullet^0	=
	(F_{\bullet,i}^0)_{i\in Q_0} \in \operatorname{Fl}_{\mathbf d_\bullet}$ and identify $\operatorname{Fl}_{\mathbf{d}_\bullet}$ with the quotient $\operatorname{GL}(\mathbf{d})/\operatorname{P}(\mathbf{d}_\bullet)$ where $\operatorname{P}(\mathbf{d}_\bullet)=\prod_{i\in Q_0}\operatorname{P}(\mathbf{d}_{\bullet,i})$ is the product of the parabolic subgroups fixing the flag $F_{\bullet,i}$ for each $i\in Q_0$. Let
	$$
	Z_{\mathbf d_\bullet}
	:=
	\left\{ X\in\operatorname{Rep}(Q,\mathbf d,\mathcal R)
	\ \middle|\
	X_\alpha(F_{k,t\alpha}^0)
	\subseteq
	F_{k,h\alpha}^0
	\text{ for every }
	\alpha\in Q_1,\;
	k=1,\ldots,s-1\right\}.
	$$
	This is a closed $\operatorname{P}(\mathbf{d}_\bullet)$-invariant subvariety of $\operatorname{Rep}(Q,\mathbf d,\mathcal R)$. One can then see that
	$$
	\operatorname{GL}(\mathbf{d})\times_{\operatorname{P}(\mathbf{d}_\bullet)}Z_{\mathbf{d}_{\bullet}}\simeq \operatorname{Frep}_{\mathbf{d}_\bullet}(Q,\mathcal{R})
	$$
	and that the projection $p_2$ is identified with the canonical map $	\operatorname{GL}(\mathbf{d})\times_{\operatorname{P}(\mathbf{d}_\bullet)}Z_{\mathbf{d}_{\bullet}}\to \operatorname{GL}(\mathbf{d})/\operatorname{P}(\mathbf{d}_\bullet)$. The local triviality of $p_2$ just discussed will be useful for the motivic computations we carry out in Section \ref{motivicClassHigherRankStarShapedQuiver}. 

	\subsection{Grothendieck ring of varieties}\label{GrothendieckRingOfVarieties}
	Denote by $\operatorname{Var}_{\mathbb{C}}$ the category of quasi-projective varieties over $\mathbb{C}$. For $X\in \operatorname{Var}_{\mathbb{C}}$, let $[X]$ denote its isomorphism class. The \emph{Grothendieck ring of varieties}, $K_0(\operatorname{Var}_{\mathbb{C}})$, is the commutative ring whose underlying group is the free abelian group generated by the isomorphism classes $[X]$ modulo the relation 
	$$
	[X]=[X']+[X\setminus X']
	$$ 
	where $X'\subseteq X$ is a Zariski-closed subset (notice, in particular, that $[X]+[X']=[X\sqcup X']$). The product in $K_0(\operatorname{Var}_{\mathbb{C}})$ is given by 
	$$
	[X]\cdot [X']=[X\times X'].
	$$
	The additive and multiplicative units are $0=[\varnothing]$ and $1=[\text{Spec}(\mathbb{C})]$ respectively. The class in $K_0(\operatorname{Var}_{\mathbb{C}})$ of the affine line is known in the literature as the \emph{Lefschetz motive} and we denote it by $\mathbb{L}:=[\mathbb{A}^1_{\mathbb{C}}]$. In our computations, we will need to invert, for instance, the Lefschetz motive or elements of the form $\mathbb{L}^n-1$ so this is why one often works in the localization $K_0^{\operatorname{loc}}(\operatorname{Var}_{\mathbb{C}}):=K_0(\operatorname{Var}_{\mathbb{C}})[[\text{GL}(d)]^{-1}:d\geq 1]$ (see \cite[Lemma 3.8]{BridgelandIntroMotivicHallAlgebras}). In the following lemma we compute motivic classes that are useful for some computations carried out in this paper. 
	\begin{lemma}\label{computationsMotivicClasses}
		Under the previous notations, in $K_0^{\operatorname{loc}}(\operatorname{Var}_{\mathbb{C}})$:
		\begin{enumerate}
			\item If $f:X\to Y$ is a Zariski locally trivial fibration with fibre $F$, then $[X]=[F]\cdot [Y]$. 
			\item More generally, if $X_n\to X_{n-1}\to \cdots\to X_1\to X_0$ is a tower of Zariski-locally trivial fibrations such that the fiber of $X_{k}\to X_{k-1}$ is $F_k$, then 
			$$
			[X_n]=[F_n]\cdot [F_{n-1}] \cdots[F_1]\cdot[X_0].
			$$
			\item $[\operatorname{GL}(d)]= (\mathbb{L}^{d}-1)(\mathbb{L}^d-\mathbb{L})\cdots (\mathbb{L}^d-\mathbb{L}^{d-1})$.
			\item Let $0\leq n\leq d$ be integers and let $\operatorname{Gr}(n,d)$ be the Grassmannian of $n$-planes inside a $d$-dimensional vector space, then 
			$$[\operatorname{Gr}(n,d)]={d\choose n}_{\mathbb{L}}:=\frac{[\operatorname{GL}(d)]}{[\operatorname{GL}(n)][\operatorname{GL}(d-n)]\mathbb{L}^{n(d-n)}}=\frac{(\mathbb{L}^d-1)(\mathbb{L}^d-\mathbb{L})\cdots(\mathbb{L}^{d}-\mathbb{L}^{n-1})}{(\mathbb{L}^{n}-1)(\mathbb{L}^n-\mathbb{L})\cdots (\mathbb{L}^n-\mathbb{L}^{n-1})}$$.
			\item Let $\operatorname{Hom}^r(\mathbb{C}^n,\mathbb{C}^m)$ denote the variety of linear maps $\mathbb{C}^{n}\to \mathbb{C}^{m}$ of rank $r\leq \min\{n,m\}$. Then $$[\operatorname{Hom}^r(\mathbb{C}^n,\mathbb{C}^m)]=[\operatorname{GL}(r)][\operatorname{Gr}(r,m)][\operatorname{Gr}(n-r,n)].$$
			\item Let $\operatorname{Fl}_{d_\bullet}(\mathbb{C}^d)$ be the type $d_\bullet$ flag variety as defined in Section \ref{flagVarietiesSection}. Then, following the same convention in loc. cit.: 
			$$
			[\operatorname{Fl}_{d_\bullet}(\mathbb{C}^d)]=\prod_{k=1}^{s-1}[\operatorname{Gr}(\delta_k,\mathbb{C}^{\delta_{k+1}})]=\prod_{k=1}^{s-1}{\delta_{k+1}\choose\delta_k}_{\mathbb{L}}={d\choose d_1,\ldots,d_s}_{\mathbb{L}}.
			$$ 
		\end{enumerate}
	\end{lemma}
	\begin{proof}
		For (1) start with an open cover $Y=\bigcup_{i}U_i$, which we can assume to be finite since we are working with varieties, such that $f^{-1}(U_i)\cong U_i\times F$. Then, we have the inclusion-exclusion relation
		$$
		[X]=\sum_{i}[f^{-1}(U_i)]-\sum_{i_1<i_2}[f^{-1}(U_{i_1}\cap U_{i_2})]+\cdots=[F]\cdot [Y].
		$$
		The second statement follows by induction from the first one. For (3), let 
		$$
		X_k=\{(v_1,v_2,\ldots,v_k)\in \mathbb{C}^{dk}\ | \ v_1,\ldots,v_k \text{ are linearly independent} \}, \ k=1,\ldots,d.
		$$
		This is the open subvariety of $\mathbb{A}^{dk}_{\mathbb{C}}$ given by the complement of the subvariety defined by the equations of all $k\times k $ minors of the matrix $(v_1,\ldots,v_k)$. Notice, in particular, that $X_d=\operatorname{GL}(d)$. Let $f:X_{k-1}\to \operatorname{Gr}(k-1,d)$ be the morphism sending a tuple $(v_1,\ldots,v_{k-1})$ to the $k-1$ dimensional vector subspace it spans.  Over $\operatorname{Gr}(k-1,d)$ we consider the tautological vector bundle $\mathcal{S}$ and we note that $X_k\cong (X_{k-1}\times \mathbb{C}^d\setminus\text{Tot}(f^*\mathcal{S}))$. Thus $[X_k]=[X_{k-1}\times \mathbb{C}^d]-[\text{Tot}(f^*\mathcal{S})]=[X_{k-1}]\cdot\mathbb{L}^d-[X_{k-1}]\cdot \mathbb{L}^{k-1}$. By induction we arrive to the desired formula. 
		\\
		\\
		We now prove the fourth formula of the lemma. Again, let $\mathcal{S}$ be the tautological bundle over $\text{Gr}(n,d)$ and let $\text{Fr}(\mathcal{S})$ be its frame bundle. Recall that this is a principal $\operatorname{GL}(n)$-bundle so $[\operatorname{Fr}(\mathcal{S})]=[\operatorname{GL}(n)][\operatorname{Gr}(n,d)]$. Now, note that $\operatorname{Fr}(\mathcal{S})\cong X_n$ for which we already know the motivic class. This yields (4).
		\\
		\\
		For the fifth formula, it is enough to show that the map
		\begin{equation}\label{mapSendingMatrixToKerAndImage}
		\begin{matrix}
		p: & \operatorname{Hom}^r(\mathbb{C}^n,\mathbb{C}^m) & \longrightarrow & \operatorname{Gr}(r,m)\times \operatorname{Gr}(n-r,n) \\
		&A &\longmapsto & (\operatorname{im}(A),\ker(A))
		\end{matrix}
		\end{equation}
		is a Zariski-locally trivial fibration with typical fiber $\operatorname{GL}(r)$. Firstly we show why the map $p$ is algebraic. Consider the canonical morphism $\Phi:\operatorname{Hom}^r(\mathbb{C}^n,\mathbb{C}^m)\times \mathbb{C}^n\to \operatorname{Hom}^r(\mathbb{C}^n,\mathbb{C}^m)\times \mathbb{C}^m$ such that $\Phi|_{\{A\}}=A$. This is a constant rank morphism of vector bundles so $\operatorname{im}(\Phi)\subseteq\operatorname{Hom}^r(\mathbb{C}^n,\mathbb{C}^m)\times\mathbb{C}^m$ and $\ker(\Phi)\subseteq \operatorname{Hom}^r(\mathbb{C}^n,\mathbb{C}^m)\times \mathbb{C}^n$ define vector subbundles. The universal property of the Grassmannian tells us that these vector subbundles are determined by algebraic morphisms $p_1:\operatorname{Hom}^r(\mathbb{C}^n,\mathbb{C}^m)\to\operatorname{Gr}(r,m)$ and $p_2:\operatorname{Hom}^r(\mathbb{C}^n,\mathbb{C}^m)\to \operatorname{Gr}(n-r,n)$ respectively. The pair $(p_1,p_2)$ is precisely the map $p$. 
		\\
		\\
		Secondly, we show why the typical fiber of the morphism $p$ can be identified with $\operatorname{GL}(r)$. Let $(W,K)\in \operatorname{Gr}(r,m)\times \operatorname{Gr}(n-r,n)$. Choose bases of $\mathbb C^m$ and $\mathbb C^n$ such that the first $r$ basis vectors of $\mathbb C^m$ span $W$, and the last $n-r$ basis
		vectors of $\mathbb C^n$ span $K$. With respect to this choice of coordinates, every matrix of the preimage $p^{-1}(W,K)$ has the form
		\begin{equation}\label{fibresOfMapSendingRankRToItsImageAndKernel}
			\begin{tikzpicture}[baseline=(M.center),scale=0.5,row sep=small,column sep=small]
				\matrix (M) [
				matrix of math nodes,
				left delimiter=(,
				right delimiter=),
				nodes={
					minimum width=1.2em,
					minimum height=1.65em,
					inner sep=0pt,
					anchor=center
				},
				column sep=0pt,
				row sep=0pt
				] {
					\vphantom{0} & \vphantom{0} & \vphantom{0} & \vphantom{0} & \vphantom{0} \\
					\vphantom{0} & \vphantom{0} & \vphantom{0} & \vphantom{0} & \vphantom{0} \\
					\vphantom{0} & \vphantom{0} & \vphantom{0} & \vphantom{0} & \vphantom{0} \\
					\vphantom{0} & \vphantom{0} & \vphantom{0} & \vphantom{0} & \vphantom{0} \\
					\vphantom{0} & \vphantom{0} & \vphantom{0} & \vphantom{0} & \vphantom{0} \\
				};
				
				\draw[line width=.8pt]
				([yshift=0.4em]M-1-3.north east)
				-- ([yshift=-0.4em]M-5-3.south east);
				
				\draw[line width=.8pt]
				([xshift=-0.4em]M-3-1.south west)
				-- ([xshift=0.4em]M-3-5.south east);
				
				\node at ($(M-1-1.north west)!0.5!(M-3-3.south east)$) {\LARGE $B$};
				\node at ($(M-1-4.north west)!0.5!(M-3-5.south east)$) {\LARGE $0$};
				\node at ($(M-4-1.north west)!0.5!(M-5-3.south east)$) {\LARGE $0$};
				\node at ($(M-4-4.north west)!0.5!(M-5-5.south east)$) {\LARGE $0$};
				
				\draw[decorate,decoration={brace,amplitude=4pt,raise=2pt}]
				([yshift=1.0em]M-1-1.north west)
				-- node[midway,above=6pt] {$r$}
				([yshift=1.0em]M-1-3.north east);
				
				\draw[decorate,decoration={brace,amplitude=4pt,raise=2pt}]
				([yshift=1.0em]M-1-4.north west)
				-- node[midway,above=6pt] {$n-r$}
				([yshift=1.0em]M-1-5.north east);
				
				\draw[decorate,decoration={brace,amplitude=4pt,raise=5pt}]
				([xshift=1.1em]M-1-5.north east)
				-- node[midway,right=10pt] {$r$}
				([xshift=1.1em]M-3-5.south east);
				
				\draw[decorate,decoration={brace,amplitude=4pt,raise=5pt}]
				([xshift=1.1em]M-4-5.north east)
				-- node[midway,right=10pt] {$m-r$}
				([xshift=1.1em]M-5-5.south east);
			\end{tikzpicture}
		\end{equation}
		where $B\in \operatorname{GL}(r)$. 
		\\
		\\
		To prove that \(p\) is Zariski-locally trivial, let $B:=\operatorname{Gr}(r,m)\times\operatorname{Gr}(n-r,n)$
		and denote by $\mathcal S,\mathcal{S}'$ the pullbacks to $B$ of the tautological
		bundles on $\operatorname{Gr}(r,m)$ and $\operatorname{Gr}(n-r,n)$ respectively. Set
		$$
		\mathcal Q:=\bigl(B\times\mathbb C^n\bigr)/\mathcal{S}'.
		$$
		Thus, the fibers of $\mathcal S$ and $\mathcal Q$ over a point $(I,K)\in B$ are $I$ and $\mathbb C^n/K$, respectively. We claim that there is an isomorphism
		$$
		\operatorname{Hom}^r(\mathbb C^n,\mathbb C^m)\cong \operatorname{Isom}_B(\mathcal Q,\mathcal S)
		$$
		of varieties over $B$ given by the maps $A\mapsto (\operatorname{im}(A),\ker(A),\overline{A})$ and $\varphi\mapsfrom(I,K,\overline{\varphi})$. Here, $\overline A:\mathbb C^n/\ker(A)\longrightarrow\operatorname{im}(A)$
		is the isomorphism induced by $A$. And, conversely, given $(I,K)\in B$ and an isomorphism
		$\overline{\varphi}:\mathbb C^n/K\to I$, the composition
		$\varphi:\mathbb C^n\overset{p_K}{\twoheadrightarrow}\mathbb C^n/K\xrightarrow{\ \overline{\varphi}\ }I\overset{\iota_I}{\hookrightarrow}\mathbb C^m$ is a rank $r$ map with kernel $K$ and image $I$. These
		constructions are inverse to each other, and under this identification
		the morphism \(p\) corresponds to the natural projection
		$
		\operatorname{Isom}_B(\mathcal Q,\mathcal S)\longrightarrow B$. Over any open subset \(U\subseteq B\) on which both
		\(\mathcal Q\) and \(\mathcal S\) are trivial, we have
		$$
		\operatorname{Isom}_B(\mathcal Q,\mathcal S)|_U
		\cong
		U\times\operatorname{GL}(r),
		$$
		which follows from Lemma \ref{usefulLemmaInvolvingVBAndZariskiLTFibrations}. This shows the desired local triviality. 
		\\
		\\
		It only remains to justify why the maps $A\mapsto (\operatorname{im}(A),\ker(A),\overline{A})$ and $(I,K,\overline{\varphi})\mapsto\varphi=\iota_I\circ \overline{\varphi}\circ p_K$ are algebraic. First notice that the canonical morphism $\Phi:\operatorname{Hom}^r(\mathbb{C}^n,\mathbb{C}^m)\times \mathbb{C}^n\to \operatorname{Hom}^r(\mathbb{C}^n,\mathbb{C}^m)\times \mathbb{C}^m$ factors
		uniquely as
		$$
		\operatorname{Hom}^r(\mathbb{C}^n,\mathbb{C}^m)\times \mathbb{C}^n\
		\twoheadrightarrow
		(\operatorname{Hom}^r(\mathbb{C}^n,\mathbb{C}^m)\times \mathbb{C}^n)/\ker(\Phi)
		\xrightarrow{\ \overline\Phi\ }
		\operatorname{im}(\Phi)
		\hookrightarrow
		\operatorname{Hom}^r(\mathbb{C}^n,\mathbb{C}^m)\times \mathbb{C}^m
		$$
		where $\overline\Phi$ is an algebraic morphism of
		vector bundles. Its fiber over $A\in \operatorname{Hom}^r(\mathbb{C}^n,\mathbb{C}^m)$ is precisely the induced map	$\overline A:\mathbb C^n/\ker(A)\longrightarrow\operatorname{im}(A)$.
		Since $A$ has rank $r$, the map $\overline A$ is an isomorphism for
		every $A\in \operatorname{Hom}^r(\mathbb{C}^n,\mathbb{C}^m)$. Hence $\overline\Phi$ is an algebraic isomorphism of
		vector bundles or, equivalently, a section of the corresponding isomorphism bundle and therefore the assignment $A\mapsto \overline{A}$ is simply the evaluation map of this section. To see the algebraicity of the second map it is enough to show the algebraicity of the maps $I\mapsto \iota_I$ and $K\mapsto p_K$. This is because these correspond to the evaluation maps of the sections corresponding to the morphisms associated to the  pullbacks of the canonical quotient
		\[
		B\times\mathbb C^n\twoheadrightarrow\mathcal Q,
		\]
		and inclusion
		\[
		\mathcal S\hookrightarrow B\times\mathbb C^m.
		\]
		to $\operatorname{Isom}_B(\mathcal{Q},\mathcal{S})$ respectively.
				\\
		\\
		For the last formula, we recall Section \ref{flagVarietiesSection} where we saw that there is a Zariski-locally trivial fibration
		$$
		\operatorname{Fl}_{d_\bullet}(\mathbb{C}^d)=\mathbf{Gr}(d_1,\mathcal{S}_2)\longrightarrow \mathbf{Gr}(\delta_2,\mathcal{S}_3)\longrightarrow \cdots\longrightarrow \mathbf{Gr}(\delta_{s-2},\mathcal{S}_{s-1})\longrightarrow\operatorname{Gr}(\delta_{s-1},\mathbb{C}^d).
		$$
		Therefore, item (2) of this lemma gives the desired formula. 
	\end{proof}
	\begin{remark}
	Notice that all the identities in the preceding lemma hold already in the unlocalized Grothendieck ring $K_0(\operatorname{Var}_{\mathbb C})$. For item (4), this follows from the Schubert-cell decomposition of the Grassmannian; see, for instance, \cite[Section 9.1]{AndersonFultonEquivariantCohomology}. Nevertheless, we retain the proof using frame bundles and localization because it follows directly from the preceding discussion and makes the exposition more self-contained.
	\end{remark}
	
	\subsection{Some useful lemmas involving vector bundles and Zariski-locally trivial fibrations}
	\begin{lemma}\label{usefulLemmaInvolvingVBAndZariskiLTFibrations}
		Let $X$ be a scheme, and let $\mathcal E$ and $\mathcal F$ be vector bundles on $X$. 
		\begin{enumerate}
			\item Let
			$$
			\mathcal{E}_\bullet:=0=\mathcal E_0\subseteq \mathcal E_1\subseteq \cdots \subseteq \mathcal E_s=\mathcal E,
			\qquad
			\mathcal{F}_\bullet:=0=\mathcal F_0\subseteq \mathcal F_1\subseteq \cdots \subseteq \mathcal F_s=\mathcal F
			$$
			be filtrations by vector subbundles. Let
			$\mathcal H om_{\mathcal F\ell}(\mathcal E_\bullet,\mathcal F_\bullet)$ be the subsheaf of $\mathcal H om(\mathcal E,\mathcal F)$ whose sections over an open subset $U\subseteq X$ are the morphisms $\varphi\in \operatorname{Hom}_{\mathscr{O}_U}(\mathcal E|_U,\mathcal F|_U)$
			such that, for every $k$, $\varphi(\mathcal E_k|_U)\subseteq \mathcal F_k|_U$. Then
			$\mathcal H om_{\mathcal F\ell}(\mathcal E_\bullet,\mathcal F_\bullet)$ defines a vector subbundle of
			$\mathcal H om(\mathcal E,\mathcal F)$.
			\item For $0\leq \gamma \leq \min\{\operatorname{rk}(\mathcal{E}),\operatorname{rk}(\mathcal{F})\}$, there is a Zariski-locally trivial fibration, $\mathcal{H}om^\gamma(\mathcal{E},\mathcal{F})\subseteq\mathcal{H}om(\mathcal{E},\mathcal{F})$, whose fiber over $x\in X$ is given by all the rank $\gamma$ linear maps $\mathcal{E}_{\{x\}}\to \mathcal{F}_{\{x\}}$.  
			\item Following the same notation as in the first item,  fix a tuple of integers
			$$
			\Gamma_{\bullet}=(\gamma_{1},\ldots,\gamma_{s}), \  \Gamma_{k}:=\gamma_{1}+\cdots+\gamma_{k}\text{ and } 0\leq \gamma_{k}\leq \min\{\operatorname{rk}(\mathcal{E}_k)-\operatorname{rk}(\mathcal{E}_{k-1}),\operatorname{rk}(\mathcal{F}_k)-\Gamma_{k-1}\}.$$ 
			Then there is a locally closed subscheme $\mathcal{H}om_{\mathcal{F}\ell}^{\Gamma_\bullet}(\mathcal{E}_\bullet,\mathcal{F}_\bullet)\subseteq\mathcal{H}om_{\mathcal{F}\ell}(\mathcal{E}_\bullet,\mathcal{F}_\bullet)$
			whose fiber over a point \(x\in X\) is the variety
			\begin{equation}\label{varietyOfRankProfileFilteredMorphisms}
				X_{\Gamma_\bullet}=\left\{
				\varphi:\mathcal E_{\{x\}}\to \mathcal F_{\{x\}}
				\ \middle|\
				\varphi((\mathcal E_{k})_{\{x\}})\subseteq (\mathcal{F}_{k})_{\{x\}}
				\text{ and }
				\operatorname{rk}(\varphi|_{(\mathcal E_{k})_{\{x\}}})=\Gamma_k
				\text{ for every }k
				\right\}.
			\end{equation}
			Moreover, the morphism
			$\mathcal{H}om_{\mathcal{F}\ell}^{\Gamma_\bullet}(\mathcal{E}_\bullet,\mathcal{F}_\bullet)\to X$ is a Zariski-locally trivial fibration and 
			$$
			\mathcal{H}om_{\mathcal{F}\ell}(\mathcal{E}_\bullet,\mathcal{F}_\bullet)=\bigsqcup_{ \Gamma_\bullet}\mathcal{H}om_{\mathcal{F}\ell}^{\Gamma_\bullet}(\mathcal{E}_\bullet,\mathcal{F}_\bullet)
			$$
			is a stratification by locally closed subsets. 
		\end{enumerate}
		
	\end{lemma}
	
	\begin{proof}
		We start with the first item of the statement. Since the filtrations are by vector subbundles, the quotients
		$$
		\operatorname{gr}_i\mathcal E:=\mathcal E_i/\mathcal E_{i-1},
		\qquad
		\operatorname{gr}_i\mathcal F:=\mathcal F_i/\mathcal F_{i-1}
		$$
		are vector bundles. Let $x\in X$ and $U\subseteq X$ be an open subset containing $x$ over which all the vector bundles
		$\operatorname{gr}_i\mathcal E$ and $\operatorname{gr}_i\mathcal F$ are free, and over which the short exact sequences
		$$
		0\longrightarrow \mathcal E_{i-1}|_U
		\longrightarrow \mathcal E_i|_U
		\longrightarrow \operatorname{gr}_i\mathcal E|_U
		\longrightarrow 0
		$$
		and
		$$
		0\longrightarrow \mathcal F_{i-1}|_U
		\longrightarrow \mathcal F_i|_U
		\longrightarrow \operatorname{gr}_i\mathcal F|_U
		\longrightarrow 0
		$$
		split (see \cite[Proposition A3.1]{EisenbudCommutativeAlgebra}). Thus, we have decompositions
		$$
		\mathcal E|_U\simeq \bigoplus_{i=1}^s \operatorname{gr}_i\mathcal E|_U,
		\qquad
		\mathcal F|_U\simeq \bigoplus_{j=1}^s \operatorname{gr}_j\mathcal F|_U,
		$$
		which are compatible with the filtrations, in the sense that
		$$
		\mathcal E_k|_U\simeq \bigoplus_{i\leq k}\operatorname{gr}_i\mathcal E|_U,
		\qquad
		\mathcal F_k|_U\simeq \bigoplus_{j\leq k}\operatorname{gr}_j\mathcal F|_U.
		$$
		With respect to these decompositions, we have
		$$
		\mathcal H om(\mathcal E|_U,\mathcal F|_U)
		\simeq
		\bigoplus_{i,j}
		\mathcal H om
		\left(
		\operatorname{gr}_i\mathcal E|_U,
		\operatorname{gr}_j\mathcal F|_U
		\right).
		$$
		A morphism $\varphi:\mathcal E|_U\to \mathcal F|_U$ preserves the filtrations if and only if the component
		$\operatorname{gr}_i\mathcal E|_U\to
		\operatorname{gr}_j\mathcal F|_U$
		of $\varphi$ vanishes whenever $j>i$. Hence, under the above decomposition,
		$$
		\mathcal H om_{\mathcal F\ell}(\mathcal E_\bullet,\mathcal F_\bullet)|_U
		\simeq
		\bigoplus_{1\leq j\leq i\leq s}
		\mathcal H om
		\left(
		\operatorname{gr}_i\mathcal E|_U,
		\operatorname{gr}_j\mathcal F|_U
		\right)
		$$
		where the right hand side is locally free.
		In particular, $\mathcal H om_{\mathcal F\ell}(\mathcal E_\bullet,\mathcal F_\bullet)|_U$ is a direct summand of
		$\mathcal H om(\mathcal E|_{U},\mathcal F|_U)$, with complement
		$$
		\bigoplus_{1\leq i<j\leq s}
		\mathcal H om
		\left(
		\operatorname{gr}_i\mathcal E|_U,
		\operatorname{gr}_j\mathcal F|_U
		\right)
		$$
		which is also locally free. 
		Hence
		$\mathcal H om_{\mathcal F\ell}(\mathcal E_\bullet,\mathcal F_\bullet)$ is a vector subbundle of
		$\mathcal H om(\mathcal E,\mathcal F)$.
		\\
		\\
		Now we deal with the second statement. Let $\pi:\mathcal{H}om(\mathcal{E},\mathcal{F})\to X$. There is a tautological morphism 
		\begin{equation}\label{tautologicalHomBetweenTwoVB}
			\Phi:\pi^*\mathcal{E}\to \pi^*\mathcal{F},\quad\Phi|_{\varphi_x}:(\pi^*\mathcal{E})_{\{\varphi_x\}}=(\mathcal{E})_{\{x\}}\to (\mathcal{F})_{\{x\}}=(\pi^*\mathcal{F})_{\{\varphi_x\}}\equiv \varphi_x
		\end{equation}
		for all $x\in X$. The $k$-th degeneracy locus is 
		\begin{equation}\label{DegeneracyLocus}
			D_k(\Phi)=\{\varphi_x\in \mathcal{H}om(\mathcal{E},\mathcal{F})\ | \ \operatorname{rk}(\varphi_x)\leq k\}.
		\end{equation}
		This is a closed subscheme of $\mathcal{H}om(\mathcal{E},\mathcal{F})$ since it can be regarded as the zero subscheme of the section $\wedge^{k+1}\Phi$ \cite[Chapter 14]{FultonIntersectionTheory}. We then define
		$$
		\mathcal{H}om^\gamma(\mathcal{E},\mathcal{F}):= D_{\gamma}(\Phi)\setminus D_{\gamma-1}(\Phi)
		$$
		which is a locally closed subscheme of $\mathcal{H}om(\mathcal{E},\mathcal{F})$ whose fiber over $x$ corresponds to all rank $\gamma$ linear maps $\mathcal{E}|_{\{x\}}\to \mathcal{F}|_{\{x\}}$. The only thing that remains to be checked is that $\mathcal{H}om^\gamma(\mathcal{E},\mathcal{F})$ is Zariski-locally trivial. Let \(U\subset X\) be an open subset trivializing both \(\mathcal E\) and
		\(\mathcal F\), say $\mathcal E|_U\simeq \mathcal O_U^m,\ \mathcal F|_U\simeq \mathcal O_U^n$.
		Then
		$ \mathcal H om(\mathcal E,\mathcal F)|_U \simeq U\times \operatorname{Hom}(\mathbb C^m,\mathbb C^n)$. Under this identification the tautological morphism $\Phi:\pi^*\mathcal E\to \pi^*\mathcal F$
		is represented by the $n\times m$ matrix $A=(a_{ij})_{i,j}$ on
		\(\operatorname{Hom}(\mathbb C^m,\mathbb C^n)\). Hence the degeneracy locus
		\(D_\gamma(\Phi)\) is, over \(U\), the product of \(U\) with the usual
		determinantal variety of linear maps of rank at most $\gamma$, that is, 
		$
		D_\gamma(\Phi)|_{\pi^{-1}(U)}
		\simeq
		U\times \operatorname{Hom}^{\le \gamma}(\mathbb C^m,\mathbb C^n).
		$
		Consequently,
		$$
		\mathcal H om^\gamma(\mathcal E,\mathcal F)|_{\pi^{-1}(U)}
		=
		D_\gamma(\Phi)|_{\pi^{-1}(U)}
		\setminus
		D_{\gamma-1}(\Phi)|_{\pi^{-1}(U)}
		\simeq
		U\times \operatorname{Hom}^{\gamma}(\mathbb C^m,\mathbb C^n).
		$$
		where $\text{Hom}^\gamma(\mathbb{C}^m,\mathbb{C}^n)$ is the variety of rank $\gamma$ linear maps. Thus \(\mathcal H om^\gamma(\mathcal E,\mathcal F)\to X\) is a
		Zariski-locally trivial fibration with fiber
		\(\operatorname{Hom}^{\gamma}(\mathbb C^m,\mathbb C^n)\).
		\\
		\\
		Finally, we prove the third statement. We denote by $\iota:\mathcal{H}om_{\mathcal{F}\ell}(\mathcal{E}_\bullet,\mathcal{F}_\bullet)\hookrightarrow\mathcal{H}om(\mathcal{E},\mathcal{F})$ the inclusion morphism and denote abusively
		\[
		\Phi:(\pi\circ \iota)^*\mathcal E\longrightarrow (\pi\circ \iota)^*\mathcal F
		\]
		the pullback of the tautological morphism in Equation (\ref{tautologicalHomBetweenTwoVB}).
		Since the points of $\mathcal{H}om_{\mathcal{F}\ell}(\mathcal{E}_\bullet,\mathcal{F}_\bullet)$ correspond to
		filtration-preserving morphisms, the tautological morphism satisfies
		$\Phi((\pi\circ \iota)^*\mathcal E_k)\subseteq (\pi\circ\iota)^*\mathcal F_k$
		for every $k$. Therefore, it induces a morphism $\Phi_k:(\pi\circ \iota)^*\mathcal E_k\longrightarrow (\pi\circ \iota)^*\mathcal F_k	$. Define
		$$
		\mathcal{H}om_{\mathcal{F}\ell}^{\Gamma_\bullet}(\mathcal{E}_\bullet,\mathcal{F}_\bullet)
		:=
		\bigcap_{k=1}^s
		\left(
		D_{\Gamma_k}(\Phi_k)\setminus D_{\Gamma_k-1}(\Phi_k)
		\right),
		$$
		where $D_{\gamma}(\Phi)$ is as in Equation (\ref{DegeneracyLocus}). Since each
		$D_{\Gamma_k}(\Phi_k)$ is closed, $\mathcal{H}om_{\mathcal{F}\ell}^{ \Gamma_\bullet}(\mathcal{E}_\bullet,\mathcal{F}_\bullet)$ is a locally closed subscheme of $\mathcal{H}om_{\mathcal{F}\ell}(\mathcal{E}_\bullet,\mathcal{F}_\bullet)$. By construction, its fiber over
		$x\in X$ consists exactly of the filtration-preserving maps $\varphi:\mathcal E_{\{x\}}\to \mathcal F_{\{x\}}$
		such that $\operatorname{rk}(\varphi|_{(\mathcal{E}_{k})_{\{x\}}})=\Gamma_k$ for every $k$. It remains to show the local triviality of the map $\mathcal{H}om^{\Gamma_\bullet}_{\mathcal{F}\ell}(\mathcal{E}_\bullet,\mathcal{F}_\bullet)\to X$. Fix \(x\in X\). By the construction in the proof of part~(1), there
		exists an open neighbourhood \(U\subseteq X\) of \(x\) and compatible
		trivializations of filtered vector bundles
		\[
		\mathcal E_\bullet|_U
		\simeq
		U\times(\mathcal E_\bullet)_{\{x\}},
		\qquad
		\mathcal F_\bullet|_U
		\simeq
		U\times(\mathcal F_\bullet)_{\{x\}}.
		\]
		These trivializations induce an isomorphism
		\[
		\mathcal H om_{\mathcal F\ell}
		(\mathcal E_\bullet,\mathcal F_\bullet)|_U
		\simeq
		U\times
		\operatorname{Hom}_{\mathcal F\ell}
		\left(
		(\mathcal E_\bullet)_{\{x\}},
		(\mathcal F_\bullet)_{\{x\}}
		\right).
		\]
		Under this identification, the equations defining
		\(D_{\Gamma_k}(\Phi_k)\) are the minors imposing the corresponding rank
		condition on the second factor. Therefore,
		\[
		\mathcal H om_{\mathcal F\ell}^{\Gamma_\bullet}
		(\mathcal E_\bullet,\mathcal F_\bullet)|_U
		\simeq
		U\times X_{\Gamma_\bullet},
		\]
		where \(X_{\Gamma_\bullet}\) is the variety defined in
		Equation~\eqref{varietyOfRankProfileFilteredMorphisms}.
	\end{proof}
	\begin{lemma}\label{fiberProductOfZariskiLocallyTrivialFibrationsIsZariskiLocallyTrivial}
		Let $S$ be a scheme and $p_X:X\to S$ and $p_Y:Y\to S$ be Zariski-locally trivial fibrations with typical fibers $F_X$ and $F_Y$ respectively. Then, the fibered product $X\times_SY\to S$ is a Zariski-locally trivial fibration with typical fiber $F_X\times F_Y$. 
	\end{lemma}
	\begin{proof}
		Let $\{U_i\}_{i\in I}$ be an open covering of $S$ that trivializes both $X$ and $Y$. Then $\{p_X^{-1}(U_i)\times_{U_i}p_Y^{-1}(U_i)\}_{i\in I}$ is an open cover of the fibered product $X\times_S Y$ \cite[Corollary 4.19]{GortzWedhornAG}.  But, for all $i\in I$, 
		$$
		p_X^{-1}(U_i)\times_{U_i}p_Y^{-1}(U_i)\simeq (U_i\times F_X)\times_{U_i}(U_i\times F_Y)\simeq U_i\times (F_X\times F_Y).
		$$
		The last isomorphism follows from standard properties of fibered products (see, for instance, \cite[Chapter 3, Proposition 1.4]{LiuAG}).
	\end{proof}
	\subsection{Motivic Hall algebras}\label{motivicHallAlgebrasSection}
	For this section we closely follow Bridgeland \cite{BridgelandIntroMotivicHallAlgebras}. Fix a quiver $Q$, a set of relations $\mathcal{R}$ and a dimension vector $\mathbf{d}\in \mathbb{N}^{|Q_0|}$. The \emph{stack of $\mathbf{d}$-dimensional representations with relations} is the quotient stack 
	$$
	\mathscr{M}_{\mathbf{d}}=[\text{Rep}(Q,\mathbf{d},\mathcal{R})/\operatorname{GL}(\mathbf{d})]
	$$
	and, for a choice of stability parameter $\theta\in \mathbb{Z}^{|Q_0|}$, the quotient stacks 
	$$
	\mathscr{M}^{\theta\text{-}(s)s}_{\mathbf{d}}=[\text{Rep}^{\theta\text{-}(s)s}(Q,\mathbf{d},\mathcal{R})/\operatorname{GL}(\mathbf{d})]
	$$
	are open substacks of $\mathscr{M}_\mathbf{d}$. 
	The so-called \emph{moduli stack of finite-dimensional representations of $Q$ with relations} is the stack 
	$$
	\mathscr{M}=\bigsqcup_{\mathbf{d}\in\mathbb{N}^{|Q_0|}}\mathscr{M}_\mathbf{d}.
	$$
	Denote $\operatorname{St}/\mathscr{M}$ the 2-category of finite type stacks over $\mathscr{M}$. The \emph{relative Grothendieck group}, $K_0(\operatorname{St}/\mathscr{M})$, is the free abelian group spanned by isomorphism classes of objects in $\operatorname{St}/\mathscr{M}$ modulo some relations which generalize those described for the Grothendieck ring of varieties described in Section \ref{GrothendieckRingOfVarieties}. We refer the reader to Bridgeland \cite[Section 3.4]{BridgelandIntroMotivicHallAlgebras} for a detailed exposition of these. 
	\\
	\\
	Let $\mathscr{M}^{(2)}$ the stack of short exact sequences, that is, the stack whose $S$-points are short exact sequences
	\begin{equation}\label{objectsStackExactSequences}
		\begin{tikzcd}[column sep =normal, row sep =normal]
			0\arrow[r]& \mathcal{V}_1\arrow[r]&\mathcal{V}\arrow[r]&\mathcal{V}_2\arrow[r]&0
		\end{tikzcd}
	\end{equation}
	of families of representations of $Q$ with relations over $S$. There are morphisms $s,m,q:\mathscr{M}^{(2)}\to \mathscr{M}$ remembering $\mathcal{V}_1$, $\mathcal{V}$ and $\mathcal{V}_2$ respectively. 
	\\
	\\
	The \emph{motivic Hall algebra}, $\mathcal{H}(Q,\mathcal{R})$, of the category of finite-dimensional representations of $Q$ with relations is given by $K_0(\operatorname{St/\mathscr{M}})$ together with the \emph{convolution product}:
	$$
	[\mathfrak{X}_1\longrightarrow \mathscr{M}]\textasteriskcentered[\mathfrak{X}_2\longrightarrow \mathscr{M}]=[\mathfrak{Z}\overset{m \circ h}{\longrightarrow}\mathscr{M}]
	$$
	where $\mathfrak{Z}$ and $h$ are defined by the following cartesian diagram 
	\begin{center}
		\begin{tikzcd}[column sep =normal, row sep =normal]
			\mathfrak{Z}\arrow[r,"h"]\arrow[d]&\mathscr{M}^{(2)}\arrow[r,"m"]\arrow[d,"(s{,}q)"]&\mathscr{M} \\ 
			\mathfrak{X}_1\times\mathfrak{X}_2\arrow[r]&\mathscr{M}\times\mathscr{M}.&
		\end{tikzcd}
	\end{center}
	For instance, for $\delta_{\mathbf{d}_k}=[\mathscr{M}_{\mathbf{d_k}}\to \mathscr{M}]$, the product $\delta_{\mathbf{d}_1}\textasteriskcentered\delta_{\mathbf{d}_2}$ is the stack of exact sequences, or equivalently two-step filtrations, as in Equation (\ref{objectsStackExactSequences}) where $\mathcal{V}_k$ is a family of $\mathbf{d}_k$-dimensional representations and $m\circ h$ is mapping such an exact sequence to its middle term $\mathcal{V}$. With the product just defined, $\mathcal{H}(Q,\mathcal{R})$ is a unital associative algebra over $K_0(\operatorname{St}/\mathbb{C})$ \cite[Theorem 4.3]{BridgelandIntroMotivicHallAlgebras}. 
	\\
	\\
	Given a decomposition $\mathbf{d}=\mathbf{d}_1+\ldots+\mathbf{d}_s$, Bridgeland \cite[Lemma 4.4]{BridgelandIntroMotivicHallAlgebras} shows that the product $\delta_{\mathbf{d}_1}\textasteriskcentered\ldots\textasteriskcentered\delta_{\mathbf{d}_s}$ is induced by the diagram 
	\begin{center}
		\begin{tikzcd}[column sep =normal, row sep =normal]
			\mathscr{M}^{(s)}\arrow[r]\arrow[d]&\mathscr{M} \\ 
			\mathscr{M}^s.&
		\end{tikzcd}
	\end{center}
	Here $\mathscr{M}^{(s)}$ is the stack whose $S$-points are flags
	$$
	0\subseteq \mathcal{V}_1\subseteq\ldots\subseteq\mathcal{V}_{s-1}\subseteq\mathcal{V}
	$$
	where $\mathcal{V}_k$ is a family of finite-dimensional representations of $Q$ with relations parameterised by $S$. The horizontal morphism sends such a flag to its total object
	$\mathcal V$, whereas the vertical morphism sends it to the tuple of
	successive quotients
	$$
	\left(\mathcal V_1/\mathcal V_0,\,\mathcal V_2/\mathcal V_1,\,\ldots,\,\mathcal V_s/\mathcal V_{s-1}\right).
	$$
	Restricting the vertical morphism to $\mathscr M_{\mathbf d_1}\times\cdots\times\mathscr M_{\mathbf d_s}$
	therefore gives the stack of flags for which
	$\dim\left(\mathcal V_j/\mathcal V_{j-1}\right)=\mathbf d_j$, $j=1,\ldots,s$.
	Equivalently,
	$\dim(\mathcal V_k)=\mathbf d_1+\cdots+\mathbf d_k$ for $k=1,\ldots,s$,
	and consequently,
	\[
	\delta_{\mathbf d_1}\ast\cdots\ast\delta_{\mathbf d_s}
	=
	\left[
	\left[
	\operatorname{Frep}_{\mathbf d_\bullet}(Q,\mathcal R)
	/
	\operatorname{GL}(\mathbf d)
	\right]
	\longrightarrow
	\mathscr M
	\right],
	\]
	where the morphism to $\mathscr M$ forgets the flag and retains the
	underlying representation.
	\\
	\\
	Since every $\mathbf{d}$-dimensional representation admits an unique Harder-Narasimhan filtration \cite[Lemma 4.7]{ReinekeSurvey}, we have that 
	$$
	\delta_{\mathbf{d}} =\sum_{\substack{\mathbf{d}=\mathbf{e}_1+\ldots+\mathbf{e}_r \\ \mu_{\theta}(\mathbf{e}_1)>\cdots> \mu_{\theta}(\mathbf{e}_r)}}\delta_{\mathbf{e}_1}^{\theta\text{-}ss}\textasteriskcentered\cdots\textasteriskcentered \delta_{\mathbf{e}_r}^{\theta\text{-}ss} 
	$$
	Therefore, 
	$$
	\delta^{\theta\text{-}s}_{\mathbf{d}}=\delta_{\mathbf{d}}-\sum_{\substack{r\geq 2,\mathbf{d}=\mathbf{e}_1+\ldots+\mathbf{e}_r \\ \mu_{\theta}(\mathbf{e}_1)>\cdots> \mu_{\theta}(\mathbf{e}_r)}}\delta_{\mathbf{e}_1}^{\theta\text{-}ss}\textasteriskcentered\cdots\textasteriskcentered \delta_{\mathbf{e}_r}^{\theta\text{-}ss}.
	$$
	For each $\mathbf{e}_k\leq \mathbf{d}$, we can apply the same formula so a finite recursion gives the following identity in the motivic Hall algebra:
	\begin{equation}\label{identityInMotivicHallAlgebra}
		\delta_{\mathbf{d}}^{\theta\text{-}s}=\sum_{*}(-1)^{s-1}\delta_{\mathbf{d_1}}\textasteriskcentered\cdots\textasteriskcentered\delta_{\mathbf{d}_s}
	\end{equation}
	where the sum runs over all decompositions $\mathbf{d}_1+\cdots+\mathbf{d}_s=\mathbf{d}$ of $\mathbf{d}$ into non-zero dimension vectors such that $\mu_{\theta}(\sum_{j=1}^k\mathbf{d}_j)>\mu_{\theta}(\mathbf{d})$ for all $k=1,\ldots,s-1$.
	Applying the forgetful map $K_0(\operatorname{St}/\mathscr{M})\to K_0(\operatorname{St}/\mathbb{C})$ to the identity in Equation (\ref{identityInMotivicHallAlgebra}) gives that 
	$$
	[\mathscr{M}^{\theta\text{-}s}_{\mathbf{d}}]=\sum_{*}(-1)^{s-1}\left[[\operatorname{Frep}_{\mathbf{d}_\bullet}(Q,\mathcal{R})/\operatorname{GL}(\mathbf{d})]\right].
	$$
	Recall from Remark \ref{diagonalSubgroupsActingTrivially} that the diagonal subgroup $\mathbb{C}^*\hookrightarrow\operatorname{GL}(\mathbf{d})$ acts trivially on $\operatorname{Rep}(Q,\mathbf{d})$ and hence on $\operatorname{Rep}(Q,\mathbf{\mathbf{d}},\mathcal{R})$, then the morphism of stacks
	\begin{equation}\label{morphismQuotientStackToCoarseModuliSpace}
			\mathscr{M}_{\mathbf{d}}^{\theta\text{-}s}\to [\mathcal{M}^{\theta}(Q,\mathbf{d},\mathcal{R})]
	\end{equation}
	is a $\mathbb{G}_m$-gerbe \cite[p. 1288]{HoskinsSchaffhauserRationalPoints}.  Isomorphism classes of $\mathbb{G}_m$-gerbes are parameterised by the cohomology group $H^2_{\text{ét}}([\mathcal{M}^{\theta}(Q,\mathbf{d},\mathcal{R})],\mathbb{G}_m)$ \cite[Proposition 4.1]{HoskinsSchaffhauserRationalPoints} and the class of the gerbe in Equation (\ref{morphismQuotientStackToCoarseModuliSpace}) measures the obstruction to the existence of a universal family on $\mathcal{M}^{\theta}(Q,\mathbf{d},\mathcal{R})$. Since we assumed $\theta$ generic and $\mathbf{d}$ indivisible, such an universal family exists (see Proposition \ref{existenceUniversalFamilyNakajimaQuiverVarieties} and Theorem \ref{quiverChainsClosedImmersion}) which means that the class of this gerbe is trivial, or equivalently, that the gerbe is \emph{neutral}. Thus
	$$
	\mathscr{M}^{\theta\text{-}s}_\mathbf{d}\simeq [\mathcal{M}^{\theta}(Q,\mathbf{d},\mathcal{R})]\times B\mathbb{G}_m
	$$
	and in $K_0(\operatorname{St}/\mathbb{C})$
	$$
	[\mathscr{M}^{\theta\text{-}s}_{\mathbf{d}}]=\frac{[\mathcal{M}^\theta(Q,\mathbf{d},\mathcal{R})]}{\mathbb{L}-1}.
	$$
	Now we appeal to the fact that $K_0(\operatorname{St}/\mathbb{C})\simeq K_0^{\text{loc}}(\text{Var}_\mathbb{C})$ \cite[Lemma 3.9]{BridgelandIntroMotivicHallAlgebras} to obtain the desired formula in Lemma \ref{MotivicFeiIdentity}.
	\nocite{*}
	\bibliographystyle{alpha}
	\bibliography{biblio}

\end{document}